\documentclass[a4paper,11pt,oneside,reqno]{amsart}
\usepackage{geometry}
\usepackage{changepage}

\usepackage[english]{babel}
\usepackage{csquotes}

\usepackage{lmodern}
\usepackage{amsmath,amssymb}
\usepackage{mathtools}
\usepackage{bm}
\usepackage{stmaryrd,MnSymbol}

\usepackage{xcolor}
\usepackage{enumitem}
\usepackage{xpatch}

\usepackage[textsize=tiny]{todonotes}

\usepackage{tikz-cd}

\usepackage{setspace}
\setdisplayskipstretch{}

\definecolor{linkblue}{RGB}{1,1,190}
\definecolor{citered}{RGB}{190,1,1}

\usepackage[linkcolor=linkblue
           ,urlcolor=linkblue
           ,citecolor=citered
           ,ocgcolorlinks        
           ,bookmarksopen=True,  
           ]{hyperref}
\usepackage[ocgcolorlinks]{ocgx2}
\makeatletter
  \AtBeginDocument{
    \hypersetup{
      pdftitle  = {\@title},
    }
  }
\makeatother
  
\usepackage[capitalize]{cleveref}

\newtheorem{theorem}{Theorem}[section]
\newtheorem{definition}[theorem]{Definition}
\newtheorem{lemma}[theorem]{Lemma}
\newtheorem{proposition}[theorem]{Proposition}
\newtheorem{corollary}[theorem]{Corollary}

\theoremstyle{remark}
\newtheorem{remark}[theorem]{Remark}

\newtheorem{example}[theorem]{Example}

\setlist[enumerate,1]{label=\textup{(\arabic*)}, ref=\textup{(}\arabic*\textup{)},
  itemsep=0.5em plus 0.15em minus 0.05em,
  topsep=0.5em plus 0.15em minus 0.05em,
  leftmargin=0.75cm}
\setlist[enumerate,2]{label=\textup{(\roman*)}, ref=\textup{(}\roman*\textup{)}
  itemsep=0.5em plus 0.15em minus 0.05em,
  topsep=0.5em plus 0.15em minus 0.05em}
\setlist[itemize, 1]{itemsep=0.5em plus 0.15em minus 0.05em,
  topsep=0.5em plus 0.15em minus 0.05em, leftmargin=0.75cm}

\newlist{equivenumerate}{enumerate}{1}
\setlist[equivenumerate,1]{%
  label=\textup{(\alph*)},
  ref=\textup{(}\alph*\textup{)},
  itemsep=0.5em plus 0.15em minus 0.05em,
  topsep=0.5em plus 0.15em minus 0.05em,
  leftmargin=0.75cm
}

\newlist{proofenumerate}{enumerate}{1}
\setlist[proofenumerate,1]{%
  itemsep=0.5em plus 0.15em minus 0.05em,
  topsep=0.5em plus 0.15em minus 0.05em,
  wide, labelindent=0pt
}

\xpatchcmd{\paragraph}{\normalfont}{{\normalfont\bfseries}}{}{}

\usepackage[backend=biber,
            style=alphabetic,
            giveninits=true,
            url=false,
            isbn=false,
            maxnames=99]
           {biblatex}
\renewbibmacro{in:}{}
\DeclareFieldFormat[article,inbook,incollection,inproceedings,patent,thesis,unpublished]{title}{{#1\isdot}}
\DeclareFieldFormat[article,inbook,incollection,inproceedings,patent,thesis,unpublished]{note}{\textit{#1\isdot}}
\DeclareFieldFormat[article,inproceedings,incollection]{pages}{#1}
\DeclareFieldFormat{postnote}{#1} 

\renewbibmacro*{issue+date}{%
  \iffieldundef{year}{}{ 
  \printtext[parens]{%
    \printfield{issue}%
    \setunit*{\addspace}%
    \usebibmacro{date}}%
  \newunit}}

\AtBeginBibliography{\small}

\usepackage{microtype}

\newcommand{\defit}[1]{\textsf{#1}}

\newcommand{\C}{\mathbb C}

\newcommand{\Q}{\mathbb Q}
\newcommand{\R}{\mathbb R}
\newcommand{\Z}{\mathbb Z}

\newcommand{\bP}{\mathbb P}

\newcommand{\cA}{\mathcal A}

\newcommand{\algc}[1]{\overline{#1}}
\newcommand{\behav}[1]{\llbracket{#1}\rrbracket} 
\renewcommand{\vec}[1]{\bm{#1}}
\newcommand{\sumset}[2]{ {#1} \cdot {#2}} 

\DeclareMathOperator{\id}{id}

\DeclareMathOperator{\Hom}{Hom}
\DeclareMathOperator{\im}{im}
\DeclareMathOperator{\End}{End}
\DeclareMathOperator{\Tr}{Tr}
\DeclareMathOperator{\GL}{GL}
\DeclareMathOperator{\Ind}{Ind}

\DeclareMathOperator{\lcm}{lcm}

\DeclareMathOperator{\lspan}{span}
\DeclareMathOperator{\characteristic}{char}

\DeclarePairedDelimiter{\abs}{\lvert}{\rvert}
\DeclarePairedDelimiter{\card}{\lvert}{\rvert}
\DeclarePairedDelimiter{\length}{\lvert}{\rvert}

\title[Factoring through monomial representations]{Factoring through monomial representations: arithmetic characterizations and ambiguity of weighted automata}

\author{Antoni Puch}
\address{Faculty of Mathematics, Informatics and Mechanics (MIMUW)\\
  University of Warsaw\\
  Banacha 2\\
  02-097 Warsaw, Poland}
\email{a.puch@student.uw.edu.pl}

\author{Daniel Smertnig}
\address{Faculty of Mathematics and Physics (FMF)\\
  University of Ljubljana
  and Institute of Mathematics, Physics and Mechanics (IMFM)\\
  Jadranska ulica 21\\
  1000 Ljubljana, Slovenia}
\email{daniel.smertnig@fmf.uni-lj.si}


\begin{document}

\begin{abstract}
  \begin{singlespace}
  We characterize group representations that factor through monomial representations, respectively, block-triangular representations with monomial diagonal blocks, by arithmetic properties.
  Similar results are obtained for semigroup representations by invertible transformations.
  The characterizations use results on unit equations from Diophantine number theory (by Evertse, van der Poorten, and Schlickewei in characteristic zero, and by Derksen and Masser in positive characteristic).

  Specialized to finitely generated groups in characteristic zero, one of our main theorems recovers an improvement of a very recent similar characterization by Corvaja, Demeio, Rapinchuk, Ren, and Zannier that was motivated by the study of the bounded generation (BG) property. In positive characteristic, we get a characterization of linear BG groups, recovering a theorem of Abért, Lubotzky, and Pyber from 2003.

  Our motivation comes from weighted finite automata (WFA) over a field.
  For invertible WFA we show that $M$-ambiguity, finite ambiguity, and polynomial ambiguity are characterized by arithmetic properties.
  We discover a full correspondence between arithmetic properties and a complexity hierarchy for WFA based on ambiguity.
  In the invertible case, this is a far-reaching generalization of a recent result by Bell and the second author, characterizing unambiguous WFA, that resolved a 1979 conjecture of Reutenauer.
  As a consequence, using the computability of the (linear) Zariski closure of a finitely generated matrix semigroup, the $M$-ambiguity, finite ambiguity, and polynomial ambiguity properties are algorithmically decidable for invertible WFA.
  \end{singlespace}
\end{abstract}

\date{}

\maketitle

\vspace*{-1cm}
\setcounter{tocdepth}{1}
{\small\tableofcontents}

\section{Introduction}

Let $f\colon \Z_{\ge 0} \to \Q$ be a linear recurrence sequence (LRS). 
It is well-known that this is equivalent to the formal power series $F=\sum_{n=0}^\infty f(n) x^n$ representing a \emph{rational} function (without a pole at the origin).
Within the class of LRSs, several \emph{structural} properties of the function $F$ are reflected in \emph{arithmetic} properties of the sequence $f$ of coefficients, as demonstrated in the following instances.
\begin{itemize}
\item If each pole of $F$ is at a root of unity, then the partial fraction decomposition immediately shows that the asymptotic growth of $f$, as measured by the logarithmic Weil height (meaning $h(a/b)=\log\max\{\abs{a},\abs{b}\}$ if $\gcd(a,b)=1$), is in $O(\log(n))$.
Hence, the sequence $f$ satisfies a \emph{growth restriction}.
\item If $F$ has only simple poles, then the coefficients of $f$ are linear combinations of exponentials, and hence are contained in a set of the form 
\[
\sumset{M}{\Gamma_0} \coloneqq \Big\{\, \sum_{i=1}^m g_i : m \le M,\, g_i \in \Gamma \,\Big\}
\] 
with $\Gamma \le \Q^\times$ a finitely generated subgroup and $M \in \Z_{\ge 0}$ (a \emph{multiplicative restriction}).
\item If $F$ has only simple poles \emph{and} they are all at roots of unity, then $f$ takes only finitely many values (a \emph{finiteness restriction}).
\end{itemize}

It is much harder to show that one can sometimes recover such structural properties from the induced arithmetic ones, that is, there are corresponding \emph{inverse results}.
For the examples given here, this is possible even when one relaxes the condition of rationality on $F$ and only assumes that $F$ is \emph{$D$-finite}, with rationality being obtained as part of the conclusion:
for finiteness this was shown by van der Poorten and Shparlinski \cite{vanderpoorten-shparlinski96} and by Bell and Chen \cite{bell-chen17}; for the stated growth restriction it is due to Bell, Nguyen, and Zannier \cite[Theorem 1.3]{bell-nguyen-zannier20};
for the stated multiplicative restriction, it is a result of Bézivin \cite[Théorème 4]{bezivin86}, who generalized results of Pólya \cite{polya21}, Benzaghou \cite{benzaghou70}, and himself \cite{bezivin87} which dealt with the case $M=1$.

Recently, numerous new such inverse results have been established, often by using deep theorems from number theory: for \emph{\textup(multivariate\textup) $D$-finite} series there are growth restriction results by Bell, Chen, Nguyen, and Zannier \cite{bell-nguyen-zannier20,bell-nguyen-zannier23,bell-chen-nguyen-zannier24}.
Here, a full growth-based classification is conjectured, but currently unproven \cite[Question 4.3]{bell-nguyen-zannier23}; it is related to long-standing open problems on Siegel $E$-functions. 
Results for multiplicative restrictions have been proven by Bell and the second author \cite{bell-smertnig23}.
For power series \emph{Mahler functions}, a complete growth-based classification of five distinct structural classes has been established by Adamczewski, Bell, and the second author \cite{adamczewski-bell-smertnig23}, characterizing in particular \emph{$k$-automatic} and \emph{$k$-regular} sequences by their growth.
Among the $k$-regular sequences, a more refined growth-based classification was given by Bell, Coons, and Hare \cite[Theorem 2]{bell-coons-hare16}, who in fact establish a growth-based characterization of finitely generated subsemigroups of $\Z^{d \times d}$ \cite[Theorem 1]{bell-coons-hare16}.

A \emph{noncommutative rational series} is a function $f\colon X^* \to K$ from a free monoid $X^*$ on a finite alphabet $X$ to a field $K$ that can be computed by a \emph{weighted finite automaton \textup(WFA\textup)} \cite{berstel-reutenauer11,sakarovitch09}.
The one-letter case corresponds exactly to LRSs.
For noncommutative series, an inverse result shows that the Pólya restriction (all nonzero values contained in $\Gamma$, with $\Gamma \le K^\times$ finitely generated) corresponds precisely to those series computable by \emph{unambiguous automata}.
This was conjectured by Reutenauer in 1979 \cite{reutenauer79} and recently proven by Bell and the second author \cite{bell-smertnig21}.

The motivation, and goal, of the present work is a generalization of the main results of \cite{bell-smertnig21} to establish a full correspondence between an ambiguity hierarchy (unambiguous, finitely ambiguous, polynomially ambiguous) of weighted automata and a hierarchy based on arithmetic restrictions on the outputs of these automata (or the semigroups arising from their minimal linear representations).
We achieve a surprisingly complete and natural correspondence for the special case of \emph{invertible} WFA, that is, WFA for which the transition matrices are invertible (\cref{t:wfa-ambiguity} and \cref{fig:correspondence}). 
It leads to decidability of all of these properties (\cref{cor:intro-wfa-decidability}).

In fact, the problem leads us to consider arithmetic restrictions on linear groups: we obtain arithmetic characterizations of finitely generated linear groups that are epimorphic images of \emph{monomial representations}, respectively, \emph{block-triangular representations with monomial diagonal blocks}. 
In characteristic zero, the results also hold for non-finitely generated linear groups.
As we will explain below, in a sense, these results can be seen as an extension of the classical Burnside--Schur theorem (torsion linear groups are locally finite): we characterize multiplicative-type restrictions, and the classic Burnside--Schur theorem characterizes a corresponding finiteness restriction.
Corresponding to that, we encounter analogous limitations in positive characteristic.

\subsection*{Linear groups}
Let $\Gamma$ be a \emph{finitely generated} multiplicative subgroup of a field $K$.
In this case we always set $\Gamma_0 \coloneqq \Gamma \cup \{0\}$.
For $M \in \Z_{\ge 0}$ the $M$-fold sumset of $\Gamma_0$ is denoted by $\sumset{M}{\Gamma_0} \coloneqq \{\, \sum_{i=1}^M g_i : g_i \in \Gamma_0 \,\}$; we write the explicit dot to distinguish it from the dilation of $\Gamma_0$ by $M$.
A set $S \subseteq K$ is \defit{Bézivin} if it is a subset of a set of the form $\sumset{M}{\Gamma_0}$ with $\Gamma$ finitely generated, and it is \defit{Pólya} if it is a subset of some $\Gamma_0$ (that is, if one can take $M=1$).

A matrix group $G \le \GL_d(K)$ is \defit{Bézivin} if there is a Bézivin set containing every entry of every $A \in G$; the \defit{Pólya} property for $G$ is defined analogously.
The Pólya property is clearly not preserved under a change of basis, but the Bézivin property is (\cref{l:bezivin-base-change}).

Let us stress that even the Bézivin property depends on the particular (faithful) representation of $G$ as a linear group, that is, a group $G$ may have embeddings into some $\GL_d(K)$ that are Bézivin, while others are not.
For instance, the representation $(\Z,+) \to  \GL_1(\Q)$, $n \mapsto 2^n$ is trivially Bézivin, but $(\Z,+) \to \GL_2(\Q)$, $n \mapsto \begin{psmallmatrix} 1 & n \\ 0 & 1 \end{psmallmatrix}$ is not even locally Bézivin.
Thus, the (local) Bézivin property is not a property of the abstract group, but really of its chosen linear representation.

A \defit{monomial matrix} (or \defit{generalized permutation matrix}, or \defit{weighted permutation matrix}) is a matrix $A$ for which every row and column has at most one nonzero entry.
Of course, if $A$ is invertible, it has exactly one nonzero entry in each row and column.
Naturally, any finitely generated group of monomial matrices is Bézivin, even Pólya.
Our first main result is a sharpening of this observation together with an inverse result.%
\footnote{More refined statements of our main theorems will be given in \cref{sec:main-theorems}, after introducing some more technical notions in \cref{sec:notions}.}

The field $K$ is \defit{uniformly power-splitting} for $G \le \GL_d(K)$ if there exists some $N \ge 1$ such that $\lambda^N \in K$ for every eigenvalue of every matrix $A \in G$.
As usual, a group property holds \defit{locally} if it holds for every finitely generated subgroup.

\begin{theorem} \label{t:intro-bezivin}
  Let $K$ be a field and $G \le \GL_d(K)$.
  Suppose that $\characteristic{K}=0$ or that $G$ is finitely generated.
  Then the following statements are equivalent.
  \begin{equivenumerate}
    \item \label{intro-bez:polya} There exists $T \in \GL_d(K)$ such that $TGT^{-1}$ is locally Pólya and $K$ is uniformly power-splitting for $G$.
    \item \label{intro-bezivin:bez} The group $G$ is locally Bézivin and $K$ is uniformly power-splitting for $G$.
    \item \label{intro-bezivin:mon} The group representation $G \hookrightarrow \GL_d(K)$ is an epimorphic image of a monomial representation of $G$ \textup(over $K$\textup).
    \item \label{intro-bezivin:diag} The group $G$ has a finite-index subgroup of simultaneously diagonalizable \textup(over $K$\textup) matrices.
  \end{equivenumerate}
\end{theorem}

An arbitrary group $G$ has a faithful monomial representation over $K$ if and only if it has a finite-index subgroup $H$ that is isomorphic to a subgroup of some $(K^\times)^d$ (a monomial representation is obtained as an induced representation from a representation of $H$).
In particular, any such group is virtually abelian.
If $G$ is finitely generated, the uniform power-splitting hypothesis in \ref{intro-bez:polya} and \ref{intro-bezivin:bez} can be dropped and the equivalence still holds.

The interesting direction here is that one can deduce from the arithmetic restriction \ref{intro-bezivin:bez} structural properties of the representation.
We also observe that the local Pólya property and the local Bézivin property characterize the same types of groups, however, the Pólya property only holds for a suitable choice of basis, while the Bézivin property holds in any basis and is therefore much more natural in this context.

The proof of \cref{t:intro-bezivin} makes essential use of the fundamental theorem on unit equation in characteristic zero, a deep theorem from Diophantine number theory proven by Evertse \cite{evertse84} and van der Poorten and Schlickewei \cite{vdpoorten-schlickewei91}, that can nowadays be deduced from Schmidt's subspace theorem.
In positive characteristic, solutions of unit equations are effectively described by a theorem of Derksen and Masser \cite{derksen-masser12}, as suitable orbits of compositions of twisted Frobenius automorphisms, and by a theorem of Adamczewski and Bell \cite[Theorem 3.1]{adamczewski-bell12}, as $p$-automatic sets.
We use the result by Derksen and Masser. 
The necessary background is given in \cref{sec:unit-equations}.
Using a corollary that works in any characteristic (\cref{p:unit-equation-crude}), we are able to give a proof of \cref{t:intro-bezivin} that works for any characteristic.

\begin{remark}
While finishing the write-up of the present paper, we were made aware of the recent papers \cite{corvaja-rapinchuk-ren-zannier22,corvaja-demeio-rapinchuk-ren-zannier22} and preprint \cite{corvaja-demeio-rapinchuk-ren-zannier23} by Corvaja, Demeio, Rapinchuk, Ren, and Zannier.
The special case of \cref{t:intro-bezivin} in which $G$ is finitely generated \emph{and} $\characteristic K=0$, yields a generalization of one of their main theorems \cite[Theorem 1.13]{corvaja-demeio-rapinchuk-ren-zannier23}.
Let us discuss the connection and differences in detail.

Corvaja, Demeio, Rapinchuk, Ren, and Zannier introduce the notions of a \emph{Purely Exponential Parametrization} and \emph{\textup(PEP\textup) sets}.
From their definitions, it is obvious that every (PEP) subset of $K^{d \times d}$ is contained in a set of the form $(\sumset{M}{\Gamma_0})^{d \times d}$; conversely, every set of the form $(\sumset{M}{\Gamma_0})^{d \times d}$ can be seen to be a (PEP) set (by using enough parameters in the purely exponential parametrization).
So, if $G$ is finitely generated and $\characteristic{K}=0$, one could replace \ref{intro-bezivin:bez} in \cref{t:intro-bezivin} by ``$G$ is a subset of a (PEP) set''.

In notation adapted to the present paper, \cite[Theorem 1.13]{corvaja-demeio-rapinchuk-ren-zannier23} states that, for a field $K$ of \emph{characteristic zero} and a group $G \le \GL_d(K)$, the following statements are equivalent.
\begin{equivenumerate}[label=(\Alph*)]
  \item \label{other:pep} $G$ is a (PEP) set.
  \item \label{other:bg} $G$ has bounded generation (BG) and every element of $G$ is diagonalizable (over $\algc{K}$).
  \item \label{other:torus} $G$ is finitely generated and the connected component $\widetilde G^0$ of the Zariski closure $\widetilde G$ is a torus.
\end{equivenumerate}

If $G$ is finitely generated, then \ref{other:pep}$\,\Leftrightarrow\,$\ref{other:torus} is almost \ref{intro-bezivin:bez}$\,\Leftrightarrow\,$\ref{intro-bezivin:diag} of \cref{t:intro-bezivin}.
A difference is that \ref{other:pep} requires $G$ to \emph{be} a (PEP) set, whereas \ref{intro-bezivin:bez} only requires containment \emph{in} a (PEP) set. So \cref{t:intro-bezivin}\ref{intro-bezivin:bez} a priori makes no assumption on the existence of certain elements, making \ref{intro-bezivin:bez}$\,\Rightarrow\,$\ref{intro-bezivin:diag} stronger.
Of course, then we do not quite get \ref{other:torus}$\,\Rightarrow\,$\ref{other:pep}, but this is the ``easier'' direction (although it requires a clever observation to see that finite unions of (PEP) sets are again (PEP) sets \cite[Proposition 4.3]{corvaja-demeio-rapinchuk-ren-zannier23}).

The stronger implication \ref{intro-bezivin:bez}$\,\Rightarrow\,$\ref{intro-bezivin:diag} that we deduce is crucial in the derivation of our second main theorem and the applications to automata, because it is easy to see that the involved sets are Bézivin (i.e., subsets of (PEP) sets) but not that they are indeed full (PEP) sets, as the involved properties do not allow us to deduce enough information about existence of certain elements:
this occurs in \cref{p:absirred-fg-bezivin} in the implications \ref{absirred:fg}$\,\Rightarrow\,$\ref{absirred:trbez} and \ref{absirred:trbez}$\,\Rightarrow\,$\ref{absirred:locbez}, and in \ref{wfa-necc:m-ambig} of \cref{p:wfa-necessary}.

However, the arguments of \cite{corvaja-demeio-rapinchuk-ren-zannier23} can also yield the stronger implication in characteristic zero \cite{corvaja-email24}.
Variants of our main results hold more generally for semigroups of automorphisms (\cref{sec:main-theorems}), and these more general results are also used in our application to weighted automata.

For not necessarily finitely generated groups $G$, the implication \ref{other:pep}$\,\Rightarrow\,$\ref{other:torus} yields finite generation.
Our \cref{t:intro-bezivin} does not directly imply that Bézivin groups in characteristic $0$ are finitely generated, however this can be deduced from \ref{intro-bezivin:diag} using an additional small observation (see \cref{l:subgroup-bezivin,cor:bez-0-is-fg}).
On the other hand, the article \cite{corvaja-demeio-rapinchuk-ren-zannier23} says nothing about the local properties or positive characteristic.

Given the similarity of \cref{t:intro-bezivin} and \cite[Theorem 1.13]{corvaja-demeio-rapinchuk-ren-zannier23}, let us note that both proofs are powered by unit equations, however the approach appears to be quite different: Corvaja, Demeio, Rapinchuk, Ren, and Zannier first derive a quantitative point counting theorem \cite[Theorem 1.5]{corvaja-demeio-rapinchuk-ren-zannier23}.
For this, they need stronger results on unit equations.
They also make use of the theory of \emph{generic elements} developed by Prasad and Rapinchuk \cite{prasad-rapinchuk03,prasad-rapinchuk14,prasad-rapinchuk17}. 

By contrast, we only make use of the basic finiteness result on unit equations (of course, then we also do not get a quantitative point counting theorem). 
We work directly with linear groups, having no need to pass to linear algebraic groups or to use the theory of generic elements.
The key to our proof is instead \cref{c:term-bound}: this allows us to iteratively diagonalize suitable powers of matrices, while forcing a compatible structure on the remaining group (\cref{l:permuting-spaces,l:simultaneous-decomposition}).
In this way, step by step, the locally Bézivin group is compelled to reveal its diagonalizable finite-index subgroup to us.

Interestingly, the motivation of \cite{corvaja-demeio-rapinchuk-ren-zannier23} comes from a very different direction, namely the study of the bounded generation (BG) property for linear groups, but one ends up characterizing the same class of groups.
In particular, a group consisting of diagonalizable matrices (over $\algc{K}$) has BG if and only if it is finitely generated Bézivin.
In characteristic zero we have nothing new to add to this angle.
In positive characteristic it was already known that BG groups are virtually abelian by a theorem of Abért, Lubotzky, and Pyber \cite[Theorem 1.1]{abert-lubotzky-pyber03}.
Since our results hold in any characteristic, in the spirit of \cite{corvaja-demeio-rapinchuk-ren-zannier23}, we now also easily recover this theorem from \cref{t:intro-bezivin} (see \cref{subsec:bg}).
In fact, we get a characterization of linear BG groups in positive characteristic (\cref{cor:bg-pos}).
\end{remark}

The \defit{spectrum} of a group $G \le \GL_d(K)$ is the set of its eigenvalues, as a subset of the algebraic closure $\algc{K}$.
The group $G$ has \defit{finitely generated spectrum} if its spectrum is contained in a finitely generated subgroup of $\algc{K}^\times$.
Since connected solvable linear groups are simultaneously triangularizable over the algebraic closure $\algc{K}$ by the Lie--Kolchin Theorem \cite[Theorem 10.5]{borel91}, it is not hard to see that every virtually solvable group, and thus in particular every block-triangular group with monomial diagonal blocks, has locally finitely generated spectrum (\cref{l:virt-solvable-locfg}).
Our second main theorem is an inverse result for this arithmetic restriction.

\begin{theorem} \label{t:intro-block}
  Let $K$ be a field and $G \le \GL_d(K)$.
  Suppose that $\characteristic{K}=0$ or that $G$ is finitely generated.
  Then the following statements are equivalent.

  \begin{equivenumerate}
    \item The group $G$ has locally finitely generated spectrum and $K$ is uniformly power-splitting for $G$.
    \item The representation $G \hookrightarrow \GL_d(K)$ is the epimorphic image of an \textup(upper\textup) block-triangular representation with monomial diagonal blocks \textup(over $K$\textup).
  \end{equivenumerate}
\end{theorem}

The next corollary shows that, unlike the Bézivin property, having locally finitely generated spectrum does \emph{not} depend on the particular representation of $G$.

\begin{corollary}[\cite{bernik05}] \label{c:intro-spec}
  Let $K$ be a field and $G \le \GL_d(K)$ a linear group.
  Suppose that $\characteristic{K}=0$ or that $G$ is finitely generated.
  Then $G$ has locally finitely generated spectrum if and only if it is virtually solvable.
\end{corollary}

\Cref{c:intro-spec} has previously been deduced by Bernik \cite{bernik05} in characteristic zero as a consequence of the theory of generic elements by Prasad and Rapinchuk.
Bernik even obtains virtual solvability under the---a priori---weaker condition that the spectrum is contained in a finitely generated field.
Our proof does not make use of the theory of generic elements or Bernik's result. 

For non-finitely generated groups, \cref{t:intro-bezivin,t:intro-block} do not hold in positive characteristic in general: the group $\GL_d(\algc{\mathbb F_p})$ is not virtually solvable for $d \ge 2$, yet, being a torsion group, it is locally finite by Burnside--Schur, hence it has locally finitely generated spectrum and is even locally Bézivin.

However, by Tits' alternative, a finitely generated linear group is either virtually solvable or contains a non-cyclic free subgroup.
Combining this with \cref{c:intro-spec}, the following does hold in any characteristic even for non-finitely generated groups.

\begin{corollary}
  Let $K$ be a field.
  A group $G \le \GL_d(K)$ has a non-cyclic free subgroup if and only if it does \emph{not} have locally finitely generated spectrum.
\end{corollary}

To put \cref{t:intro-block} into perspective, let us compare it to some other results characterizing linear groups by restrictions on their spectrum.
We say that $G \le \GL_d(\algc{K})$ is \defit{tame} if its spectrum consists of roots of unity.
Equivalently, the group $G$ has \emph{locally finite spectrum}.%
\footnote{A finitely generated group is defined over a finitely generated field and consists of matrices of a fixed dimension. Hence, only finitely many roots of unity are possible for each finitely generated subgroup.}
Tame groups show up in \cite{bell-coons-hare16} \cite[Section 8]{adamczewski-bell-smertnig23}. 
Torsion groups are tame, and a well-known variant of Burnside--Schur shows, using a trace argument, that absolutely irreducible tame groups are locally finite (see for instance \cite[Lemma 9.3]{lam91}; compare to \cref{p:absirred-fg-bezivin}).
In characteristic zero, absolutely irreducible tame groups are even finite, because torsion groups are finite.

Let us now assume that $\characteristic{K}=0$ and $K=\algc{K}$ is algebraically closed.
Let $G \le \GL_d(\algc{K})$ be tame.
Choosing a suitable basis, we can assume that $G$ is block-triangular with absolutely irreducible diagonal blocks.
Now these diagonal blocks are \emph{finite}.
Explicitly,
\[
G = \begin{pmatrix}
   S_1 & K^{d_1 \times d_2} &  \dots & K^{d_1 \times d_r} \\
   0 & S_2 &  \dots & K^{d_2 \times d_r}  \\
   \vdots & \ddots & \ddots & \vdots \\
   0 & \dots & \dots  & S_r
\end{pmatrix}
\]
with finite $S_i \le \GL_{d_i}(\algc{K})$.
Irreducible representations of finite groups are always epimorphic images of the regular representation, and we get that $G$ has a block-triangular representation with permutation matrices in the diagonal blocks.

In this spirit, \cref{t:intro-block} extends this characterization from groups having \emph{locally finite spectrum} to groups having \emph{locally finitely generated spectrum}.

At the other end of the spectrum (pun intended), if the spectrum of an irreducible group $G \le \GL_d(\algc{K})$ is contained in the field $K$, then $G$ is conjugate to a subgroup of $\GL_d(K)$ by a theorem of Bernik \cite{bernik07,radjavi-yahaghi19}.
Thus, in Bernik's theorem, a weaker restriction on the spectrum controls representability over a subfield.

Many other spectral properties of matrix (semi)groups have been studied.
One that is particularly close in spirit to the locally finitely generated spectrum property is the \emph{submultiplicative spectrum} property \cite{radjavi00,kramar05,grunenfelder-kosir-omladic-radjavi12}:
A semigroup $S \subseteq \C^{d \times d}$ has submultiplicative spectrum if, for all $A$,~$B \in S$ every eigenvalue of $AB$ is of the form $\alpha\beta$ with $\alpha$ an eigenvalue of $A$ and $\beta$ an eigenvalue of $B$.
Semigroups with submultiplicative spectrum have locally finitely generated spectrum, but submultiplicativity is much more restrictive.
If $S \subseteq \C^{d \times d}$ is irreducible with submultiplicative spectrum, then there exists a finite nilpotent group $G \le \C^{d \times d}$ such that $\C^\times S = \C^\times G$ \cite[Proposition 2.7]{radjavi00}.

Incidentally, the Burnside--Schur Theorem as well as the mentioned results extend to matrix \emph{semigroups} $S \subseteq K^{d \times d}$, although in the generalizations of the Burnside--Schur Theorem, the restriction to finite generation is now essential, even in characteristic zero (see \cite{steinberg12} for a proof and a comprehensive list of earlier proofs).
Extending \cref{t:intro-bezivin,t:intro-block} to semigroups is beyond the scope of the present work, but may be feasible and would allow us to remove the invertibility restriction in the WFA case (see \cref{sec:conclusion}).

In a related spirit to \cref{t:intro-bezivin}, Bernik, Mastnak, and Radjavi \cite{bernik-mastnak-radjavi11} and Cigler and Drnovšek \cite{cigler-drnovsek13} have shown that irreducible semigroups $S \subseteq \C^{d \times d}$ with \emph{nonnegative diagonal entries} can be conjugated to a monomial semigroup.

\begin{figure}
  \begin{center}
  \includegraphics[height=70mm]{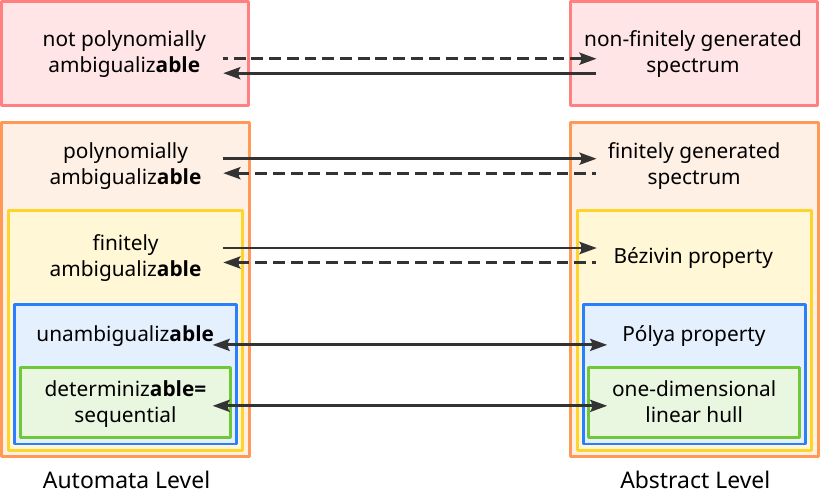}
  \end{center}
  \caption{The natural complexity hierarchy for rational series (functions computable by WFA) based on ambiguity is reflected in arithmetic properties of the outputs, respectively the semigroup of a minimal linear representation.
  Full arrows represent unconditional implications, the implications of the dashed arrows are currently known in the invertible case (\cref{t:wfa-ambiguity}, \cref{p:wfa-necessary}, and \cite{bell-smertnig21}). 
  The picture represents the case of algebraically closed fields, for simplicity.}
  \label{fig:correspondence}
\end{figure}

\subsection*{Ambiguity of WFA}

Let $X$ be a finite alphabet and $K$ a field.
The fundamental Kleene--Schützenberger Theorem shows that a function $f\colon X^* \to K$ is a \emph{noncommutative rational series} if and only if it can be computed by a WFA (this holds even when $K$ is just a semiring instead of a field) \cite[Chapter 1]{berstel-reutenauer11} \cite[Chapter III.2]{sakarovitch09}.

WFA can equivalently be described through their linear representations.
In this context, a \defit{linear representation} is a tuple $(u,\mu,v)$ of vectors $u \in K^{1 \times d}$, $v \in K^{d \times 1}$ and a monoid homomorphism $\mu \colon X^* \to K$ such that $f(w) = u \mu(w) v$ for every word $w \in X^*$.
We say that a WFA (or its linear representation) is \defit{invertible} if $\mu(X^*) \subseteq \GL_d(K)$.
Represented in a suitable way as a directed graph, WFA can be viewed as a computational model; as such they are a central object of study in theoretical computer science, we refer to \cite{sakarovitch09,droste-kuich-vogler09,berstel-reutenauer11} as starting points into the vast literature.

Among WFA, there is a natural complexity hierarchy, determined by how many successful runs a given word can have, leading to deterministic, unambiguous, finitely ambiguous, polynomially ambiguous, and exponentially ambiguous WFA (see \cref{d:wfa-ambiguity}).
The ambiguity of a given WFA does not depend on the weights; it can be elegantly characterized in terms of certain features of the underlying directed graph.
These features can be algorithmically decided by classical work of Weber and Seidl \cite{weber-seidl91}.

However, a given rational series $f \colon X^* \to K$ can be computed by many different WFA.
Over a field, there is, up to a base change, a unique minimal one (minimal number of states, corresponding to minimal dimension of the linear representation).
Unfortunately, the minimal WFA will typically \emph{not} achieve a low ambiguity class, as the low ambiguity comes at the expense of increasing the number of states.
One therefore faces the question: given a rational $f\colon X^* \to K$, represented by \emph{some} WFA, is it decidable what the lowest ambiguity class of a WFA that recognizes $f$ is? 
In particular, one can ask if $f$ is computable by a \emph{deterministic} WFA.
This question has a long history for WFA over tropical semirings (see the survey \cite{lombardy-sakarovitch06} as a starting point).
The general problem is still open in this setting \cite[Problem 2]{lombardy-sakarovitch06}, but \emph{within} the class of polynomially ambiguous WFA the problem has since been resolved by Kirsten and Lombardy \cite{kirsten-lombardy09}.

For fields, there was a long-standing conjecture of Reutenauer that was recently solved in \cite{bell-smertnig21}, and that shows that computability by unambiguous WFA corresponds exactly to the rational series having the Pólya property.
Together with the computability of the (linear) Zariski closure, established for matrix groups by Derksen and Jeandel \cite{derksen-jeandel05} and for matrix semigroups by Hrushovski, Ouaknine, Pouly, and Worrell \cite{hrushovski-ouaknine-pouly-worrell18,hrushovski-ouaknine-pouly-worrell23} (see also \cite{nosan-pouly-schmitz-shirmohammadi-worrell22,bell-smertnig23b}), this shows that both determinizability and unambigualizability over fields is decidable, resolving \cite[Problem 1]{lombardy-sakarovitch06}.
First complexity bounds were given by Benalioua, Lhote, and Reynier \cite{benalioua-lhote-reynier24}, and in the subclass of polynomially weighted automata by Jecker, Mazowiecki, and Purser \cite{jecker-mazowiecki-purser24}.

We mention \cite{kirsten12,colcombet15,colcombet-petrisan17,daviaud-jecker-pierrealain-villevalois17,mohri-riley17,colcombet-petrisan20,kostolanyi22,hrushovski-ouaknine-pouly-worrell23,kostolyani24} as further pointers into the recent literature on determinizability and unambigualizability of weighted automata.
An interesting connection between rational series and one-dimensional topological theories was recently established by Im and Khovanov \cite{im-khovanov22,im-khovanov24}.

\emph{Assuming the WFA is invertible}, our group-theoretic results yield a perfect correspondence between the best possible ambiguity class of a rational $f \colon X^* \to K$, and corresponding arithmetic properties on the values, respectively, the semigroup of a minimal linear representation, as depicted in \cref{fig:correspondence}.
Two WFA are \defit{equivalent} if they compute the same series.

\begin{theorem} \label{t:wfa-ambiguity}
  Let $\cA=(u,\mu,v)$ be a minimal linear representation of an \emph{invertible} WFA over a field $K$, and let $M \in \Z_{\ge 0}$.
  \begin{enumerate}
  \item \label{wfa-ambig:finamb} The WFA is equivalent to an $M$-ambiguous WFA if and only if there exists a finitely generated $\Gamma \le K^\times$ such that $u \mu(X^*) v \subseteq \sumset{M}{\Gamma_0}$.
  In particular, it is equivalent to a finitely ambiguous WFA if and only if the outputs form a Bézivin set.
  \item The WFA is equivalent to a polynomially ambiguous WFA if and only if $\mu(X^*)$ has finitely generated spectrum and $K$ is power-splitting for $\mu(X^*)$.
  \end{enumerate}
\end{theorem}

In fact, in these cases the equivalent automaton $\widehat \cA$ can be chosen in such a way that the transition matrices are monomial, respectively, block-triangular with monomial diagonal blocks (\cref{p:nice-wfa}).
In the special case $M=1$, \cref{t:wfa-ambiguity}\ref{wfa-ambig:finamb} recovers the main result of \cite{bell-smertnig21}, when restricted to invertible WFA (no restriction is necessary in \cite{bell-smertnig21}).

Using \cref{t:wfa-ambiguity}, the complexity classes become decidable. (The restriction on the field can be loosened, see \cref{sec:decidability}.)
\begin{corollary} \label{cor:intro-wfa-decidability}
  Let $\cA=(u,\mu,v)$ be a linear representation of an \emph{invertible} WFA over a number field $K$.
  \begin{enumerate}
  \item It is decidable whether $\cA$ is equivalent to a finitely ambiguous, respectively, a polynomially ambiguous WFA.
  \item If $\cA$ is equivalent to a finitely ambiguous WFA, then the minimal ambiguity $M$ of such a WFA is computable.
  \end{enumerate}
\end{corollary}

It is also possible to compute an equivalent finitely ambiguous, respectively, polynomially ambiguous WFA if it exists.
In the finitely ambiguous case, this WFA can be computed to be of the minimal possibly ambiguity $M$.
While we show that the problems are decidable, our algorithms are presumably not particularly efficient or practical, as they are based on the computation of the (linear) Zariski closure of a finitely generated group of matrices.
In several special cases better algorithms are known.

\subsection*{Proofs} 
In \cite{bell-smertnig21} unit equations were used together with the linear Zariski topology to great effect in the unambiguous (Pólya) case of WFA.
Because the Pólya property depends on the initial and terminal weights of a WFA, it does not reduce to a purely (semi)group-theoretic question.

Unfortunately, the key lemma \cite[Lemma 4.2]{bell-smertnig21} does not appear to generalize in a suitable form to attack Bézivin set problems we face here.
Instead, we need to inspect the matrix groups much more closely, applying unit equations in a novel but very natural way, in the key results \cref{p:full-term-bound,c:term-bound}.

The heavy lifting for the Bézivin case is done in \cref{sec:locbez,sec:bezchar}. The heart of the arguments is contained in \cref{p:full-term-bound,l:permuting-spaces,l:simultaneous-decomposition}.
The finitely generated spectrum case is dealt with in \cref{sec:locfg}.
As we hope the results may be useful to people with a variety of mathematical backgrounds, we strive to provide detailed and mostly self-contained proofs, where this is reasonably possible. 
In contrast to \cite{bell-smertnig21}, we avoid using the height machinery also in positive characteristic; considerations using the natural density in subsets of the integers end up being sufficient.
\subsection*{Acknowledgements} We thank Ganna (Anja) Kudryavtseva for participating in extensive discussions that helped flesh out the main arguments in this paper.
We also thank Laure Daviaud, Urban Jezernik, Filip Mazowiecki, and David Purser for discussions and Janez Bernik for making us aware of the papers \cite{corvaja-rapinchuk-ren-zannier22,corvaja-demeio-rapinchuk-ren-zannier22,corvaja-demeio-rapinchuk-ren-zannier23}.
Additional thanks go to Pietro Corvaja for confirming that the arguments in \cite{corvaja-demeio-rapinchuk-ren-zannier23} can also be used to deduce our version of \cref{t:intro-bezivin} and to Andrei Rapinchuk for pointing out that \cref{c:intro-spec} has been observed before.

Smertnig was supported by the Slovenian Research and Innovation Agency (ARIS) program P1-0288.
Part of the research was conducted during a research stay of Puch at IMFM in Ljubljana, funded by ARIS program P1-0288 and Polish National Science Centre SONATA BIS-12 grant number 2022/46/E/ST6/00230.

\section{Notions and Preliminaries} \label{sec:notions}

For $m$,~$n \in \Z_{\ge 0}$, we write $[m,n] \coloneqq \{\, x \in \Z : m \le x \le n \}$ for the discrete interval and set $[n]\coloneqq [1,n]$.
If $K$ is a field, we denote by $\mu_n(K)$ the group of all $n$-th roots of unity, by $\mu_n^*(K)$ the set of primitive $n$-th roots of unity, and set $\mu(K)\coloneqq\bigcup_{n \ge 1} \mu_n(K)$.

It will be useful to make use of the language of linear representations of groups (and sometimes semigroups).
Let $K$ be field, embedded in a fixed algebraic closure $\algc{K}$, and let $V$ be a finite-dimensional vector space over $K$.
A ($K$-)representation of a semigroup $S$ is a semigroup homomorphism $\rho\colon S \to \End(V)$, or, more explicitly, it is the pair $(V,\rho)$. 
We only consider finite-dimensional representations.
A representation is \defit{faithful} if $\rho$ is injective.
A \defit{morphism} (or \defit{$S$-equivariant map}) between representations $(V,\rho)$ and $(W,\sigma)$ is a linear map $\Psi\colon V \to W$ such that $\Psi\rho(s) =\sigma(s)\Psi$ for all $s \in S$.

\emph{Almost exclusively} we will only consider representations $S \to \GL(V)$ with invertible images (which is of course a severe restriction), with \cref{p:full-term-bound,c:term-bound} being the notable exceptions.
For the problems we consider, this will turn out to essentially be equivalent to considering group representations $G \to \GL(V)$. 
However, the more general statements are more natural in the applications to automata theory.

A matrix semigroup $S \subseteq \GL_d(K)$ is \defit{Bézivin} if there is a Bézivin set containing every entry of every $A \in S$; the \defit{Pólya} property for $S$ is defined analogously.

\begin{lemma} \label{l:bezivin-base-change}
  If $S \subseteq \GL_d(K)$ is a Bézivin semigroup and $T \in \GL_d(K)$, then also $TST^{-1}$ is Bézivin.
\end{lemma}

\begin{proof}
  Let $\Gamma \le K^\times$ be finitely generated and let $M \ge 0$ be such that all entries of all matrices $A \in S$ are contained in $\sumset{M}{\Gamma_0}$.
  Suppose $T=(t_{ij})_{1\le i,j\le d}$ and $T^{-1}=(t_{ij}')_{1\le i,j\le d}$.
  If $A=(a_{ij})_{1\le i,j\le d} \in S$, then the $rs$-entry of $TST^{-1}$ is $\sum_{i,j=1}^d t_{ri} a_{ij} t_{js}'$.
  Taking $\Gamma'$ to be the group generated by $\Gamma$ together with all $t_{ij}$ and $t_{ij}'$, we see that all elements of $TST^{-1}$ have their entries in $\sumset{d^2M}{\Gamma_0'}$.
\end{proof}

Thus, the Bézivin property is independent of the choice of basis (although the constant $M$ and the group $\Gamma$ may change).
Given a representation $\rho \colon S \to \End(V)$ of a semigroup $S$, we therefore say that $\rho$ has \defit{finitely generated spectrum}, respectively the \defit{Bézivin} property, if this is the case for $\rho(S)$ for some choice of basis on $V$ (and hence for all choices of bases on $V$).

A semigroup $S \subseteq \GL(V)$ is \defit{irreducible} if it acts irreducibly on $V$, that is, there is no nonzero proper $S$-invariant subspace of $V$. 
The semigroup is \defit{absolutely irreducible} if it is irreducible when considered over $\algc{K}$, that is, as a subsemigroup of $\GL(\algc{K} \otimes_K V)$.

The notion of a steady endomorphism will be of central importance in our proofs.

\begin{definition} \label{d:steady}
  An endomorphism $A \in \End(V)$ is \defit{steady} if
  \begin{enumerate}
  \item every subspace $U \subseteq V$ that is $A^n$-invariant for some $n \ge 1$ is $A$-invariant, and 
  \item all eigenvalues of $A$ are contained in $K$.
  \end{enumerate} 
\end{definition}

The following basic properties of steadyness are easily deduced using the Jordan normal form (see \cref{sec:locbez} below for proofs).

\begin{lemma} \label{l:steady-basic} Let $A \in \End(V)$.
  \begin{enumerate}
    \item \label{steady-basic:power} There exists some $n \ge 1$ such that $A^n$ is steady if and only if there exists some $m \ge 1$ such that $\lambda^m \in K$ for all eigenvalues $\lambda \in \algc{K}$ of $A$.
    \item \label{steady-basic:norootofunity} If $A$ is steady and $\lambda$,~$\mu \in K$ are eigenvalues with $\lambda/\mu$ a root of unity, then $\lambda=\mu$.
    \item \label{steady-basic:pos-char-diagonalizable} If $\characteristic{K} > 0$, then every steady $A \in \End(V)$ is diagonalizable.
  \end{enumerate}
\end{lemma}

Since property \ref{steady-basic:power} of \cref{l:steady-basic} will be important, we make the following definitions.

\begin{definition}
  Let $S \subseteq \GL(V)$ be a semigroup.
  \begin{enumerate}
  \item   A field $L \supseteq K$ is \defit{power-splitting} for $S$ if for every $A \in S$ and every eigenvalue $\lambda \in \algc{K}$ of $A$, there exists some $m \ge 1$ such that $\lambda^m \in K$.
  \item   A field $L \supseteq K$ is \defit{uniformly power-splitting} for $S$ if there is an $N \ge 1$ such that $\lambda^N \in K$ for every eigenvalue $\lambda$ of every $A \in S$.
  \end{enumerate}
\end{definition}

Given a representation $\rho \colon S \to \GL(V)$ every decomposition $V = V_1 \oplus \cdots \oplus V_r$ gives rise to a block-decomposition of $\rho$.
We denote by $\rho_{ij}$ the corresponding map $\rho_{ij}\colon S \to \Hom_K(V_j, V_i)$.
Each diagonal block $\rho_{ii}$ is of course itself a representation of $S$.

\begin{definition} \label{d:monomial-rep}
  A semigroup representation $\rho \colon S \to \GL(V)$ is 
  \begin{enumerate}
    \item \defit{monomial} if there exists a decomposition $V = V_1 \oplus \cdots \oplus V_d$ into one-dimensional spaces, such that for every $s \in S$ and every $j \in [d]$, there exists $i \in [d]$ such that $\rho(s)V_j \subseteq V_i$.
    \item \defit{\textup(upper\textup) block-triangular with monomial diagonal blocks} if there exists a decomposition $V = V_1 \oplus \cdots \oplus V_r$ such that $\rho_{ij} = 0$ for $i < j$ and $\rho_{ii}$ is monomial for all $i \in [r]$.
  \end{enumerate}
\end{definition}

In other words, a representation is \emph{monomial} if and only if, after a suitable choice of basis, every matrix of the representation has exactly one nonzero entry in each row and each column.
A representation is \emph{block-triangular with monomial diagonal blocks} if, after a suitable choice of basis, every matrix has that shape.

In general, locally Bézivin representations (even of finitely generated semigroups) will not be monomial.
For instance, every finite matrix group is trivially Bézivin, but not necessarily monomial.
For this reason, the following weaker notions are useful.

\begin{definition} \label{d:weakly-epimonomial}
  Let $\rho\colon S \to \GL(V)$ be a semigroup representation.
  A \defit{weakly epimonomial decomposition} \textup(of $V$ for $\rho$\textup) is a vector space decomposition $V = V_1 \oplus \cdots \oplus V_r$ such that
  \begin{enumerate}
    \item \label{d-wm:intersection} every $V_i$ is contained in an eigenspace of all steady $\rho(s)$ \textup($s \in S$\textup),
    \item \label{d-wm:permutation} every $\rho(t)$, with $t \in S$, permutes the spaces $V_1$, \dots,~$V_r$, and
    \item \label{d-wm:splitting} the field $K$ is power-splitting for $\rho(S)$.
  \end{enumerate}
  The representation $\rho$ is \defit{weakly epimonomial} if there exists a weakly epimonomial decomposition.
\end{definition}

In different terminology, the spaces $V_1$, \dots,~$V_r$ are an imprimitive system for $\rho$, consisting of subspaces of joint-eigenspaces of the steady elements.
Given a weakly monomial decomposition, we call
\[
  D \coloneqq \{\, \rho(s) \in \rho(S) : \text{every $V_i$ is contained in an eigenspace of $\rho(s)$} \,\}
\]
the \defit{diagonal} of the $\rho$ with respect to the decomposition.
The diagonal $D$ is a commutative subsemigroup of $\rho(S)$ (subgroup if $S=G$ is a group), and it embeds into $(K^\times)^r$.

After a choice of basis that is compatible with the decomposition $V=V_1 \oplus \cdots \oplus V_r$, the set $D$ is indeed precisely the subset of diagonal matrices in $\rho(S)$ that are scalar on each $V_i$.
Restricting a weakly epimonomial representation to a subsemigroup $S'$, trivially gives again a weakly epimonomial representation of $S'$.

We will later see: if $\rho\colon S \to \GL(V)$ is weakly epimonomial and $G = \langle \rho(S) \rangle$ is the group generated by $\rho(S)$, then $G \hookrightarrow \GL(V)$ is weakly epimonomial as well (\cref{l:lift-wepi-to-group}).

\begin{definition} \label{d:epimonomial}
  \begin{enumerate}
  \item Let $\rho\colon S \to \GL(V)$ be a semigroup representation with diagonal $D$, and let $G = \langle \rho(S) \rangle$.
  A weakly epimonomial decomposition for $\rho$ is \defit{epimonomial} if $G/\langle D\rangle$ is finite.

  \item A semigroup representation $\rho\colon S \to \GL(V)$ is \defit{epimonomial} if it has an epimonomial decomposition.
  \end{enumerate}
\end{definition}

If $S \subseteq \GL(V)$ is a subsemigroup, we apply all of this terminology to $S$ as well, via the representation $S \hookrightarrow \GL(V)$ arising from the inclusion.
Note that these notions depend on the representation, respectively, the particular embedding of $S$ into $\GL(V)$, and as such are not properties of the abstract semigroup $S$ (or abstract group $G$).

If a group representation $\rho\colon G \to \GL(V)$ is epimonomial with diagonal $D$, it is easy to see that first restricting to $\rho^{-1}(D)$, and then considering the induced representation on $G$, gives a monomial representation for $G$. Specifically, we have the following (see \cref{l:factor-through-monomial} for a proof).

\begin{proposition} \label{p:epimonomial-lift} 
  If $\rho\colon S \to \GL(V)$ is an epimonomial semigroup representation, then there exists a monomial representation $\widehat\rho\colon S \to \GL(\widehat V)$ and an epimorphism $(\widehat V, \widehat \rho) \twoheadrightarrow (V,\rho)$.
\end{proposition}

\section{Statements of Main Results} \label{sec:main-theorems}

We now state our main results in full generality.
Let $K$ be a field and $V$ a finite-dimensional vector space.
Our first main theorem characterizes locally Bézivin representations (over any field).

\begin{theorem} \label{t:main-loc-bezivin}
  Let $\rho \colon S \to \GL(V)$ be a semigroup representation.
  Then the following statements are equivalent.
  \begin{equivenumerate}
    \item \label{locbez:polya} There exists a basis of $V$ with respect to which $\rho$ is locally Pólya.
    \item \label{locbez:bezivin} The representation $\rho$ is locally Bézivin.
    \item \label{locbez:weakbezivin} The representation $\rho_L \colon S \to \GL(L \otimes_K V)$ is locally Bézivin for some extension field $L/K$, and $K$ is a power-splitting field for $\rho_L(S)$.
    \item \label{locbez:locepimon} The representation $\rho$ is weakly epimonomial.
    \item \label{locbez:finsub} The representation $\rho$ is locally epimonomial.
  \end{equivenumerate}
\end{theorem}

For finitely generated $S$, we will get the following corollary.

\begin{corollary} \label{cor:locbezivin-fg}
  If $S$ is finitely generated and $G =\langle \rho(S) \rangle$ is the group generated by $\rho(S)$, then the following statements are equivalent.
  \begin{equivenumerate}
    \item \label{lbzfg:polya} There exists a basis of $V$ with respect to which $\rho$ is Pólya.
    \item \label{lbzfg:bez} The representation $\rho$ is Bézivin.
    \item \label{lbzfg:weakbezivin} The representation $\rho_L\colon S \to \GL(L \otimes_K V)$ is Bézivin for some extension field $L/K$, and $K$ is a power-splitting field for $\rho_L(S)$.
    \item \label{lbzfg:epimon} The representation $\rho$ is epimonomial.
    \item \label{lbzfg:diag} The group $G$ is virtually simultaneously diagonalizable \textup(over K\textup). 
    \item \label{lbzfg:monomial} The representation $\rho$ is the epimorphic image of a monomial representation of $S$.
  \end{equivenumerate} 
\end{corollary}

If $S$ is not finitely generated, we still get a similar result in characteristic $0$.

\begin{corollary} \label{cor:locbezivin-char0}
  If $\characteristic K=0$ and $G =\langle \rho(S) \rangle$, the following statements are equivalent.
  \begin{equivenumerate}
    \item \label{lbz0:locbez} The representation $\rho$ satisfies the equivalent conditions of \cref{t:main-loc-bezivin} and $K$ is uniformly power-splitting for $\rho(S)$.
    \item \label{lbz0:diag} The group $G$ is virtually simultaneously diagonalizable \textup(over K\textup). 
    \item \label{lbz0:monomial} The representation $\rho$ is the epimorphic image of a monomial representation of $S$.
  \end{equivenumerate}
\end{corollary}

The proof of \cref{cor:locbezivin-char0} makes use of the Jordan--Schur Theorem and the variant of the Burnside--Schur Theorem that shows that a linear torsion group of finite exponent over a field of characteristic zero is finite.

For representations with locally finitely spectrum, we obtain the following.

\begin{theorem} \label{t:trilocmonomial}
  Let $\rho\colon S \to \GL(V)$ be a semigroup representation.
  Then the following statements are equivalent.
  \begin{equivenumerate}
  \item \label{triloc:spec} The representation $\rho$ has locally finitely generated spectrum and $K$ is power-splitting for $\rho(S)$.
  \item \label{triloc:repr} The representation $\rho$ is block-triangular with locally epimonomial diagonal blocks.
  \end{equivenumerate}
\end{theorem}

Combining \cref{t:trilocmonomial} with \cref{cor:locbezivin-fg,cor:locbezivin-char0} we obtain the following.

\begin{corollary} \label{cor:trilocmonomial-lift}
  If 
  \begin{itemize}
  \item the semigroup $S$ is finitely generated, \emph{or}
  \item the field $K$ is uniformly power-splitting for $\rho(S)$ and $\characteristic K=0$,
  \end{itemize}
  then the statements of \cref{t:trilocmonomial} are in addition equivalent to
  \begin{equivenumerate}
    \item[\textup{(c)}] There exists a representation $\widehat \rho \colon S \to \GL(\widehat V)$ that is block-triangular with monomial diagonal blocks and an epimorphism $(\widehat V,\widehat \rho) \twoheadrightarrow (V,\rho)$.
  \end{equivenumerate}
\end{corollary}

Here, it is \emph{not} possible to remove the power-splitting hypothesis when $S$ is finitely generated (\cref{ex:fibanocci-not-bezivin}).

\Cref{t:intro-bezivin,t:intro-block,c:intro-spec} are now easily deduced.

\begin{proof}[Proof of \cref{t:intro-bezivin}]
  We consider the faithful representation $\rho\colon G \hookrightarrow \GL_d(K)$.
  If $G$ is finitely generated, then the equivalences, without the uniform power-splitting hypothesis in \ref{intro-bez:polya} and \ref{intro-bezivin:bez}, hold by \cref{cor:locbezivin-fg} applied to $\rho$.
  Since $K$ is trivially uniformly power-splitting for an epimonomial $G$, also the uniform power-splitting hypothesis holds.

  Now suppose $\characteristic{K}=0$. Then the claims hold by \cref{cor:locbezivin-char0}.
\end{proof}

\begin{proof}[Proof of \cref{t:intro-block}]
  Again consider $\rho\colon G \hookrightarrow \GL_d(K)$.
  If $G$ has locally finitely generated spectrum and $K$ is uniformly power-splitting for $G$, then \cref{cor:trilocmonomial-lift} yields the claim.

  In the converse direction, to apply \cref{cor:trilocmonomial-lift} again, it suffices to observe that $K$ is uniformly power-splitting for $G$.
  Let $V \coloneqq K^{d\times 1}$ and let $\widehat\rho \colon G \to \GL(\widehat V)$ be a block-triangular representation with monomial diagonal blocks such that $\rho$ is an epimorphic image of $\widehat\rho$.
  Let $\pi \colon \widehat V \to V$ denote the $G$-equivariant epimorphism.

  There exists $N \ge 1$ such that $\widehat\rho(g)^N$ is upper-triangular for each $g \in G$.
  Thus, the field $K$ is uniformly power-splitting for $\widehat\rho$.
  Considered as a vector space epimorphism, the map $\pi$ splits, so that $\widehat V \cong \ker(\pi) \oplus V$.
  Since $\ker(\pi)$ is $G$-invariant, the elements of $\widehat\rho(G)$ have a block-triangular form with respect to this splitting, and the spectrum of $\rho(g)$ is a subset of the spectrum of $\widehat\rho(g)$.
  Hence, the field $K$ is also uniformly power-splitting for $G$.
\end{proof}

\begin{proof}[Proof of \cref{c:intro-spec}]
  Suppose $G$ has locally finitely generated spectrum.
  Since $\algc{K}$ is trivially uniformly power-splitting for $\overline \rho(G)$, there exists a faithful block-triangular representation of $G$ over $\algc{K}$ with monomial diagonal blocks (\cref{cor:trilocmonomial-lift}).
  The finite-index subgroup whose diagonal blocks are themselves diagonal is solvable.
  Hence, the group $G$ is virtually solvable.
  Conversely, every virtually solvable group has locally finitely generated spectrum (\cref{l:virt-solvable-locfg}).
\end{proof}

We will prove \cref{t:main-loc-bezivin,cor:locbezivin-fg,cor:locbezivin-char0} in \cref{sec:bezchar}, after doing the heavy-lifting in \cref{sec:locbez}.
\Cref{t:trilocmonomial,cor:trilocmonomial-lift} will be proved in \cref{sec:locfg}.

\section{Guiding Examples and Special Cases}

We gather some easy examples that illustrate the conditions and limitations of the theorems as well as some interesting special cases.
The section provides some context for the (uniform) power-splitting conditions and the limitations to characteristic zero in some of our theorems.
The section can otherwise safely be skipped.

\subsection{One-letter WFA: Linear Recurrence Sequences (LRS)}
Let $K$ be a field.
It is well known that a $K$-valued sequence $(a(n))_{n \ge 0}$ is an LRS if and only if there exist $d \ge 0$, a matrix $A \in K^{d \times d}$, and vectors $u \in K^{1 \times d}$ and $v \in K^{d \times 1}$ such that $a(n) = uA^n v$ for all $n  \ge 0$.
Equivalently, the generating function $\sum_{n=0}^\infty a(n) x^n$ is a rational function $P/Q \in K(x)$ with $Q(0) =1$.
Equivalently, the sequence $(a(n))_{n \ge 0}$ is generated by a one-letter weighted finite automaton: if $X=\{x\}$ is the alphabet, then $a(n)$ is the output of the word $x^n$.

\begin{example} \label{ex:fibanocci-not-bezivin}
  Let $A = \begin{psmallmatrix}
    0 & 1 \\
    1 & 1 \\
  \end{psmallmatrix} \in \GL_2(\Q)$, let $u = (1,0)$, and let $v=(0,1)^T$. 
  Then $F_n\coloneqq u A^n v$ is the sequence of Fibonacci numbers.
  The eigenvalues of $A$ are $\varphi_{\pm} \coloneqq \tfrac{1\pm\sqrt{5}}{2}$.
  Since no power of $\varphi_+$ is in $\Q$ (for instance, because $\varphi_+^n = \varphi_+ F_n + F_{n-1}$ and $\varphi_+ \not\in \Q$), the group generated by $A$ is not power-splitting over $\Q$.
  Hence, it does not even have a block-triangular representation with monomial diagonal blocks.
  On the other hand, over the field $L=\Q(\varphi_+)$, the matrix $A$ is diagonalizable.
  Thus, the group $\langle A \rangle$ and the sequence $F_n = \tfrac{1}{\sqrt{5}}\varphi_+^n - \tfrac{1}{\sqrt{5}} \varphi_-^n$ are Bézivin over $L$.
  This shows that the power-splitting condition in the main theorems cannot be removed.
\end{example}

Consider an arbitrary LRS of the form $a(n)=uA^nv$.
By dropping finitely many initial terms, we can assume $A \in \GL_d(K)$.
We also assume that $d$ is minimal for the given LRS.
If $K$ is algebraically closed of characteristic $0$, then $(a(n))_{n \ge 0}$ is Bézivin if and only if $A$ is diagonalizable: this follows from a theorem of Bézivin \cite[Théorème 4]{bezivin86} or as a special case of \cref{p:wfa-necessary}\ref{wfa-necc:bezivin} together with \cref{t:intro-bezivin}.

Similarly, for an arbitrary field $K$, the sequence $(a(n))_{n \ge 0}$ is Bézivin if and only if some power of $A$ is diagonalizable over $K$.
In this case (and only then), one can find a linear representation $(\widehat u, \widehat A, \widehat v)$ of the same LRS $(a(n))_{n \ge 0}$ with a monomial matrix $\widehat A$ by \cref{p:nice-wfa}.
The LRS case, corresponding to one-letter WFA, was also recently studied in detail by Kostolányi \cite{kostolanyi22,kostolyani24}.

By passing to an extension field (in which we can put $A$ into Jordan normal form) an LRS can always be computed by a polynomially ambiguous WFA, while computability by a finitely ambiguous WFA over an extension field corresponds to diagonalizability of $A$.
In particular, the exponentially ambiguous case only appears for alphabets with at least two letters.

\subsection{Non-finitely generated groups}
For non-finitely generated groups two issues appear: one needs to assume \emph{uniform} power-splitting and characteristic zero is essential.

\begin{example} \label{exm:locbez-not-monomial}
  \begin{enumerate}
  \item \label{exm:roots-of-unity-over-R} Consider the embedding $\C^\times \to \GL_2(\R)$ given by $a+bi \mapsto \begin{psmallmatrix} a & b \\ -b & a \end{psmallmatrix}$.
  Restricting to $\mu(\C)$, we get a locally finite, and hence locally Bézivin, representation of $\mu(\C)$ over $\R$.
  However, the group $\mu(\C)$ does not have a faithful monomial representation over $\R$: suppose $j\colon \mu(\C) \to \GL_d(\R)$ is a faithful monomial representation.
  Then $j(\zeta)^{d!}$ is diagonal for each $\zeta \in \mu(\C)$.
  Because $j(\zeta)^{d!}$ has finite order, its diagonal entries must be $\pm 1$, showing that $j(\zeta)^{2d!} = I$.
  This contradicts the existence of elements of arbitrarily large order in $\mu(\C)$. 
  Of course, the field $\R$ is power-splitting but not uniformly power-splitting for our original representation of $\mu(\C)$.

  \item Fix $n \ge 1$ and let $R_n \coloneqq \{\, z \in \C^\times : z^n \in \R \,\}$.
  Then $R_n$ is infinite abelian and $R_n$ can be embedded into $\GL_2(\R)$ as a uniformly power-splitting locally Bézivin subgroup; indeed with the embedding as in \ref{exm:roots-of-unity-over-R}, we have
  \[
  R_n = \bigcup_{\zeta \in \mu_n(\C)} \zeta\, \R^\times,
  \]
  (identifying $\R^\times \subseteq \GL_2(\R)$ with the scalar matrices).
  Hence, the group $R_n$ does have a monomial representation over $\R$ (\cref{cor:locbezivin-char0}).
  Indeed, here $\R^\times$ is simultaneously diagonal of finite index.

  \item Let $K$ be a field of positive characteristic $p > 0$.
  Consider the faithful representation $\rho\colon (K,+) \mapsto \GL_2(K)$ defined by
  \[
    \rho(x) = \begin{pmatrix} 1 & x \\ 0 & 1 \end{pmatrix}.
  \]
  Since $\im\rho$ is an abelian $p$-torsion group, it is locally finite.
  In particular, therefore $\rho$ is locally Bézivin.

  Suppose that $\pi\colon (K,+) \to \GL(V)$ is a faithful monomial representation on a vector space of dimension $d$.
  Then there is a group homomorphism $\psi\colon \im \pi \to \mathfrak S_d$ to the symmetric group, mapping an element of $\im\pi$ to its underlying permutation.
  Let $D \coloneqq \pi^{-1}(\ker(\psi))$.
  Then $D$ maps to diagonal matrices under $\pi$, hence $D \hookrightarrow (K^\times)^d$.
  Now $D$ is $p$-torsion, and hence embeds into the $p$-torsion subgroup of $(K^\times)^d$, that is $D \hookrightarrow \mu_p(K)^d$.
  Thus, the group $D$ is finite, and also of finite index in $K$.
  Hence, also $K$ is finite.

  Thus, over an infinite field of positive characteristic, the group $(K,+)$ has a faithful locally Bézivin representation and $K$ is uniformly power splitting for that representation, but $(K,+)$ has no faithful monomial representations, showing that \cref{cor:locbezivin-char0} does not hold in general in positive characteristic.
  \end{enumerate}
\end{example}

\subsection{$k$-regular and $k$-automatic sequences}

In the 90s, Allouche and Shallit introduced the notion of a $k$-regular sequence as a generalization of a $k$-automatic sequence \cite{allouche-shallit03} (here $k \in \Z_{\ge 2}$ is an arbitrary but fixed parameter).
One of the equivalent definitions is the following: a sequence $(a(n))_{n \ge 0}$ with values in the field $K$ is \defit{$k$-regular} if there exist $d \ge 0$, matrices $A_0$, \dots,~$A_{k-1} \in K^{d \times d}$ and vectors $u \in K^{1 \times d}$, $v \in K^{d \times 1}$ such that,
\[
a\big(\sum_{i=0}^l {n_i k^i}\big) = u A_{n_0} \cdots A_{n_l} v,
\]
whenever $l \ge 0$, and $n_0$, \dots,~$n_l \in [0,k-1]$.
In other words, there is a WFA on the alphabet of $k$-adic digits $X=[0,k-1]$, such that whenever a word $w \in X^*$ represents the number $n$ in base $k$, then $a(n)$ is the output of the WFA for the word $w$.%
\footnote{The word $w$ may have leading zeroes, but all such words representing the same number produce the same output, in this definition.
In defining the class of $k$-regular sequences, it does not actually matter whether the WFA reads the base $k$ representation from least to most significant digit or the other way around.
}
We suppose that the dimension $d$ is minimal for the sequence, that is, the WFA is minimal.

The class of $k$-regular sequences contains the $k$-automatic sequences as a subset; conversely, the generating function of any $k$-regular sequence is $k$-Mahler (see \cite{adamczewski-bell-smertnig23} for definitions).

Let $K=\algc{\Q}$.
Then the class of $k$-Mahler functions can be separated into five distinct classes based on the asymptotic growth of their coefficients, as measured by the logarithmic Weil height \cite{adamczewski-bell-smertnig23}.
The three slowest growing classes correspond to $k$-regular sequences: $O(\log n)$, $O(\log\log n)$ and $O(1)$.
Here the $O(\log\log n)$ class correspond to the matrix semigroup $S$ generated by $A_0$, \dots,~$A_{k-1}$ being \defit{tame}, that is, the semigroup has locally finite spectrum.
The class $O(1)$ corresponds to $k$-automatic sequences, which is equivalent to $S$ being finite,  and is trivially Bézivin.

\emph{Let us restrict to the consideration of $k$-regular sequences for which the matrices $A_0$, \dots,~$A_{k-1}$ are invertible.}
Then the classification by the Bézivin property and finitely generated spectrum, established in this paper, interfaces with the growth classification of \cite{adamczewski-bell-smertnig23} to yield the following picture.

\begin{center}
\begin{tikzcd}[column sep=small]
  & \text{$k$-regular}  & \\
  & \text{f.g. spectrum} \arrow[u, phantom, sloped, "\subsetneq"] \arrow[dr, phantom, sloped, "\supsetneq"] &  \\
  \text{Bézivin} \arrow[ur, phantom, sloped, "\subsetneq"] & & \text{tame} \\
  & \text{$k$-automatic} \arrow[ul, phantom, sloped, "\supsetneq"] \arrow[ur, phantom, sloped, "\subsetneq"]  &
\end{tikzcd}
\end{center}

\subsection{Bounded generation} \label{subsec:bg}
A motivation in \cite{corvaja-rapinchuk-ren-zannier22,corvaja-demeio-rapinchuk-ren-zannier22,corvaja-demeio-rapinchuk-ren-zannier23} is the study of the bounded generation (BG) property for linear groups.
A group $G$ has BG, if there exist $A_1$, \dots, $A_n \in G$ such that $G = \langle A_1 \rangle \cdots \langle A_n \rangle$.
One of the main results of Corvaja, Demeio, Rapinchuk, Ren, and Zannier \cite[Theorem 1.13]{corvaja-demeio-rapinchuk-ren-zannier23} is that, in \emph{characteristic zero}, a linear group $G$, all of whose elements are diagonalizable, has BG if and only if $G$ is finitely generated and the connected component of its Zariski closure is a torus.
In particular, such a group is virtually abelian. 

This is based on the observation that a BG group whose elements are diagonalizable is a PEP set \cite[Example 1.4]{corvaja-demeio-rapinchuk-ren-zannier23}; the same observation shows that a BG group is Bézivin so that \cref{t:intro-bezivin} recovers the result, but this does not give anything new.

In \emph{positive characteristic}, Abért, Lubotzky, and Pyber had already shown before that every linear BG group is virtually abelian \cite[Theorem 1.1]{abert-lubotzky-pyber03}. Here it is not necessary to assume that the elements are diagonalizable.
Since our \cref{t:intro-bezivin} also holds in positive characteristic, we obtain a strengthening of \cite[Theorem 1.1]{abert-lubotzky-pyber03}.
\begin{corollary} \label{cor:bg-pos}
  Suppose $K$ has positive characteristic.
  A group $G \le \GL(V)$ has bounded generation \textup(BG\textup) if and only if it is finitely generated and virtually simultaneously diagonalizable \textup(over $\algc{K}$\textup).
\end{corollary}

\begin{proof}
  First suppose that $G$ has BG.
  We identify $\GL(V) \cong \GL_d(K)$ and may assume $K=\algc{K}$.
  Say $G = \langle A_1 \rangle \cdots \langle A_n \rangle$ with $A_i \in \GL_d(\algc{K})$.
  Then $G$ is obviously finitely generated by $A_1$, \dots,~$A_n$.
  For each $A_i$, choose $T_i \in \GL_d(\algc{K})$ such that $B_i \coloneqq T_i A_i T_i^{-1}$ is in Jordan normal form.
  Since the characteristic is positive, there exists some power $q=p^e$ such that $B_i^{q}$ is diagonal for all $i \in [n]$ (see, for instance, the proof of \cref{l:steady-basic}).
  Now
  \[
  G = \bigcup_{r_1,\dots,r_n \in [0,q-1]} \bigcup_{m_1,\dots,m_n \in \Z} (T_1^{-1} B_1^{q m_1} B_1^{r_1} T_1) \cdots  (T_n^{-1} B_n^{q m_n} B_n^{r_n} T_n).
  \]
  For fixed $r_1$, \dots, $r_n$, these matrices have their entries in some Bézivin set (over $\algc{K}$), because the $B_i^q$ are diagonal and everything else is constant.
  But then the finite union $G$ is also Bézivin.
  Now \cref{t:intro-bezivin} implies the claim.

  For the converse direction, it is sufficient to note that $G$ is virtually abelian, and every virtually abelian group is easily seen to have BG.
\end{proof}

\section{Unit Equations: Linear Equations over Finitely Generated Groups} \label{sec:unit-equations}

Throughout this section, let $K$ be a field and let $\Gamma \le K^\times$ be a \emph{finitely generated} subgroup.
The theory of \emph{unit equations} concerns itself with the study of solutions $(x_0,\dots,x_n) \in \Gamma^{n+1}$ to a linear homogeneous equation
\begin{equation} \label{eq:unit-equation}
a_0 X_0 + \cdots + a_n X_n = 0 \qquad\qquad (a_i \in K).
\end{equation}
The restriction to solutions in $\Gamma$ turns this into a problem in Diophantine number theory.
Obviously, any $\Gamma$-multiple of a solution is again a solution, so it makes sense to consider projective solutions $(x_0:\cdots:x_n) \in \bP^n(\Gamma)$ which have a representative with $x_0$, \dots,~$x_n \in \Gamma$.
If, for a solution $\vec x=(x_0:\cdots:x_n)$, there is a subset $\emptyset \ne I \subsetneq [0,n]$ such that $\sum_{i\in I} a_i x_i=0$, then $\vec x$ is \defit{degenerate}. 
Otherwise, the solution $\vec x$ is \defit{non-degenerate}.

In a degenerate solution, one can multiply subsums by different scalars, potentially generating infinitely many solutions. However, for non-degenerate solutions and $\characteristic K=0$, we have the following celebrated theorem.
It was proven independently by Evertse \cite{evertse84} and van der Poorten--Schlickewei \cite{vdpoorten-schlickewei82} for number fields, and extended to fields of characteristic zero in \cite{vdpoorten-schlickewei91}. The theorem can also be deduced from Schmidt's subspace theorem.
See \cite[Chapter 6]{evertse-gyory15} or \cite[Theorem 7.4.1]{bombieri-gubler06} for details.

\begin{theorem}[Main Theorem on Unit Equations] \label{t:main-thm-unit-equations}
  Suppose $\characteristic K=0$.
  Then \cref{eq:unit-equation} has only finitely many non-degenerate solutions in $\bP^n(\Gamma)$.
\end{theorem}

Now suppose $\characteristic{K} = p > 0$.
Then the conclusion of the theorem does not hold, because the Frobenius automorphism can generate infinitely many solutions.
However, these solutions can still be partitioned into finitely many orbits by a theorem of Derksen and Masser \cite[Theorem 3]{derksen-masser12}, that we describe in the following.
Let $\sqrt{\Gamma} \le K^\times$ be the group of all $\gamma \in K^\times$ for which there exists $n \ge 1$ with $\gamma^n \in \Gamma$. 
A \defit{$\sqrt \Gamma$-automorphism} is a map
\[
  \psi \colon \mathbb P^{n}(K) \to \mathbb P^{n}(K),\ (x_0:\cdots:x_n) \mapsto (\alpha_0 x_0: \dots : \alpha_n x_n)
\]
with $\alpha_0$, \dots,~$\alpha_n \in \sqrt{\Gamma}$.
For $q$ a power of $p$, let
\[
  \varphi_q(x_0:\cdots:x_n) \coloneqq (x_0^q : \cdots : x_n^q).
\]
For $\sqrt{\Gamma}$-automorphisms $\psi_1$, \dots,~$\psi_k$ and $\vec y\in \mathbb P^{n}(K)$ let
\begin{equation} \label{eq:frob-orbit}
  [ \psi_{1}, \ldots, \psi_{k} ]_q\{\vec y\} \coloneqq  \big\{\, (\psi_{1}^{-1}\varphi_q^{e_1} \psi_{1}) (\psi_{2}^{-1} \varphi_q^{e_2} \psi_{2}) \cdots (\psi_{k}^{-1} \varphi_q^{e_k} \psi_{k})(\vec y) : e_1, \ldots, e_k \ge 0  \,\big\}.
\end{equation}
Thus, the set $[ \psi_{1}, \ldots, \psi_{k} ]_q\{\vec y\}$ is obtained from the point $\vec y$ by applying powers of Frobenius automorphisms, twisted by suitable $\sqrt{\Gamma}$-automorphisms.

\begin{theorem}[{\cite[Theorem 3]{derksen-masser12}}] \label{t:derksen-masser}
  Let $\characteristic K=p > 0$.
  If $\sqrt{\Gamma}$ is a finitely generated group and $S \subseteq \bP(\Gamma)$ is the set of non-degenerate solutions of \cref{eq:unit-equation}, then there exists some power $q$ of $p$ such that $S$ is contained in a finite union of sets of the form as in \cref{eq:frob-orbit} \textup(with $k \le n-1$\textup).
\end{theorem}

If $K$ is a finitely generated field and $\Gamma$ is finitely generated, then $\sqrt{\Gamma}$ is also finitely generated (see, for instance, \cite[Lemma 4.1]{bell-smertnig21} for a proof of this well-known fact).
When applying \cref{t:derksen-masser}, we will always be able to work in a finitely generated field.

An easy consequence of \cref{t:main-thm-unit-equations,t:derksen-masser} is the following useful lemma.

\begin{lemma} \label{l:powers-bezivin}
  Let $\lambda \in K$ and let $p=\characteristic{K} \ge 0$.
  Suppose there exist $a_0$, \dots,~$a_M \in K^\times$ such that
  \[
  a_0 \lambda^n \in a_1 \Gamma_0 + \cdots + a_M \Gamma_0
  \]
  for infinitely many $n \in \Z \setminus p\Z$.
  Then there exists $m \ge 1$ such that $\lambda^m \in \Gamma_0$.
\end{lemma}

\begin{proof}
  We assume $\lambda \ne 0$, as the claim is trivially true otherwise.
  Working in the field generated by $a_0$, \dots,~$a_M$ and $\Gamma$ over the prime field of $K$, we may also assume that $K$ is finitely generated.
  By assumption there exist an infinite set $\Omega \subseteq \Z\setminus p\Z$ and maps $\gamma_1$, \dots,~$\gamma_M \colon \Omega \to \Gamma_0$ such that $a_0 \lambda^n - a_1 \gamma_1(n) - \dots - a_M \gamma_M(n) = 0$ for all $n \in \Omega$.
  We view the tuples $(\lambda^n, \gamma_1(n), \dots, \gamma_M(n))$ as solutions to a unit equation in $M+1$ variables over the group generated by $\Gamma \cup \{\lambda\}$.
  Since each solution can be partitioned into non-degenerate solutions of subequations, and there are only finitely many ways to partition the set $[0,M]$, there exists a subset $I \subseteq [M]$ and an infinite set $\Omega' \subseteq \Omega$ such that $a_0 \lambda^n - \sum_{i\in I} a_i \gamma_i(n) = 0$ is non-degenerate for all $n \in \Omega'$.
  Because $a_0$,~$\lambda \ne 0$, necessarily $I \ne \emptyset$.
  After renumbering the $\gamma_i$, without restriction $I=[l]$ with $l \in [M]$.

  \emph{First suppose $\characteristic{K} = 0$.}
  \Cref{t:main-thm-unit-equations} implies that $\{\, \lambda^n / \gamma_1(n) : n \in \Omega' \,\}$ is finite.
  Hence, there are $n \ne n' \in \Omega'$ with $\lambda^n / \gamma_1(n) = \lambda^{n'} / \gamma_1(n')$.
  Then $\lambda^{n-n'} = \gamma_1(n)/\gamma_1(n') \in \Gamma$.

  \emph{Now suppose $\characteristic{K} = p > 0$.}
  Then \cref{t:derksen-masser} implies that $\{\, (\lambda^n : \gamma_1(n) : \cdots : \gamma_l(n) ) : n \in \Omega' \,\}$ is contained in a union of finitely many sets of the type as in \cref{eq:frob-orbit}.
  Hence, there exists an infinite subset $\Omega'' \subseteq \Omega'$ all of whose solutions are contained in a single such set.
  On this set, we can express the ratio of the first two coordinates in the form
  \[
  \frac{\lambda^n}{\gamma_1(n)} = \underbrace{\beta_1^{q^{e_1(n)}-1} \beta_2^{q^{e_1(n)}(q^{e_2(n)}-1)}\cdots \beta_k^{q^{e_1(n)+\cdots +e_{k-1}(n)} (q^{e_k(n)}-1)}}_{\in \sqrt\Gamma} x^{q^{e_1(n)+ \cdots + e_k(n)}},
  \]
  with suitable $\beta_1$, \dots,~$\beta_k \in \sqrt\Gamma$, and $q$ a power of $p$.
  Hence, we get $\lambda^n / \gamma(n) = x^{q^{e(n)}}$ for some $\gamma \colon \Omega'' \to \sqrt{\Gamma}$ and $e(n) \coloneqq e_1(n) + \cdots + e_k(n)$.
  Again take $n \ne n' \in \Omega''$, and without restriction let $e(n) \le e(n')$.
  Then $\lambda^{n q^{e(n') - e(n)}}/ \gamma(n)^{q^{e(n')-e(n)}} = x^{q^{e(n')}} = \lambda^{n'} / \gamma(n')$.
  Now $p \nmid n'$ and $n \ne n'$ imply $n q^{e(n') - e(n)} \ne n'$, and hence $\lambda \in \sqrt{\Gamma}$.
\end{proof}

We also note the following consequence for later use in \cref{cor:bez-0-is-fg}.

\begin{lemma} \label{l:subgroup-bezivin}
  Let $\Lambda \le K^\times$ be a group.
  If there exists $M \ge 0$ such that
  \[
  \Lambda \subseteq \sumset{M}{\Gamma_0},
  \]
  then $\Lambda$ is finitely generated.
\end{lemma}

\begin{proof}
  Let $K_0$ be the prime field of $K$.
  By the assumption, the group $\Lambda$ is contained in the finitely generated field $K_0(\Gamma)$, and we may without restriction assume $K=K_0(\Gamma)$.
  Because $\lambda^n \in \sumset{M}{\Gamma_0}$ for every $\lambda \in \Lambda$ and $n \in \Z$, \cref{l:powers-bezivin} implies $\Lambda \subseteq \sqrt\Gamma$.
  Since $\sqrt\Gamma$ is finitely generated abelian, so is $\Lambda$.
\end{proof}

It is well known that $\Z^d$ cannot be covered by finitely many cosets of subgroups of strictly smaller rank.
This can be viewed as a consequence of the fact that $\Q^d$ cannot be covered by finitely many proper affine subspaces, or as a special case of the fact that a group cannot be covered by finitely many cosets of subgroups of infinite index \cite[Lemma 4.1]{neumann54}. 
To be able to deal with positive characteristic, we use an easy variation of this result dealing with infinite unions, as long as the cosets are sufficiently sparse.

The \defit{upper natural density} (or just \defit{density}) of $A \subseteq \Z$ is
\[
d(A) \coloneqq \limsup_{N \to \infty} \frac{\card{A \cap [-N,N]}}{2N+1}.
\]
It holds that $d(A \cup B) \le d(A) + d(B)$ and $d(A+y) = d(A)$ for all $y \in \Z$.
Finite sets have zero density.

\begin{lemma} \label{l:no-low-density-cover}
  Let $d \ge 0$ and let $\varphi_1$, \dots,~$\varphi_n \colon\Z^d \to \Z$ be \emph{nonzero} group homomorphisms.
  If $F_1$, \dots,~$F_n \subseteq \Z$ are sets of zero density, then 
  \[
  \bigcup_{i=1}^n \varphi_i^{-1}(F_i) \subsetneq \Z^d.
  \]
\end{lemma}

\begin{proof}
  We first show that we can reduce to the case $\ker \varphi_j \not\subseteq \ker \varphi_{k}$ for $j \ne k$.
  Suppose $\ker \varphi_j \subseteq \ker \varphi_{k}$ for $j \ne k$.
  Because $\im \varphi_j$ and $\im \varphi_{k}$ both have rank $1$, the kernels both have rank $d-1$.
  Thus, the quotient $\ker \varphi_{k}/ \ker \varphi_j$ is finite.
  Then $\varphi_j(\ker(\varphi_{k}))$ is a finite subgroup of $\im \varphi_j \cong \Z$, and hence $\varphi_j(\ker(\varphi_k)) = \{0\}$. 
  We conclude $\ker \varphi_j = \ker \varphi_{k}$.

  Then 
  \[
  \varphi_{j}^{-1}(F_j) \cup \varphi_{k}^{-1}(F_{k}) = \varphi_{k}^{-1}(\widetilde{F_j} \cup F_{k}) 
  \quad\text{with}\quad \widetilde F_j \coloneqq \varphi_{k}(\varphi_j^{-1}(F_j)).
  \]
  Let $K \coloneqq \ker \varphi_j = \ker \varphi_{k}$ and let $\vec x \in \Z^d$ be such that $\vec x+K$ generates $\Z^d/K \cong \Z$.
  Let $n_j$,~$n_{k} \in \Z$ be such that $\varphi_j(\vec x)=n_j$ and $\varphi_{k}(\vec x)=n_{k}$.
  Then $\widetilde{F_j} \subseteq \tfrac{n_{k}}{n_j} F_j \cap \Z$, so $\widetilde{F_j}$ has zero density.
  Upon replacing $F_{k}$ by $\widetilde{F_j} \cup F_k$, we may therefore drop the homomorphism $\varphi_j$.
  Repeating this argument, we ultimately get $\ker \varphi_j \not \subseteq \ker \varphi_{k}$ for $j \ne k$.

  We proceed by induction on $d$ to show the claim of the lemma.
  The case $d=0$ is vacuously true because then necessarily $n=0$ and the union is empty while $\Z^0=\{0\}$.
  Now let $d \ge 1$ and suppose the claim holds for $d-1$.

  The image $\im \varphi_1$ is a nonzero subgroup of $\Z$.
  Such subgroups have positive density, hence $\im \varphi_1 \not\subseteq F_1$.
  Let $\vec x \in \Z^d \setminus \varphi_1^{-1}(F_1)$ and set $y_i \coloneqq \varphi_i(\vec x)$.
  Then
  \[
    \bigcup_{i=2}^n \big(\varphi_i^{-1}(F_i - y_i) \cap \ker(\varphi_1)\big) \subseteq \ker(\varphi_1).
  \]
  Since $\varphi_i |_{\ker(\varphi_1)} \ne 0$, the induction hypothesis implies that the inclusion is proper.
  Let $\vec x' \in \ker(\varphi_1)$ with $\varphi_i(\vec x') \not \in F_i - y_i$ for all $i \in [2,n]$.
  Then $\vec x + \vec x' \not \in \bigcup_{i=1}^n \varphi_i^{-1}(F_i)$.
\end{proof}

We derive a somewhat crude corollary to \cref{t:main-thm-unit-equations,t:derksen-masser} that holds in any characteristic and is sufficient for our purposes.

\begin{proposition} \label{p:unit-equation-crude}
  Let $S \subseteq \Gamma^{n+1}$ be the set of all \emph{non-degenerate} solutions to the unit equation \cref{eq:unit-equation}.
  If $\pi\colon \Gamma \to \Z$ is a group homomorphism, then
  \[
  F_{ij} = \{\, \pi(x_i/x_j) : (x_0, \dots, x_n) \in S \,\}
  \] 
  has zero density for all $i$,~$j \in [0,n]$.
\end{proposition}

\begin{proof}
  In characteristic $0$, each set $F_{ij}$ is finite by \cref{t:main-thm-unit-equations}.
  Suppose $\characteristic{K} = p > 0$.
  Replacing $K$ be the field generated by $\Gamma$ and $a_0$, \dots,~$a_n$ over $\mathbb F_p$, we may assume that $K$ is a finitely generated field, and hence that $\sqrt{\Gamma}$ is a finitely generated group.
  By \cref{t:derksen-masser}, the set $\mathbb P(S)$ is a finite union of sets of the form $[ \psi_{1}, \ldots, \psi_{k} ]_q\{\vec y\}$ as in \cref{eq:frob-orbit}.
  Fixing such a set and $i$,~$j \in [0,n]$, it therefore suffices to show that
  \[
      P_l \coloneqq \{\, \pi( x_i/x_j) : (x_0:\cdots:x_n) \in [ \psi_{k-l+1}, \ldots, \psi_{k} ]_q\{\vec y\} \,\}
  \]
  has zero density for all $l \in [0,k]$.
  To do so, we prove $\card{P_l \cap [-N,N]} \in O(\log(N)^l)$ by induction on $l$.

  First note that, for convenience, we can assume that $\pi$ is defined on all of $\sqrt\Gamma$.
  Indeed, since $\sqrt\Gamma$ is finitely generated abelian, the quotient $\sqrt{\Gamma}/\Gamma$ has finite exponent $e$.
  Extending first $\pi$ to $\pi\colon \sqrt{\Gamma} \to \Q$ (using that $\Q$ is injective as an abelian group), we have $\pi(\sqrt{\Gamma}) \subseteq \tfrac{1}{e} \Z$. Replacing $\pi$ by $e \pi$ (it is sufficient to prove the claim for this homomorphism) we have $\pi\colon \sqrt{\Gamma} \to \Z$.

  The set $P_0$ is a singleton, so it contains $O(1)$ elements.
  Let $l \ge 1$, and suppose $\card{P_{l-1} \cap [-N,N]} = O(\log(N)^{l-1})$.
  Let $\alpha_0$, \dots,~$\alpha_n \in \sqrt{\Gamma}$ be such that $\psi_{k-l+1}(x_0:\cdots:x_n) = (\alpha_0 x_0 :\cdots : \alpha_n x_n)$. 
  For $(x_0:\cdots:x_n) \in \mathbb P^n(\Gamma)$ we have
  \[
  \pi\bigg(\frac{\alpha_i^{-1} (x_i \alpha_i)^{q^e}}{\alpha_j^{-1} (x_j \alpha_j)^{q^e}} \bigg) = 
  q^e (y + c) - c \qquad\text{with $c = \pi(\alpha_i/\alpha_j)$ and $y = \pi(x_i/x_j)$}.
  \]
  If $q^e(y+c) -c \in [-N,N]$, then $\abs{q^e(y+c)} \le N + \abs{c}$.
  We get $\abs{y+c} \le N+\abs{c}$ and if $y+c \ne 0$, then also $e \le \log_q(N+\abs{c})$.
  Then $\abs{y} \le N' \coloneqq N+2\abs{c}$.
  Now there are at most $O(\log(N')^{l-1})$ such choices for $y \in P_{l-1}$ and $O(\log(N'))$ choices for $e$ (if $y+c \ne 0$, but if $y+c=0$ then $e$ does not matter).
  Thus 
  \[
  \card{P_l \cap [-N,N]} \in O(\log(N')^l) = O(\log(N)^l). \qedhere
  \]
\end{proof}

The argument we used in positive characteristic is similar to the one in \cite[Lemma 4.4]{bell-smertnig21}.
In either case, the fact that the non-degenerate solutions are parametrized by a finite union of orbits of the type in \cref{eq:frob-orbit}, shows that while there are infinitely many non-degenerate solutions, there are \emph{few} of them, in a suitable quantitative sense.

\section{Locally Bézivin implies Weakly Epimonomial} \label{sec:locbez}

Throughout this section, let $K$ be a field, embedded in some fixed algebraic closure $\algc{K}$, and let $V$ be a finite-dimensional $K$-vector space.
In this section we establish one of the main structural implications, namely that locally Bézivin groups are weakly epimonomial.

We start by verifying the basic properties of steady endomorphisms (\cref{l:steady-basic}).

\begin{proof}[Proof of \cref{l:steady-basic}]
  \begin{proofenumerate}
    \item[\ref{steady-basic:power}]
    If $\lambda \in \algc{K}$ is an eigenvalue of $A$, then $\lambda^n$ is an eigenvalue of $A^n$.
    If $A^n$ is steady, then $\lambda^n \in K$ by the definition of steadyness.

    Suppose conversely that $m \ge 1$ is such that $\lambda^m \in K$ for all eigenvalues $\lambda \in \algc{K}$.
    Choosing a suitable basis of $V$, the endomorphism $A^m$ can be put into Jordan normal form over $K$, and we identify it with its corresponding matrix. 
    Now take $n$ a multiple of $m$ such that 
    \begin{itemize}
    \item $(\lambda/\mu)^n=1$ if $\lambda$,~$\mu \in \algc{K}$ are eigenvalues of $A$ with $\lambda/\mu$ a root of unity,
    \item $A^n$ has no nontrivial nilpotent block (so $\ker(A^n)=\ker(A^{n+j})$ for all $j \ge 0$), and
    \item if $\characteristic{K} = p > 0$, then $A^n$ is diagonal.
    \end{itemize}
    The last property can be achieved because in a $p^e$-power of a Jordan block with eigenvalue $\lambda$, the entries in the $j$-th upper diagonal ($j \ge 0$) are of the form $\binom{p^e}{j} \lambda^{p^e - j}$.
    Choosing $e$ sufficiently large, the binomial coefficients are divisible by $p$, and the off-diagonal entries become $0$.

    With such a choice of $n$, it is not hard to see that $A^n$ is steady (see for instance \cite[Lemmas 11 and 42]{bell-smertnig23b}).
    
    \item[\ref{steady-basic:norootofunity}]
    Suppose $\lambda \ne \mu$ and $(\lambda/\mu)^n=1$ for $n \ge 1$.
    If $v$,~$w \in V$ with $Av=\lambda v$ and $Aw = \mu w$, then $K(v+w)$ is $A^n$-invariant but not $A$-invariant. 

    \item[\ref{steady-basic:pos-char-diagonalizable}]
    As in \ref{steady-basic:power} we can assume that $A$ is a matrix in Jordan normal form.
    Suppose $A$ is not diagonal.
    Then $A$ contains an upper-triangular Jordan block $J_k(\lambda)$ of size $k \ge 2$ and with $\lambda \in K$.
    Taking this block without restriction in the upper left corner, and denoting by $e_1$, \dots,~$e_d$ the standard basis, we have $A^n e_2 = \lambda^n e_2 + n \lambda^{n-1} e_1$.
    Thus, the space $Ke_2$ is $A^p$-invariant but not $A$-invariant. \qedhere
  \end{proofenumerate}
\end{proof}

\begin{lemma} \label{l:bezivin-diagonalize}
  Let $S \subseteq \GL(V)$ be a locally Bézivin semigroup.
  \begin{enumerate}
  \item \label{bez-diag:char0} If $\characteristic{K}=0$, then every $A \in S$ is diagonalizable over $\algc{K}$.
  \item \label{bez-diag:steady} Every steady $A \in S$ is diagonalizable over $K$.
  \end{enumerate}
\end{lemma}

\begin{proof}
  \begin{proofenumerate}
    \item[\ref{bez-diag:char0}]
    Let $A \in S$. 
    Fixing a suitable basis of $\algc{K} \otimes_K V$ we can assume that $A$ is an invertible matrix in Jordan normal form.
    Suppose that $A$ has an upper-triangular Jordan block $J_k(\lambda)$ with eigenvalue $\lambda$ and size $k \ge 2$.
    Then the $(1,2)$ entry of $J^n$ is $n \lambda^{n-1}$.
    Because the subsemigroup generated by $A$ is Bézivin, we have $n \lambda^{n-1} \in \sumset{M}{\Gamma_0}$ for a finitely generated $\Gamma \le \algc{K}^\times$ and $M \ge 0$.
    Hence, for every $n \ge 0$, we have  $n \in \sumset{M}{\Gamma_0'}$ with $\Gamma' = \Gamma \cup \{\lambda\}$.
    However, this is impossible, say, by applying \cite[Lemma 4.7]{bell-smertnig23} with the linear polynomial $y$.

    \item[\ref{bez-diag:steady}]
    By definition, all eigenvalues of $A$ are contained in $K$.
    If $\characteristic K=0$, then \ref{bez-diag:char0} implies that $A$ is diagonalizable over $K$.
    If $\characteristic K > 0$, then $A$ is diagonalizable by \ref{steady-basic:pos-char-diagonalizable} of \cref{l:steady-basic}. \qedhere
  \end{proofenumerate}
\end{proof}

\begin{remark}
  Statement \ref{bez-diag:char0} of \cref{l:bezivin-diagonalize} also follows from a theorem of Bézivin \cite{bezivin86}.
  He showed that a univariate $D$-finite power series in $\algc{K}\llbracket x \rrbracket$ whose coefficients form a Bézivin set is, firstly, already rational, and secondly has only simple poles.
  The coefficients of a rational series are given by an LRS, and the $n$-th coefficient can be represented in the form $uA^n v$ with $u \in \algc K^{1 \times d}$, $v \in \algc K^{d \times 1}$, and $A \in \algc{K}^{d \times d}$.
  Given any non-diagonalizable $A \in \GL_d(\algc{K})$, using the Jordan normal form, we can pick $u$, $v$ such that the rational series with coefficients $uA^n v$ does \emph{not} have simple poles and hence apply \cite[Théorème 4]{bezivin86} to obtain \cref{l:bezivin-diagonalize}\ref{bez-diag:char0}.
\end{remark}

\begin{example}
  Claim \ref{bez-diag:char0} of \cref{l:bezivin-diagonalize} is false in characteristic $p > 0$.
  Indeed,
  \[
  J_2(\lambda)^n=
  \begin{pmatrix}
  \lambda & 1 \\
  0 & \lambda
  \end{pmatrix}^{\pm n} =
  \begin{pmatrix}
    \lambda^{\pm n} & \pm n \lambda^{\pm n-1} \\
    0 & \lambda^{\pm n}
  \end{pmatrix}
    \quad\text{for $n \in \Z$,}
  \]
  shows that the group generated by a Jordan block of size two is Bézivin, because $n$ takes only finitely many values modulo $p$.
  Note however that the $p$-th power of the Jordan block is diagonal (and steady).
\end{example}

Suppose $V = V_1\oplus \cdots \oplus V_r$ with corresponding projections $P_i\colon V \to V$ such that $P_i|_{V_i}$ is the identity and $P_i|_{V_j} =0$ for $j \ne i$.
If $B \in \End(V)$, we denote by $[B]_{ij} \coloneqq P_i \circ B \circ P_j$ the $ij$ ``block'' of $B$, so that $B=\sum_{i,j=1}^r [B]_{ij}$.
If we fix bases of the $V_i$ and identify $B$ with its corresponding matrix, this indeed corresponds to a block of $B$.

We formulate the next proposition for a semigroup of endomorphisms because this gives us a slightly stronger statement, with implications for \cref{c:automaton-one-separating} later.

\begin{proposition} \label{p:full-term-bound}
  Let $S \subseteq \End(V)$ be a Bézivin semigroup.
  Let $A \in S \cap \GL(V)$ be steady and let $V=V_1\oplus \cdots \oplus V_r$ be the decomposition of $V$ into $A$-eigenspaces.
  Suppose that there exist $v=v_1+\cdots +v_r \in V$ \textup(with $v_i \in V_i$\textup), a linear map $\psi\colon V \to K$, a constant $M \ge 0$ and a finitely generated $\Gamma \le K^\times$ such that
  \[
  \psi(Bv) \in \sumset{M}{\Gamma_0} \quad\text{for all}\quad B \in S.
  \]
  Then, for all $l \ge 0$ and $B_1$, \dots,~$B_l \in S$, it holds that
  \[
  \card[\big]{\big\{\, (i_0,\cdots,i_l) \in [r]^{l+1} : \psi([B_1]_{i_0 i_1} [B_2]_{i_1 i_2} \cdots [B_l]_{i_{l-1}i_l} v_{i_l}) \ne 0 \,\big\}} \le M.
  \]
\end{proposition}

Thus, if we evaluate $\psi(B_1 B_2 \cdots B_l v)$ as a block product corresponding to the decomposition $V=V_1\oplus \cdots \oplus V_r$, the number of nonzero terms in the sum is bounded by $M$, no matter how large $l$ is.

\begin{proof}
  By \cref{l:bezivin-diagonalize}, the automorphism $A$ is diagonalizable over $K$, so that the required decomposition of $V$ indeed exists.
  Let $\lambda_i$ be the eigenvalue of $A$ on $V_i$.
  Keep in mind that $\lambda_1$, \dots,~$\lambda_r$ are pairwise distinct.

  We use multi-index notation $\vec n = (n_0,\dots, n_{l}) \in \Z^{l+1}$, $\vec i = (i_0,\dots, i_l) \in [r]^{l+1}$, $\vec \lambda(\vec i) = (\lambda_{i_0},\dots, \lambda_{i_l})$, and set
  \[
  \vec \lambda(\vec i)^{\vec n} \coloneqq \lambda_{i_0}^{n_0} \cdots \lambda_{i_{l}}^{n_{l}} \quad\text{and}\quad B(\vec i) \coloneqq [B_1]_{i_0i_1} [B_2]_{i_1 i_2} \cdots [B_l]_{i_{l-1}i_l}.
  \]
  We have to show $\{\, \vec i  \in [r]^{l+1} : \psi(B(\vec i)v_{i_l}) \ne 0 \,\} \le M$.

  For all $\vec n \in \Z^{l+1}$, 
  \[
  A^{n_0} B_1 A^{n_1} B_2 A^{n_2} B_3 \cdots A^{n_{l-1}} B_l A^{n_l}
  = \sum_{\vec i \in [r]^{l+1}} \vec \lambda(\vec i)^{\vec n} B(\vec i).
  \] 

  By assumption, there exist functions $\gamma_1$, \dots,~$\gamma_M \colon \Z^{l+1} \to \Gamma_0$ such that
  \begin{equation} \label{eq:bezivin-unit-eqn}
  \sum_{\vec i \in [r]^{l+1}} \vec \lambda(\vec i)^{\vec n} \psi(B(\vec i) v_{i_l}) = \gamma_1(\vec n) + \dots + \gamma_M(\vec n)
  \qquad\text{for all $\vec n \in \Z^{l+1}$}.
  \end{equation}
  We consider the equations~\eqref{eq:bezivin-unit-eqn} as solutions to a unit equation over a group $\Gamma'$, where $\Gamma'$ is generated by $\Gamma$ together with
  \begin{itemize}
  \item the elements $\psi(B(\vec i) v_{i_l})$ for $\vec i \in [r]^{l+1}$, and 
  \item the elements $\lambda_j$ for $j \in [r]$.
  \end{itemize}
  Note that these data remain fixed as $\vec n$ varies, so that $\Gamma'$ is indeed finitely generated and the same for all $\vec n$.

  For all subsets $\Omega \subseteq [r]^{l+1}$ and $\Omega' \subseteq [M]$, let $N(\Omega, \Omega') \subseteq \Z^{l+1}$ be the set of all those $\vec n$ for which
  \[ 
   \sum_{\vec i \in \Omega} \vec \lambda(\vec i)^{\vec n} \psi(B(\vec i) v_{i_l})= \sum_{k \in \Omega'} \gamma_k(\vec n),
  \]
  is a non-degenerate solution to a unit equation (that is, equality does not hold for a nonempty proper subsum).
  Because we can partition any solution to \cref{eq:bezivin-unit-eqn} into non-degenerate solutions, certainly $\Z^{l+1} \subseteq \bigcup_{\Omega, \Omega'} N(\Omega, \Omega')$.

  We now consider one such set $N(\Omega,\Omega')$ with $\card{\Omega} \ge 2$.
  Let $\vec i \ne \vec j \in \Omega$.
  By non-degeneracy, necessarily $\psi(B(\vec i)v_{i_l})$,~$\psi(B(\vec j)v_{j_l}) \ne 0$.
  Consider the group homomorphism
  \[
  \varphi_{\vec i, \vec j}\colon \Z^{l+1} \to \Gamma', \quad \vec n \mapsto \frac{\vec \lambda(\vec i)^{\vec n}}{\vec \lambda(\vec j)^{\vec n}} = 
  \bigg(\frac{\lambda_{i_0}}{\lambda_{j_0}}\bigg)^{n_0} \cdots \bigg(\frac{\lambda_{i_{l}}}{\lambda_{j_{l}}}\bigg)^{n_{l}}.
  \]
  Because $\vec i \ne \vec j$, there exists some $\nu \in [0,l]$ such that $i_\nu \ne j_\nu$.
  Then also $\lambda_{i_\nu} \ne \lambda_{j_\nu}$, since these eigenvalues are pairwise distinct.
  But $A$ is steady, so $\lambda_{i_\nu} / \lambda_{j_\nu}$ has infinite order in $K^\times$ by \ref{steady-basic:norootofunity} of \cref{l:steady-basic}.
  Thus, the image $\im \varphi_{\vec i, \vec j} \le \Gamma'$ has a torsion-free component.
  Hence, there exists a homomorphism $\pi_{\vec i, \vec j}\colon \Gamma' \to \Z$ that is nonzero on $\im \varphi_{\vec i, \vec j}$. 
  By \cref{p:unit-equation-crude}, the set $F_{\vec i,\vec j} \coloneqq \pi_{\vec i, \vec j} \circ\varphi_{\vec i, \vec j}\big(N(\Omega,\Omega')\big)$ has zero density.

  Now
  \[
  N \coloneqq \bigcup_{\substack{\Omega, \Omega' \\ \card{\Omega} \ge 2}} N(\Omega, \Omega') \,\subseteq\, \bigcup_{\vec i, \vec j} (\pi_{\vec i,\vec j} \circ \varphi_{\vec i, \vec j})^{-1}(F_{\vec i,\vec j}).
  \]
  By \cref{l:no-low-density-cover} necessarily $N \subsetneq \Z^{l+1}$.
  
  Fix $\vec n \in \Z^{l+1} \setminus N$.
  Partitioning the solution \cref{eq:bezivin-unit-eqn} into non-degenerate subequations for this specific $\vec n$, we can, for every $\vec i \in [r]^{l+1}$, find a subset $\Omega'(\vec i) \subseteq [M]$ such that $\vec n \in N(\{\vec i\}, \Omega'(\vec i))$ and such that $\Omega'(\vec i) \cap \Omega'(\vec j) =\emptyset$ for $\vec i \ne \vec j$.
  If $\varphi(B(\vec i)v_{i_l}) \ne 0$, then necessarily $\Omega'(\vec i) \ne \emptyset$.
  The pairwise disjointness of these sets implies $\card{\{ \vec i \in [r]^{l+1} : \varphi(B(\vec i)v_{i_l}) \ne 0\}} \le M$.
\end{proof}

\Cref{p:full-term-bound} in this form will only be needed at the end, when we deal with automata.
In dealing with linear groups, we use the following corollary.

\begin{corollary} \label{c:term-bound}
  Let $S \subseteq \End(V)$ be a Bézivin semigroup.
  Suppose that there exists a steady $A \in S \cap \GL(V)$.
  Let $V=V_1\oplus \cdots \oplus V_r$ be the decomposition of $V$ into $A$-eigenspaces.
  Then there exists a constant $C \in \Z_{\ge 0}$ such that, for all $l \ge 0$ and $B_1$, \dots,~$B_l \in S$, it holds that
  \[
  \card[\big]{\big\{\, (i_0,\cdots,i_l) \in [r]^{l+1} : [B_1]_{i_0 i_1} [B_2]_{i_1 i_2} \cdots [B_l]_{i_{l-1}i_l} \ne 0 \,\big\}} \le C.
  \]
\end{corollary}

In other words, in evaluating $B_1\cdots B_l$ as a block-matrix product with respect to that decomposition of $V$, there are at most $C$ nonzero terms (independent of $l$).

\begin{proof}
  We use the same notation as in the proof of \cref{p:full-term-bound}.
  Now we have to show $\{\, \vec i  \in [r]^{l+1} : B(\vec i) \ne 0 \,\} \le C$.
  Let $e_1$, \dots,~$e_d$ be a basis of $V$ with dual basis $e_1^*$, \dots,~$e_d^*\colon V \to K$.
  It suffices to show that for all $\mu$,~$\nu \in [d]$, there exists at most $M \ge 0$ multi-indices $\vec i$ for which $e_\mu^* B(\vec i) e_\nu \ne 0$.
  Then the required bound holds with $C=Md^2$.
  But for fixed $\mu$, $\nu$ this follows from \cref{p:full-term-bound} by setting $\psi =e_\mu^*$ and $v=e_\nu$.
\end{proof}

\begin{lemma} \label{l:permuting-spaces}
  Let $S \subseteq \GL(V)$ be a locally Bézivin semigroup, let $A \in S$ be steady, and let $V = V_1 \oplus \cdots \oplus V_r$ be the decomposition of $V$ into eigenspaces of $A$.
  \begin{enumerate}
  \item \label{perm:general} For every $B \in S$, there exist decompositions $V_i = V_{i1} \oplus \cdots \oplus V_{is_i}$ such that $B$ permutes the spaces $\{\, V_{ij} : i \in [r],\, j \in [s_i] \,\}$. 
  \item \label{perm:invariant} If $B$ is steady, then $BV_{ij}=V_{ij}$ for every $i$,~$j$ in a decomposition as in \ref{perm:general}. In particular, it holds that $BV_i = V_i$ for all $i \in [r]$.
  \end{enumerate}
\end{lemma}

\begin{proof}
  \begin{proofenumerate}
  \item[\ref{perm:general}]
  For every sequence $\vec m = (m_j)_{j\ge 0}$ with $m_j \in [r]$, let 
  \[
  W(\vec m) = \{\, v \in V : B^j v \in V_{m_j} \text{ for all $j \ge 0$}\,\}.
  \]
  Observe that $W(\vec m)$ is a vector space and $BW(m_0,m_1,m_2,\dots)\subseteq W(m_1,m_2,\dots)$.
  We will show
  \begin{equation} \label{eq:decomp-seq}
  V = \bigoplus_{\vec m} \, W(\vec m).
  \end{equation}
  
  Because $W(\vec m) \subseteq V_{m_0}$, and thus in particular for each $\vec m$ and $i \in [r]$ either $W(\vec m) \cap V_{i}=0$ or $W(\vec m) \subseteq V_i$, it will then follow that
  \[
  V_i = \bigoplus_{\substack{\vec m \\ W(\vec m) \subseteq V_i}} W(\vec m).
  \]
  Renaming the finitely many nonzero spaces in \cref{eq:decomp-seq} to $V_{i1}$, \dots,~$V_{is_i}$ will then yield claim \ref{perm:general}.

  First, we show that the sum in \cref{eq:decomp-seq} is direct.
  Let 
  \begin{equation} \label{eq:wm-relation}
    0 = \sum_{\vec m} w(\vec m)
  \end{equation}
  with $w(\vec m) \in W(\vec m)$, almost all of which are zero.
  Supposing, for the sake of contradiction, that some $w(\vec m)$ is not zero, we can also take the sum in \cref{eq:wm-relation} in such a way that the number of nonzero summands is minimal.
  Then there are at least two nonzero summands, say $w(\vec m')$ and $w(\vec m'')$.
  We may suppose that $\vec m'=(m_k')_{k \ge0}$ and $\vec m''=(m_k'')_{k \ge 0}$ differ in the $j$-th place.
  Applying $B^j$ to \cref{eq:wm-relation},
  \[
  0 = \sum_{\vec m} B^j w(\vec m) = \sum_{i=1}^r \underbrace{\sum_{\substack{\vec m=(m_k)_{k \ge 0}\\m_j = i}} B^j w(\vec m)}_{\in V_i}.
  \]
  Because the sum $V_1 \oplus \cdots \oplus V_r$ is direct and $B$ is invertible, we obtain
  \[
  0 = \sum_{\substack{\vec m=(m_k)_{k \ge 0}\\m_j = m_j'}} w(\vec m) = \sum_{\substack{\vec m=(m_k)_{k \ge 0}\\m_j = m_j''}} w(\vec m).
  \]
  Now $w(\vec m')$ appears in the first sum, but not the second, and conversely for $w(\vec m'')$.
  Thus, both sums are nonempty, but shorter than the one in \cref{eq:wm-relation}, contradicting the minimality of its choice.

  We still need to show $V = \sum_{\vec m} W(\vec m)$.
  As in \cref{c:term-bound}, we write $[A]_{ij}$ and $[B]_{ij}$ for blocks of $A$, respectively, of $B$ corresponding to the decomposition $V_1 \oplus \cdots \oplus V_r$. 
  For $l \ge 1$ and $\vec i=(i_0,\cdots,i_l) \in [r]^{l+1}$, define
  \[
  B(\vec i)\coloneqq B_{i_l i_{l-1}} \cdots B_{i_2 i_1}B_{i_1 i_0}.
  \]
  and let  
  \[
  \Omega(l) \coloneqq \{\, \vec i \in [r]^{l+1} :  B(\vec i)\ne 0\,\}.
  \]

  By \cref{c:term-bound}, applied to the Bézivin subsemigroup $\langle A, B \rangle$ of $S$, there exists $C \ge 0$ such that $\card{\Omega(l)} \le C$ for all $l$.
  If $\vec i=(i_0,\dots,i_l) \in \Omega(l)$, then $B \cdot B(\vec i) \ne 0$, since $B$ is invertible.
  Hence, there must exist $i_{l+1}$ with $B(i_0,i_1,\dots,i_{l+1}) \ne 0$ and hence $\vec i'=(i_1,\dots,i_{l+1}) \in \Omega(l+1)$.

  Thus, the cardinality $\card{\Omega(l)}$ is non-decreasing in $l$ and must stabilize, say at some $l=l_0$.
  But then every $\vec i \in \Omega(l)$ with $l \ge l_0$ can be extended \emph{uniquely} to some $\vec i'\in \Omega(l+1)$.
  This uniqueness implies $B \cdot B(\vec i) V \subseteq V_{i_{l+1}}$.
  Doing this for all $l \ge l_0$, we find a sequence $\vec m=(m_k)_{k\ge 0}$ such that $B^k B(\vec i)V \subseteq V_{m_k}$ for all $k \ge 0$. 
  Hence, for every $\vec i \in \Omega(l_0)$, there exists a sequence $\vec m$ such that $B(\vec i) V \subseteq W(\vec m)$.

  Surjectivity of $B$ implies $\sum_{\vec i \in \Omega(l_0)} B(\vec i) V = V$, and thus
  \[
    V = \sum_{\vec i \in \Omega(l_0)} B(\vec i) V \subseteq \sum_{\vec m} W(\vec m).
  \]

  \item[\ref{perm:invariant}]
  Suppose that we have a decomposition as in \ref{perm:general}.
  Then there exists some $n \ge 1$ such that $B^n V_{ij}=V_{ij}$ for all $i$,~$j$.
  Since $B$ is steady, therefore $BV_{ij}=V_{ij}$. \qedhere
  \end{proofenumerate}
\end{proof}

\begin{lemma} \label{l:bezivin-is-power-splitting}
  If $S \subseteq \GL(V)$ is a locally Bézivin semigroup, then $K$ is a power-splitting field for $S$.
\end{lemma}

\begin{proof}
  Fix $A \in S$ and consider $\overline{V}=\algc{K} \otimes_K V$.
  Without restriction, we can consider the finitely generated subgroup $S$ generated by $A$, so that $S$ is Bézivin.
  We identify $\GL(\overline{V}) \cong \GL_d(\overline{K})$ using a basis in which $A$ has Jordan normal form.
  Then there exist a finitely generated subgroup $\Gamma \le K^\times$ and coefficients $a_1$, \dots,~$a_m \in \algc{K}^\times$ such that all entries of elements of $S$ are contained in $\sum_{i=1}^M a_i \Gamma_0$.
  We stress that $\Gamma$ can be taken over the base field $K$ by our assumption, while the coefficients can only be taken in $\algc{K}$.
  If $\lambda$ is an eigenvalue of $A$, then $\lambda^n \in \sum_{i=1}^M a_i \Gamma_0$ for all $n \in \Z$.
  \Cref{l:powers-bezivin} shows $\lambda^m \in \Gamma_0 \subseteq K$ for some $m \ge 1$.
\end{proof}

Given a set $\mathcal S \subseteq \End(V)$, we say that $W \subseteq V$ is a \defit{joint-eigenspace of $\mathcal S$} if for every $A \in \mathcal S$ there exists an eigenspace $W_A$ such that
\[
W = \bigcap_{A \in \mathcal S} W_A.
\]

\begin{lemma} \label{l:simultaneous-decomposition}
  Let $S \subseteq \GL(V)$ be a locally Bézivin semigroup.
  Let $V_1$, \dots,~$V_r$ be the joint-eigenspaces of all steady elements of $S$.
  Then $V_1\oplus \cdots \oplus V_r$ is a weakly epimonomial decomposition for $S$.
\end{lemma}

The weakly epimonomial decomposition in the lemma is the coarsest possible. Since each component of any weakly epimonomial decomposition is contained in an eigenspace of all steady elements, every weakly epimonomial decomposition of $S$ must be a refinement of the one obtained here.

\begin{proof}
  We first show by induction on $n$: if $A_1$, \dots,~$A_n \in S$ are steady, then there exists a decomposition $V=V_1\oplus \cdots \oplus V_r$ such that every $V_i$ is a joint-eigenspace of $A_1$, \dots,~$A_n$. 

  For $n=0$, the claim holds vacuously.
  Let $n\ge 1$ and suppose the claim holds for $n-1$.
  Let $i \in [r]$. 
  We know $V_i = W_{i,1} \cap \cdots \cap W_{i,{n-1}}$ with each $W_{i, j}$ an eigenspace of $A_j$ for $j \in [n-1]$.
  By \ref{perm:invariant} of \cref{l:permuting-spaces}, each $W_{i, j}$ is invariant under $A_n$.
  Hence, each $V_i$ is invariant under $A_n$.
  Since $A_n$ is diagonalizable over $K$ by \cref{l:bezivin-diagonalize}, the restriction $A_n|_{V_i}$ is also diagonalizable.
  We can therefore decompose $V_i$ into eigenspaces of $A_n|_{V_i}$, and since each such space is of the form $U \cap V_i$ with $U$ an eigenspace of $A_n$, the claim of the induction follows.

  We can now decompose $V= V_1 \oplus \cdots \oplus V_r$ in such a way that each $V_i$ is a joint-eigenspace of \emph{all} steady $A \in S$.
  Indeed, decomposing iteratively, at each step each $V_i$ is either already contained in an eigenspace of each steady $A \in S$, and we are done, or there exists some steady $A \in S$ for which this is not the case.
  In the second case, we use $A$ to decompose $V_i$ further.
  Since $V$ is finite-dimensional, this process stops after finitely many refinements.

  Let $B \in S$.
  We now show that $B$ permutes the spaces $V_1$, \dots,~$V_r$.
  By \ref{perm:general} of \cref{l:permuting-spaces}, we can refine $V_i = V_{i1} \oplus \cdots \oplus V_{is_i}$ so that $B$ permutes the spaces $\{\, V_{ij} : i \in [r], j \in [s_i] \,\}$.
  Let $f$,~$g $ be the functions for which $B V_{ij} = V_{f(i,j),g(i,j)}$.
  We will be done if we can show that $f(i,j)$ does not actually depend on $j$, because then $BV_i \subseteq V_{f(i,*)}$.

  Suppose there exist $i' \in [r]$ and $j' \ne j'' \in [s_{i'}]$ such that $f(i',j') \ne f(i',j'')$.
  By construction of the $V_i$, there exists a steady $A \in S$ such that $A$ has distinct eigenvalues on $V_{f(i',j')}$ and $V_{f(i',j'')}$.
  By \cref{l:bezivin-is-power-splitting} the field $K$ is power-splitting for $S$.
  Hence, there exists $n \ge 1$ such that $B^n$ is steady (using \ref{steady-basic:power} of \cref{l:steady-basic}).
  Then each $V_i$ is also a subspace of an eigenspace of $B^n$.
  Let $\alpha_i$ be the eigenvalue of $A$ on $V_i$, and $\beta_i$ the eigenvalue of $B^n$.

  Since $B$ permutes the spaces $V_{ij}$, we have
  \[
  B^{n-1} A B|_{V_{ij}} = \alpha_{f(i,j)} B^n|_{V_{ij}} = \alpha_{f(i,j)} \beta_i \id_{V_{ij}}\qquad\text{for $i \in [r]$ and $j \in s_i$}.
  \]
  Let $C \coloneqq (B^{n-1} A B)^m$ with $m \ge 1$ such that $C$ is steady (again using \cref{l:bezivin-is-power-splitting,l:steady-basic}).
  Then $C$ has eigenvalue $\alpha_{f(i',j')}^m \beta_{i'}^m$ on $V_{i'j'}$ and eigenvalue $\alpha_{f(i',j'')}^m \beta_{i'}^m$ on $V_{i'j''}$.
  Because $A$ is steady, these two eigenvalues are distinct (using \ref{steady-basic:norootofunity} of \cref{l:steady-basic}).
  This contradicts the choice of $V=V_1\oplus \cdots \oplus V_r$. 
  (In other words, the automorphism $C$ would allow us to further refine the decomposition, contradicting the maximally refined choice).
\end{proof}

\section{Characterization of Locally Bézivin Groups} \label{sec:bezchar}

Let $K$ be a field and $V$ a finite-dimensional vector space.
We are now in a position to establish the characterization of locally Bézivin groups.
Let us first show that a weakly epimonomial decomposition of a semigroup of invertible endomorphisms lifts to the group.

\begin{lemma} \label{l:lift-wepi-to-group}
  Let $S \subseteq \GL(V)$ be a semigroup and let $G = \langle S \rangle \le \GL(V)$.
  \begin{enumerate}
  \item \label{lift:wepi} If $S$ is weakly epimonomial, then $G$ is weakly epimonomial.
  \item \label{lift:power-splitting} If $S$ is weakly epimonomial as a subsemigroup of $\GL(L \otimes_K V)$ for some extension field $L$ and $K$ is power-splitting for $S$, then $K$ is also power-splitting for $G$.
  \end{enumerate}
\end{lemma}

\begin{proof}
  \begin{proofenumerate}
    \item[\ref{lift:wepi}]
    We have to verify the properties in \cref{d:weakly-epimonomial} for the inclusion $G \hookrightarrow \GL(V)$. 
    Since property \ref{d-wm:permutation} holds for all $A \in S$, it also holds for all inverses and products, and hence holds for all $A \in G$.

    Let $D \subseteq S$ be the diagonal.
    Let $\widehat D \le \GL(V)$ consists of all $A \in \GL(V)$ for which every $V_i$ is contained in an eigenspace of $A$.
    Thus, we have $D = S \cap \widehat D$.
    After fixing bases of $V_1$, \dots,~$V_r$ to obtain a basis of $V$, the elements of $\widehat D$ are the diagonal matrices that are scalar when restricted to each $V_i$.
    If $P \in \GL(V)$ permutes the spaces $V_1$, \dots,~$V_r$, then $P^{-1}\widehat D P \subseteq \widehat D$.

    To verify \ref{d-wm:intersection} and \ref{d-wm:splitting} of \cref{d:weakly-epimonomial}, we first check:
    \begin{enumerate}[label=(\roman*)]
    \item\label{lift-wepi:repr} every element of $G$ is of the form $AB^{-1}$ with $A \in S$ and $B \in \widehat D$, where the eigenvalues of $B$ are contained in the group generated by the spectrum of $S$.
    \item\label{lift-wepi:power} if $A \in S$ and $B \in \widehat D$, then some power of $AB^{-1}$ is in $\widehat D$.
    \end{enumerate}

    To verify \ref{lift-wepi:repr}, note that if $A \in S$, then weak epimonomiality of $S$ together with \cref{l:steady-basic} implies $A^n \in D$ for some $n \ge 1$.
    If $B \coloneqq A^n$, then $A (A^{n-1}B^{-1})=\id$, so that $A^{-1}=A^{n-1}B^{-1}$.
    Thus, as a monoid, the group $G$ is generated by $S$ and $D^{-1}$.
    Explicitly, elements of $G$ can be represented as $A_1B_1^{-1} A_2 B_2^{-1} \cdots A_k B_k^{-1}$ with $A_i \in S \cup \{\id\}$ and $B_i \in D \cup \{\id\}$.
    We can express $B_1^{-1} A_2 = A_2 (A_2^{-1} B_1^{-1} A_2)$ with $(A_2^{-1} B_1^{-1} A_2) \in \widehat D$, because $A_2$ permutes $V_1$, \dots,~$V_r$.
    Iterating this, we find $A_1B_1^{-1} A_2 B_2^{-1} \cdots A_k B_k^{-1} = A_1\cdots A_k B^{-1}$ with $B \in \widehat D$ (but possibly $B \not \in G$).
    
    For the claim about the eigenvalues, note that the eigenvalues of $P^{-1}B^{-1}P$, with $B \in S$ and $P$ permuting the spaces $V_1$, \dots,~$V_r$, are contained in the group $\Gamma \le K^\times$ generated by all the eigenvalues of $S$.
    Since the matrices of the form $P^{-1}B^{-1}P$ are moreover diagonal, the same holds true for a product of such matrices.

    To see \ref{lift-wepi:power}, note
    \begin{equation} \label{eq:ab-power}
    (AB^{-1})^n = A^n \underbrace{ (A^{-(n-1)} B A^{n-1}) (A^{-(n-2)} B A^{n-2}) \cdots    (A^{-1} B^{-1} A) B^{-1} }_{\in \widehat D}.
    \end{equation}
    Since $S$ is weakly epimonomial, there exists some $n$ such that $A^n$ is steady (by \cref{l:steady-basic}), hence $A^n \in D$. 
    We get $(AB^{-1})^n \in \widehat D$.
    
    Now \ref{lift-wepi:repr} and \ref{lift-wepi:power} together imply that $K$ is power-splitting for $G$, verifying \ref{d-wm:splitting} of weak epimonomiality.
    Let $C \in G$ be steady.
    Since we can express $C=AB^{-1}$ with $A \in S$ and $B \in \widehat D$, each $V_i$ is contained in an eigenspace of some power $C^n$ of $C$ (again using \ref{lift-wepi:repr} and \ref{lift-wepi:power}).
    Hence, every subspace of $V_i$ is invariant under $C^n$, and by steadyness, also under $C$.
    It follows that $V_i$ is contained in an eigenspace of $C$.

    \item[\ref{lift:power-splitting}]
    Let $C \in G$.
    Let $\widehat V = L \otimes_K V$ and let $\widehat V_1 \oplus \cdots \oplus \widehat V_r$ be a weakly epimonomial decomposition for $S \subseteq \GL(\widehat V)$.
    Applying \ref{lift-wepi:repr} from the proof of \ref{lift:wepi}, we can again express $C=AB^{-1}$ with $A \in S$ and $B \in \GL(\widehat V)$ such that each $\widehat V_i$ is contained in an eigenspace of $B$.
    The eigenvalues of $B$ are contained in the group $\Gamma \le L^\times$ generated by all the eigenvalues of $S$.
    By \cref{l:steady-basic}, there exists $n \ge 1$ such that $A^n$ is steady, and hence, in particular, diagonal with respect to any basis obtained by refining $\widehat V_1 \oplus \cdots \oplus \widehat V_r$.
    \cref{eq:ab-power} shows that $(AB^{-1})^n$ is also diagonal with all eigenvalues in $\Gamma$.
    By the power-splitting hypothesis for $S$, some power of $(AB^{-1})^n$ has all its eigenvalues in $K$.
    \qedhere
  \end{proofenumerate}
\end{proof}

\begin{proposition} \label{p:wepi-torsion}
  If $G \le \GL(V)$ has a weakly epimonomial decomposition with diagonal $D$, then $D \trianglelefteq G$ and $G/D$ is a linear torsion group.
\end{proposition}

\begin{proof}
  Let $V=V_1\oplus\cdots\oplus V_r$ be the given weakly epimonomial decomposition.
  By definition, the diagonal $D$ contains all steady $A \in G$, and $D$ is clearly a subgroup of $G$.
  To see that it is normal, let $A \in D$ and $B \in G$.
  Let $i \in [r]$ and let $k \in [r]$ be the unique index for which $BV_i = V_k$.
  If $\lambda$ is the eigenvalue of $A$ on $V_k$, then $B^{-1}AB|_{V_i} = B^{-1}\lambda B|_{V_i} = \lambda \id|_{V_i}$.
  Hence, we have $B^{-1}A B \in D$.

  For every $B \in G$ some power $B^n$ is steady (using \cref{l:steady-basic}), and hence $B^n \in D$.
  Thus, the quotient $G/D$ is a torsion group.
  We show that it is a linear group, that is, that it embeds into some $\GL(W)$ with $W$ a finite-dimensional $K$-vector space.
  Note that $G$ acts on $\End(V)$ by conjugation, that is, there is a group homomorphism $\Phi\colon G \to \GL(\End(V))$, $B \mapsto \Phi_B$ with $\Phi_B(X) = BXB^{-1}$ for all $X \in \End(V)$.
  Since $B$ permutes $V_1$, \dots,~$V_r$, this restricts to a representation $\varphi$ of $G$ on $W \coloneqq \End(V_1) \oplus \cdots \oplus \End(V_r)$ (or, after fixing a basis, on block-diagonal matrices).
  Explicitly $\varphi\colon G \to \GL(W)$, $B \mapsto \varphi_B$ is given by
  \[
  \varphi_B(X_1,\dots, X_r) = (BX_{\pi(1)}B^{-1}|_{V_1}, \dots, B X_{\pi(r)}B^{-1}|_{V_r}),
  \]
  with $\pi$ the permutation for which $B^{-1}V_i = V_{\pi(i)}$, and $X_i \in \End(V_i)$.
  Because the center of $\End(V_i)$ consists of scalar multiples of the identity, we see that $B \in \ker \varphi$ if and only if $\pi=\id$ and $B|_{V_i}$ is a scalar multiple of the identity for each $i \in [r]$.
  Hence, we find $\ker(\varphi) = D$, and $G/D \cong \im \varphi$ is a linear group.
\end{proof}

\begin{remark}
  Let $\widetilde G$ denote the Zariski closure of $G$.
  Then $\widetilde D(K)$ consists of elements for which each $V_i$ is contained in an eigenspace,
  and hence $D = \widetilde D(K) \cap G$ by the definition of the diagonal (here $\widetilde D(K)$ denotes the set of $K$-rational points of the scheme $\widetilde D$).
  Since $\widetilde D \trianglelefteq \widetilde G$ is a closed normal subgroup, Chevalley's theorem on linear algebraic groups implies that $\widetilde G / \widetilde D$ is a linear algebraic group as well, and therefore $G/D$ is a linear group.
  The previous proof makes the arguments in Chevalley's theorem explicit in this special case without reference to linear algebraic groups.
  Since the argument is short, and also has the benefit of making explicit the dimension of $W$ (which is useful for estimates as in \cite[Appendix C]{bell-smertnig23b}), we have chosen to give this explicit variant.
\end{remark}

We obtain a characterization of locally Bézivin (semi)groups that holds over any field.
We need the following classical theorems of Burnside and Schur.
\begin{theorem}[Burnside--Schur] \label{t:burnside-schur}
  Let $G \le \GL(V)$.
  \begin{enumerate}
  \item \textup{(Schur)} If $G$ is a torsion group, then $G$ is locally finite.
  \item \textup{(Burnside)} If $\characteristic{K}=0$ and $G$ has finite exponent, then $G$ is finite.
  \end{enumerate}
\end{theorem}

\begin{theorem} \label{t:locally-bezivin-equivalent}
  Let $S \subseteq \GL(V)$ be a semigroup and let $G = \langle S \rangle$.
  Then the following statements are equivalent.
  \begin{equivenumerate}
    \item \label{locequiv:bezivin} The group $G$ is locally Bézivin.
    \item \label{locequiv:semigroup-bezivin} The semigroup $S$ is locally Bézivin.
    \item \label{locequiv:semigroup-wepimon} The semigroup $S$ is weakly epimonomial.
    \item \label{locequiv:wepimon} The group $G$ is weakly epimonomial.
    \item \label{locequiv:finsub} The group $G$ is locally epimonomial.
    \item \label{locequiv:locdiag} The group $G$ is locally virtually simultaneously diagonalizable \textup(over $K$\textup).
  \end{equivenumerate}
\end{theorem}

The last statement means that every finitely generated subgroup $H \le G$ has a finite-index subgroup $D \le H$ that is simultaneously diagonalizable.

\begin{proof}
  \begin{proofenumerate}
    \item[\ref{locequiv:bezivin}$\,\Rightarrow\,$\ref{locequiv:semigroup-bezivin}:] Trivial.
    
    \item[\ref{locequiv:semigroup-bezivin}$\,\Rightarrow\,$\ref{locequiv:semigroup-wepimon}:] By \cref{l:simultaneous-decomposition}.
    
    \item[\ref{locequiv:semigroup-wepimon}$\,\Rightarrow\,$\ref{locequiv:wepimon}:]
    By \cref{l:lift-wepi-to-group}.
    
    \item[\ref{locequiv:wepimon}$\,\Rightarrow\,$\ref{locequiv:finsub}:]
    Let $H \le G$ be a finitely generated subgroup.
    Then $H$ is also weakly epimonomial.
    Let $D \triangleleft H$ be the diagonal with respect to some weakly epimonomial decompositon.
    By \cref{p:wepi-torsion} the quotient $H/D$ is a finitely generated linear torsion group.
    Burnside--Schur (\cref{t:burnside-schur}) implies that $H/D$ is finite.

    \item[\ref{locequiv:finsub}$\,\Rightarrow\,$\ref{locequiv:locdiag}:]
    Let $H \le G$ be finitely generated. 
    Then $H$ is epimonomial with a diagonal $D \triangleleft H$ that has finite index, and $D$ is simultaneously diagonalizable (over $K$).

    \item[\ref{locequiv:locdiag}$\,\Rightarrow\,$\ref{locequiv:bezivin}:]
    Let $H \le G$ be finitely generated and let $D \le H$ be a finite-index subgroup of simultaneously diagonalizable elements.
    Let $B_1$, \dots,~$B_n$ be a system of representatives for the left cosets of $H/D$.
    It suffices to show that each element of $B_i D$ has entries in $\sumset{M}{\Gamma_0}$ for some $M \ge 0$ and a finitely generated $\Gamma \le K^\times$. 
    Now $D$, as a finite-index subgroup in the finitely generated group $H$, is also finitely generated \cite[1.6.11]{robinson96}.
    Since the elements of $D$ are moreover simultaneously diagonalizable, the claim follows.  \qedhere
  \end{proofenumerate}
\end{proof}

In particular, we see that the local Bézivin property for a semigroup $S \subseteq \GL(V)$ is equivalent to the local Bézivin property for the group $G=\langle S \rangle$.
Hence, we may usually deal with groups for the remainder of the section.

\begin{corollary} \label{l:fg-bezivin-epimonomial}
  A finitely generated $G \le \GL(V)$ is Bézivin if and only if it is epimonomial.  
\end{corollary}

\begin{proof}
  By \cref{t:locally-bezivin-equivalent}.
\end{proof}

By \cref{t:locally-bezivin-equivalent} and \cref{l:simultaneous-decomposition} every weakly epimonomial $G \le \GL(V)$ gives the coarsest (hence canonical) decomposition $V=V_1 \oplus \cdots \oplus V_r$ with $V_i$ an intersection of eigenspaces of all steady elements of $G$.
However, even if $G$ is epimonomial, it may not be epimonomial with respect to this particular decomposition, as the next example shows.

\begin{example}
  Let
  \[
  G \coloneqq \begin{pmatrix} \mu(\C) & 0 \\ 0 & 1 \end{pmatrix} \le \GL_2(\C).
  \]
  Then $G$ is an infinite torsion group and the identity is the only steady element.
  The weakly epimonomial decomposition of $\C^2$ arising from the steady elements is therefore just $\C^2$ itself, and $G$ is not epimonomial with respect to the corresponding diagonal.
  However, clearly $\C^2 = \C \oplus \C$ is an epimonomial decomposition.
\end{example}

We now show that, in \cref{t:locally-bezivin-equivalent}, we can also relax the locally Bézivin property to hold over an extension field, as long as $K$ is power-splitting for $S$.

\begin{proposition} \label{p:equiv-field}
  For a finitely generated semigroup $S \subseteq \GL(V)$ the following statements are equivalent.
  \begin{equivenumerate}
  \item \label{equiv-field:bezivin} The semigroup $S$ is Bézivin.
  \item \label{equiv-field:weakbezivin-all} The semigroup $S \subseteq \GL(L \otimes_K V)$ is Bézivin for every extension field $L/K$ and $K$ is a power-splitting field for $S$. 
  \item \label{equiv-field:weakbezivin} The semigroup $S \subseteq \GL(L \otimes_K V)$ is Bézivin for some extension field $L/K$ and $K$ is a power-splitting field for $S$. 
  \end{equivenumerate}
\end{proposition}

\begin{proof}
  \begin{proofenumerate}
    \item [\ref{equiv-field:bezivin}$\,\Rightarrow\,$\ref{equiv-field:weakbezivin-all}:]
    Clearly, the semigroup $S$ is still Bézivin over $L$.
    By \cref{l:bezivin-is-power-splitting} the field $K$ is power-splitting for $S$.
  
    \item [\ref{equiv-field:weakbezivin-all}$\,\Rightarrow\,$\ref{equiv-field:weakbezivin}:] Trivial.

    \item [\ref{equiv-field:weakbezivin}$\,\Rightarrow\,$\ref{equiv-field:bezivin}:]
    Let $G = \langle S \rangle$.
    We identify $\GL(V) = \GL_d(K)$ by choosing a basis of $V$.
    Extending scalars (that is, viewing $G$ as a subgroup of $\GL_d(L)$), we obtain an $L$-representation of $G$.
    By \cref{t:locally-bezivin-equivalent} the group $G$ is weakly epimonomial as a subgroup of $\GL_d(L)$.
    Then \ref{lift:power-splitting} of \cref{l:lift-wepi-to-group} shows that $K$ is power-splitting for $G$.
    
    By \ref{locequiv:locdiag} of \cref{t:locally-bezivin-equivalent}, there exists a finite-index subgroup $T \le G$ that is simultaneously diagonalizable over $L$.
    Let $B \in \GL_d(L)$ be such that all elements of $BTB^{-1}$ are diagonal.
    As a finite-index subgroup of a finitely generated group, the group $T$ is also finitely generated.
    Since $K$ is power-splitting for $G$, for each $A \in BTB^{-1}$ there exists some $n(A)$ such that $A^{n(A)}$ has all entries in $K$.
    Let $T' \le T$ be the subgroup generated by all these $A^{n(A)}$. 
    Since $T$ is finitely generated abelian, so is $T'$. 
    Since $T/T'$ is torsion, it is finite.

    Let $\Gamma \le K^\times$ be the (finitely generated) group generated by all entries of the matrices in $BT'B^{-1}$.
    Transforming back into the original basis and keeping in mind that $T'$ has finite index in $G$, there exist $c_1=1$,~$c_2$, \dots,~$c_m \in L$ such that every entry of every element of $G$ is contained in $\sum_{i=1}^m c_i \Gamma_0$.
    Among these elements $c_1$, \dots,~$c_m$, we pick a maximal $K$-linearly independent set; after reindexing, let this be $c_1=1$, $c_2$, \dots,~$c_n$ with $n \le m$.
    Express each $c_i$ (with $i\in[m]$) as $c_i = \sum_{j=1}^n \gamma_{i,j} c_j$  with $\gamma_{i,j} \in K$.
    Then the entries of elements of $G$ are contained in
    \[
    K \cap \sum_{i=1}^m c_i \Gamma_0 = K \cap \sum_{j=1}^n \Big( \sum_{i=1}^m \gamma_{i,j} \Big) c_j \Gamma_0 \subseteq \sum_{i=1}^m \gamma_{i,1} \Gamma_0,
    \]
    where the last inclusion follows from $K$-linear independence of $c_1$, \dots,~$c_n$.
    Letting $\Gamma' \le K^\times$ be the subgroup generated by $\Gamma$ and all $\gamma_{i,1}$, all entries of all elements of $G$ are contained in $\sumset{m}{\Gamma'_0}$, and thus $G$ is Bézivin.
    Hence, so is $S$.
     \qedhere
  \end{proofenumerate}
\end{proof}

The following is not essential, but allows us to add the Pólya property to the characterization in \cref{t:intro-bezivin}, for completeness.

\begin{lemma} \label{l:epi-is-locpolya}
  If $G \le \GL(V)$ is epimonomial, then there exists a basis of $V$ with respect to which $G$ is locally Pólya.
\end{lemma}

\begin{proof}
  Let $V = V_1 \oplus \cdots \oplus V_r$ be an epimonomial decomposition of $V$, and let $D \trianglelefteq G$ be the diagonal of this decomposition.
  Choose a basis of $V$ by choosing bases on each $V_i$, and identify $G = \GL_d(K)$ using this basis.
  Then the elements of $G$ have a monomial block-structure with respect to the decomposition $V=V_1\oplus\cdots \oplus V_r$.
  The elements of $D$ are diagonal, and each diagonal block is a scalar multiple of the corresponding identity matrix.
  Let $A_1$, \dots,~$A_n$ be a set of representatives for $G/D$.
  If $C \in D$, then $A_i C$ is obtained from $A_i$ by multiplying each block by a scalar.

  Let $H \le G$ be finitely generated and assume without restriction $A_1$, \dots,~$A_n \in H$.
  Then $H \cap D$ is finitely generated, as a finite-index subgroup of $H$.
  Let $\Gamma \le K^\times$ be generated by the nonzero entries of $A_1$, \dots, $A_n$ together with the entries of a set of generators of $H \cap D$.
  Then $H \subseteq \Gamma_0^{d \times d}$.
\end{proof}

\subsection{Maximal separating elements}
For the coarsest weakly epimonomial decomposition of a weakly epimonomial group, it is sometimes useful to know that there is a single element that creates this decomposition.

\begin{lemma} \label{p:maximally-separating-element}
  Let $G \le \GL(V)$ be weakly epimonomial and $V = V_1\oplus\cdots \oplus V_r$ with $V_i$ the joint-eigenspaces of all steady elements of $G$.
  Then there exists a steady $A \in G$ such that the eigenvalues $\lambda_1$, \dots,~$\lambda_r$ of $A$ on $V_1$, \dots,~$V_r$ are pairwise distinct.
\end{lemma}

\begin{proof}
  The spaces $V_i$ are intersections of eigenspaces of steady $A \in G$.
  Thus, there exist $N \ge 0$ and steady $A_1$, \dots,~$A_N$ such that, for every $i \ne i'$ there exists some $A_j$ having distinct eigenvalues on $V_i$ and $V_{i'}$.

  Let $\lambda_{ij}$ be the eigenvalue of $A_{j}$ on $V_i$.
  For all $\vec n = (n_1,\dots,n_N) \in \Z^N$, we consider the product
  \[
    B(\vec n) \coloneqq \prod_{j \in [N]} A_{j}^{n_{j}},
  \]
  which has on $V_i$ the eigenvalue
  \[
    \mu_i(\vec n) \coloneqq \prod_{j \in [N]} \lambda_{ij}^{n_{j}}.
  \]
  We seek $\vec n$ such that $\mu_i(\vec n) / \mu_{i'}(\vec n)$ has infinite order for all $i \ne i'$.

  Note that $\varphi_{ii'}\colon \Z^N \mapsto \algc K^\times$, $\vec n \mapsto \mu_i(\vec n) / \mu_{i'}(\vec n)$ is a group homomorphism.
  For every pair $i \ne i'$ there exists some $j$ with $\lambda_{ij} \ne \lambda_{i'j}$.
  Since $A_j$ is steady, the ratio $\lambda_{ij}/\lambda_{ij'}$ even has infinite order.
  Thus, the image $\im \varphi_{ii'}$ has a torsion-free part.
  Composing $\varphi_{ii'}$ with $\pi \colon \algc K^\times \to \algc K^\times / \mu(\algc K)$, we see that $\ker(\pi \circ \varphi_{ii'})$ has rank strictly smaller than $N$.

  Thus, we get $\bigcup_{i \ne i'} \ker(\pi \circ \varphi_{ii'}) \subsetneq \Z^N$.
  Taking any $\vec n$ outside this union ensures that $\mu_i(\vec n)/\mu_{i'}(\vec n)$ has infinite order for all $i \ne i'$.
  The element $B(\vec n)$ is steady and has pairwise distinct eigenvalues on $V_1$, \dots,~$V_n$.
\end{proof}

\subsection{Lifting to monomial representations}
We now show that epimonomial representations are epimorphic images of monomial representations.

\begin{lemma} \label{l:factor-through-monomial}
  Let $\rho\colon G \to \GL(V)$ be a representation and let $T \le G$ be a finite-index subgroup such that $\rho(T)$ is simultaneously diagonalizable over $K$.
  Then the induced representation $\pi \coloneqq \Ind_T^G(\rho|_T)\colon G \to \GL(W)$ is monomial and there exists a $G$-equivariant epimorphism $(W,\pi) \twoheadrightarrow (V,\rho)$.
\end{lemma}

\begin{proof}
  We check that $\pi$ is a monomial representation.
  Let $V = V_1 \oplus \cdots \oplus V_r$ with $\dim(V_i)=1$ and such that $\rho(T)$ is diagonal with respect to this decomposition.
  Let $a_1$, \dots,~$a_n$ represent the left cosets of $G/T$.
  For each $b \in G$ and $i \in [n]$, let $c_{i,b} \in T$ and $f(i,b) \in [n]$ be such that
  \[
  ba_i = a_{f(i,b)} c_{i,b}.
  \]

  Recall that the induced representation is defined on the $nr$-dimensional space
  \[
  W = \bigoplus_{i=1}^n a_i V = \bigoplus_{i=1}^n \bigoplus_{j=1}^r a_i V_j.
  \]
  For each $b \in G$ the endomorphism $\pi(b) \in \GL(W)$ is defined by
  \[
    \pi(b)(a_i v) = a_{f(i,b)}\big(\rho(c_{i,b})(v)\big) \qquad\text{for $v \in V$.}
  \]
  Since $\rho(c_{i,b})$ is diagonal, there exists $\lambda_{i,b} \in K$ for which $\rho(c_{i,b})(v_j)= \lambda_{i,b} v_j$ for $v_j \in V_j$, so in fact $\pi(b)(a_i v_j)= a_{f(i,b)} (\lambda_{i,b} v_j)$.
  Hence, the representation $\pi$ is monomial.

  Define a linear epimorphism
  \[
  \psi\colon W \twoheadrightarrow V, \quad ( a_i v_i)_{i \in[n]} \mapsto \sum_{i=1}^n \rho(a_i) v_i \qquad(\text{for }v_i \in V).
  \]
  Then $\psi$ is $G$-equivariant, because 
  \begin{align*}
  \psi\big( \pi(b)  (a_i v_i)_{i \in [n]} \big) &=  \psi\Big( \big(a_{f(i,b)} (\rho(c_{i,b})(v_i)) \big)_{i \in [n]} \Big) = \sum_{i=1}^n \rho(a_{f(i,b)} c_{i,b})(v_i) =  \sum_{i=1}^n \rho(b a_i)(v_i) \\
  &= \rho(b)\Big( \sum_{i=1}^n \rho(a_i)v_i \Big) = \rho(b)\Big( \psi\big( (a_iv_i)_{i \in [n]} \big)\Big). \qedhere
  \end{align*}
\end{proof}

\subsection{Not necessarily finitely generated groups}
\Cref{exm:locbez-not-monomial} shows that not every locally Bézivin group has a monomial representation.
In \ref{exm:roots-of-unity-over-R} of \cref{exm:locbez-not-monomial}, the issue is that $\R$ is power-splitting for $\mu(\C) \hookrightarrow \GL_2(\R)$ but not uniformly so.
In characteristic $0$, we can still get the following two results for not necessarily finitely generated groups, as long as we assume \emph{uniform} power-splitting.
First, we have the following easy observation.

\begin{corollary} \label{p:char0-epimon}
  Suppose $\characteristic{K}=0$ and $\card{\mu(K)} < \infty$.
  For $G \le \GL(V)$ the following statements are equivalent.
  \begin{equivenumerate}
    \item \label{c0-epimon:uniform} The group $G$ satisfies the equivalent conditions of \cref{t:locally-bezivin-equivalent} and $K$ is uniformly power splitting for $G$.
    \item \label{c0-epimon:epimon} The group $G$ is epimonomial.
  \end{equivenumerate}
\end{corollary}

\begin{proof}
  \begin{proofenumerate}
  \item[\ref{c0-epimon:uniform}$\,\Rightarrow\,$\ref{c0-epimon:epimon}:] By assumption $G$ is a weakly epimonomial.
  Let $D \trianglelefteq G$ be a diagonal with respect to a weakly epimonomial decomposition of $V$.
  Then, using \cref{p:wepi-torsion}, the quotient $G/D$ is a linear torsion group, and we show that it has finite exponent.
  By assumption, there exists $N \ge 1$ such that every $A^N$, with $A \in G$, has its eigenvalues in $K$.
  Let $M \coloneqq \card{\mu(K)} N$.
  If $A \in G$, then $A^N$ is diagonalizable over $K$ (\cref{l:bezivin-diagonalize}).
  Raising $A^N$ to the power $\card{\mu(K)}$ ensures $\lambda/\mu$ cannot be a nontrivial root of unity for two distinct eigenvalues $\lambda$, $\mu$ of $A$.
  Therefore, the power $A^M$ is steady, and hence $A^M \in D$.
  Because $\characteristic{K}=0$, every linear torsion group with finite exponent is finite (\cref{t:burnside-schur}).

  \item[\ref{c0-epimon:epimon}$\,\Rightarrow\,$\ref{c0-epimon:uniform}:]
  Let $D$ be a diagonal corresponding to an epimonomial decomposition of $V$.
  Since $G/D$ is finite, there exists $N \ge 1$ such that $A^N \in D$ for all $A \in G$.
  Let $A \in G$ and let $\lambda \in \algc{K}$ be an eigenvalue of $A$.
  Then $\lambda^N$ is an eigenvalue of $A^N$.
  But the eigenvalues of $A^N$ are in $K$ by definition of $D$, hence $\lambda^N \in K$. \qedhere
  \end{proofenumerate}
\end{proof}

Achieving an analogous result when $\mu(K)$ may be infinite is harder.
We need a classical theorem of Jordan.

\begin{theorem}[Jordan] \label{t:jordan}
  Let $K$ be a field of characteristic $0$.
  Then there is a function $f\colon \Z_{\ge 0} \to \Z_{\ge 0}$ such that every finite subgroup $G \le \GL_d(K)$ has an abelian subgroup of index at most $f(d)$.
\end{theorem}

\begin{theorem} \label{t:lbz0-diag}
  Suppose $\characteristic{K}=0$ and let $G \le \GL(V)$.
  For $G \le \GL(V)$ the following statements are equivalent.
  \begin{equivenumerate}
    \item \label{lbz0-diag:uniform} The group $G$ satisfies the equivalent conditions of \cref{t:locally-bezivin-equivalent} and $K$ is uniformly power splitting for $G$.
    \item \label{lbz0-diag:epimon} The group $G$ is virtually simultaneously diagonalizable over $K$.
  \end{equivenumerate}
\end{theorem}

\begin{proof}
  \begin{proofenumerate}
  \item[\ref{lbz0-diag:uniform}$\,\Rightarrow\,$\ref{lbz0-diag:epimon}:]
  By \cref{t:locally-bezivin-equivalent} there exists a weakly epimonomial decomposition $V = V_1 \oplus \cdots \oplus V_r$ of $V$.
  Let $d_i\coloneqq \dim V_i$.
  Let $D$ be the diagonal of $G$.
  Then $G/D$ is a linear torsion group that embeds into $\GL(W)$ with $W = \End(V_1) \oplus \cdots \oplus \End(V_r)$ (see \cref{p:wepi-torsion}).
  Let $G' \le G$ be the subgroup of all $B \in G$ with $BV_i=V_i$ for all $i \in [r]$.
  Then $(G:G') \le r!$ and it suffices to prove that $G'$ is virtually simultaneously diagonalizable.
  Note that $D \subseteq Z(G')$ by choice of $G'$.

  Let $\bP_{\text{fin}}(G')$ denote the set of all finite subsets of $G$. 
  For every $S \in \bP_{\text{fin}}(G')$ let $G_S=\langle S \rangle$.
  Then $G' = \bigcup_{S \in \bP_{\text{fin}(G')}} G_S$.
  The decomposition of $V$ is weakly epimonomial for each $G_S$ as well, and $D_S \coloneqq D \cap G_S$ is the corresponding diagonal.
  The map $G \to \GL(W)$ with kernel $D$ therefore gives rise to an embedding $G_S/D_S \hookrightarrow \GL(W)$, so that each $G_S/D_S$ embeds into the fixed $\GL(W)$.
  As a finitely generated linear torsion group, each $G_S/D_S$ is finite by Burnside--Schur (\cref{t:burnside-schur}).
  Now Jordan's Theorem (\cref{t:jordan}) implies that there exist a constant $C$ and subgroups $D_S \le H_S \le G_S$ such that $H_S/D_S$ is abelian and $(G_S:H_S) \le C$ for all $S \in \bP_{\text{fin}}(G')$.

  With a similar compactness argument as in a proof of Jordan--Schur \cite[Chapter 2.1]{tao14} we can choose the groups $H_S$ so that $H_T = H_S \cap G_T$ whenever $T \subseteq S$:
  Let $\mathfrak X_S$ denote the set of all subgroups $H_S$ with $D_S \le H_S \le G_S$ such that $H_S/D_S$ is abelian and $(G_S:H_S)\le C$.
  We endow $\mathfrak X_S$ with the discrete topology.
  Since each $\mathfrak X_S$ is finite, hence compact, the product space $\mathfrak X=\prod_{S \in \bP_{\text{fin}}(G')} \mathfrak X_S$ is compact by Tychonoff's Theorem.
  For each $S \in \bP_{\text{fin}}(G')$ the set 
  \[
  \mathfrak C(S) \coloneqq \big\{\, (H_T)_{T \in \bP_{\text{fin}}(G')} \in \mathfrak X : H_T = G_T \cap H_S \text{ for $T \subseteq S$}\,\big\}
  \]
  is closed (because it is defined by only finitely many conditions).
  Each $\mathfrak C(S)$ is nonempty, because we may take $H_S \in G_S$ with $H_S/D_S$ abelian and $(G_S:H_S) \le C$ and then $H_T \coloneqq H_S \cap G_T$ for $T \subseteq S$ will also satisfy $(G_T:H_T) \le C$.
  Since $\mathfrak C(S_1) \cap \cdots \cap \mathfrak C(S_n) \supseteq \mathfrak C(S)$ for $S = S_1 \cup \cdots \cup S_n$, any finite intersection of such sets is also nonempty.
  The finite intersection property for compact sets implies $\bigcap_{S \in \bP_{\text{fin}}(G')} \mathfrak C(S) \ne \emptyset$.
  Taking any element $(H_S)_{S \in \mathbb P_{\text{fin}}(G')}$ in this intersection, we have $H_T = H_S \cap G_T$ for \emph{all} $S$,~$T \in \mathbb P_{\text{fin}}(G')$ with $T \subseteq S$.

  Let $H = \bigcup_{S \in \bP_{\text{fin}}(G')} H_S$.
  Then $(G':H) \le C$: suppose $A_1$, \dots,~$A_n \in G'$ are in distinct cosets modulo $H$. 
  With $S=\{A_1,\dots,A_n\}$ we have $A_1$, \dots,~$A_n \in G_S$.
  Then $A_1$, \dots,~$A_n$ must represent distinct cosets of $H_S$ in $G_S$, hence $n \le C$.
  Furthermore, by construction $H/D$ is abelian with $D \subseteq Z(G')$.
  In particular, the group $H$ is nilpotent of nilpotency index at most $2$.
  By definition of $D$, each element of $D$, when restricted to $V_i$, consists of a scalar multiple of the identity on $V_i$.
  Let $[A,B]\coloneqq ABA^{-1}B^{-1}$ be the multiplicative commutator.
  For all $A$,~$B \in H$ it holds that $\det([A,B])=1$ and $[A,B] \in D$.
  Thus, we have $[A,B]|_{V_i} = \zeta_i \id$ for some $\zeta_i \in \mu_{d_i}(K)$. 
  By nilpotency class at most $2$, we have the identity $[XY,Z]=[X,Z][Y,Z]$ for all $X$, $Y$,~$Z \in H$. Therefore $[A^{d_i}, B]|_{V_i} = ([A,B]|_{V_i})^{d_i}=\zeta_i^{d_i}\id|_{V_i} = \id|_{V_i}$, so $[A^{d_i},B]=\id$.
  
  By \cref{l:bezivin-diagonalize}, all elements of $H$ are diagonalizable over $\algc{K}$.
  Since $K$ is uniformly power-splitting for $G$, there exists an $N \ge 1$ such that $A^N$ is diagonalizable over $K$ for each $A \in H$.
  We may choose $N$ so that also $\lcm(d_1,\dots,d_r) \mid N$.
  Then $[A^N,B]=[A,B]^N=\id$ shows that $B$ and $A^N$ commute for every $B \in H$.
  Hence, each $B \in H$ leaves every eigenspace of $A^N$ invariant.
  Since $H \le G'$, also each $V_i$ is $H$-invariant.
  We can therefore refine $V_i= W_{i1} \oplus \cdots \oplus W_{is_{i}}$ into a decomposition of joint-eigenspaces of all $A^{N}$, where $A \in H$.
  Then each $W_{ij}$ is $H$-invariant; in particular the refined decomposition $V=\bigoplus_{i,j} W_{ij}$ is a weakly epimonomial decomposition for $H$.

  Let $T$ be the diagonal of $H$ with respect to the refined decomposition.
  Then $H/T$ is a linear torsion group by \cref{p:wepi-torsion}, and it has exponent dividing $N$.
  Since $\characteristic{K}=0$, this already implies that $H/T$ is finite (\cref{t:burnside-schur}).
  Since $G/G'$, $G'/H$, and $H/T$ are finite, we conclude that $G/T$ is finite.

  \item[\ref{lbz0-diag:epimon}$\,\Rightarrow\,$\ref{lbz0-diag:uniform}:]
  Identifying $\GL(V) \cong \GL_d(K)$ along a suitable basis, we find a finite-index subgroup $T \le G$ of diagonal matrices.
  Hence, property \ref{locequiv:locdiag} of \cref{t:locally-bezivin-equivalent} holds.
  Taking $N = (G:T)$ we see that $\lambda^N$ is in $K$ for every eigenvalue of every element of $G$. \qedhere
  \end{proofenumerate}
\end{proof}

Now it is easy to see that if we assume $G$ to be not just locally Bézivin, but Bézivin, it is actually necessarily finitely generated.

\begin{corollary} \label{cor:bez-0-is-fg}
  Suppose $\characteristic{K}=0$ and let $G \le \GL(V)$.
  If $G$ is Bézivin, then $G$ is finitely generated.
\end{corollary}

\begin{proof}
  Since the claimed property for $G$ is purely group theoretic, and the Bézivin property is preserved under scalar extension, we may without restriction assume $K=\algc{K}$, so $K$ is uniformly power-splitting for $G$ for trivial reasons.
  
  By \cref{t:lbz0-diag} there exists a finite-index subgroup $T \le G$ that is simultaneously diagonalizable.
  Identify $\GL(V)=\GL_d(K)$ using a basis in which $T$ is diagonal.
  Let $\Gamma \le K^\times$ be finitely generated and let $M \ge 0$ be such that $G \subseteq (\sumset{M}{\Gamma_0})^{d \times d}$.
  Let $\Lambda \le K^\times$ be the subgroup consisting of all diagonal entries of elements of $T$.
  Then $\Lambda \subseteq \sumset{M}{\Gamma_0}$.
  \Cref{l:subgroup-bezivin} implies that $\Lambda$ is finitely generated.
  Hence, the subgroup $T$ is finitely generated, and because $G/T$ is finite, also $G$ is finitely generated.
\end{proof}

\subsection{Proofs of \cref{t:main-loc-bezivin,cor:locbezivin-fg,cor:locbezivin-char0}}

With the results we proved, the main theorems on locally Bézivin representations now follow easily.
\begin{itemize}
  \item For the proof of \cref{t:main-loc-bezivin}, note that \ref{locbez:polya}$\,\Rightarrow\,$\ref{locbez:bezivin} is trivial, and \ref{locbez:bezivin}$\,\Leftrightarrow\,$\ref{locbez:weakbezivin} follows from \cref{p:equiv-field}.
  Now \ref{locbez:bezivin}$\,\Rightarrow\,$\ref{locbez:locepimon} follows from \cref{t:locally-bezivin-equivalent}.
  For \ref{locbez:locepimon}$\,\Rightarrow\,$\ref{locbez:finsub} first note that weak epimonomiality passes from $\rho(S)$ to the group it generates by \cref{l:lift-wepi-to-group}.
  Then the claim again follows from \cref{t:locally-bezivin-equivalent}.
  The final implication \ref{locbez:finsub}$\,\Rightarrow\,$\ref{locbez:polya} follows from \cref{l:epi-is-locpolya}.

  \item We again work with $G=\langle \rho(S)\rangle$.
  In \cref{cor:locbezivin-fg} the equivalence of \ref{lbzfg:bez}, \ref{lbzfg:epimon}, and \ref{lbzfg:diag} follows from \cref{t:locally-bezivin-equivalent}.
  The implication \ref{lbzfg:polya}$\,\Rightarrow\,$\ref{lbzfg:bez} is trivial, and \ref{lbzfg:epimon}$\,\Rightarrow\,$\ref{lbzfg:polya} holds by \cref{l:epi-is-locpolya}.
  The equivalence of these conditions with \ref{lbzfg:weakbezivin} follows from \cref{p:equiv-field}.

  Now \ref{lbzfg:diag}$\,\Rightarrow\,$\ref{lbzfg:monomial} follows from \cref{l:factor-through-monomial}.
  For the converse, \ref{lbzfg:monomial}$\,\Rightarrow\,$\ref{lbzfg:diag}, let $(\widehat V, \widehat \rho)$ be a monomial representation that maps epimorphically onto $(V,\rho)$, say via an $S$-equivariant $\pi \colon \widehat V \to V$.
  Because of $\pi \widehat \rho(s) = \rho(s) \pi$ for all $s \in S$, we get a semigroup homomorphism $\psi\colon \widehat\rho(S) \to \rho(S)$, which extends to a group homomorphism $\psi \colon \widehat G \to G$ with $\widehat G = \langle \widehat\rho(S) \rangle$ and $G = \langle \rho(S)\rangle$.
  If $l=\dim(\widehat V)$, then there is a group homomorphism $\widehat G \to \mathfrak S_l$ to the symmetric group, mapping a monomial matrix to its underlying permutation matrix.
  Its kernel $D$ has finite index in $\widehat G$ and is simultaneously diagonal.
  Choose a basis $e_1$, \dots,~$e_l$ of $\widehat V$ on which $D$ is diagonal, and take as basis of $V$ a subset of $\{\, \pi(e_i) : i \in l \,\}$.
  Then $\psi(D)$ is a diagonal finite-index subgroup of $G$.

  \item In \cref{cor:locbezivin-char0} the equivalence \ref{lbz0:locbez}$\,\Leftrightarrow\,$\ref{lbz0:diag} holds by \cref{t:lbz0-diag}.
  The implication \ref{lbz0:diag}$\,\Rightarrow\,$\ref{lbz0:monomial} follows again from \cref{l:factor-through-monomial}, and the converse follows in the same way as we just argued for \cref{cor:locbezivin-fg}.
\end{itemize}

We record that \ref{lbzfg:monomial} of \cref{cor:locbezivin-fg} establishes the following.
\begin{corollary} \label{c:lepi-preserved}
  Let $\widehat \rho\colon S \to \GL(\widehat V)$ and $\rho\colon S \to \GL(V)$ be semigroup representations.
  If there exists an $S$-equivariant epimorphism $\pi \colon \widehat \rho \to \rho$ and $\widehat \rho$ satisfies the equivalent conditions of \cref{t:locally-bezivin-equivalent}, then so does $\rho$.
\end{corollary}

\section{Characterization of Locally Finitely Generated Spectrum}
\label{sec:locfg}

Again let $K$ be a field and $V$ a finite-dimensional vector space.
To establish the characterization of linear groups with locally finitely generated spectrum, we first show that virtually solvable groups have this property. We use the following observation.
\begin{lemma} \label{l:bounded-roots}
  Let $\Gamma' \le \algc K^\times$ be finitely generated and $N \ge 1$.
  Then
  \[
    \Gamma \coloneqq \{\, \gamma \in \algc{K}^\times : \gamma^N \in \Gamma' \,\}
  \]
  is finitely generated as well.
\end{lemma}

\begin{proof}
  There is a homomorphism of abelian groups $\varphi\colon \Gamma \to \Gamma',\ \gamma \mapsto \gamma^N$.
  Since $\Gamma'$ is finitely generated and $\ker(\varphi)=\mu_N(K)$ is finite, also $\Gamma$ is finitely generated.
\end{proof}

\begin{lemma} \label{l:virt-solvable-locfg}
  If $G$ is a virtually solvable group, then every representation $\rho\colon G \to \GL(V)$ has locally finitely generated spectrum.
\end{lemma}

\begin{proof}
  We may assume that $K$ is algebraically closed and replace $G$ by $\rho(G)$.
  Since $G$ is virtually solvable, a theorem of Mal'cev \cite[Theorem 3.6]{wehrfritz73} implies that there exists a finite-index subgroup $U \le G$ that is simultaneously triangularizable.\footnote{Alternatively, consider the Zariski closure $\widetilde{G}$ and apply the Lie--Kolchin Theorem \cite[Theorem 10.5]{borel91} to see that the connected component of the identity is triangularizable.}

  Let $H$ be a finitely generated subgroup of $G$.
  Then $H \cap U$ has finite index in $H$, and is thus also finitely generated.
  Let $\Gamma' \le \algc{K}^\times$ be generated by the eigenvalues of the elements of $H \cap U$.
  Because $H \cap U$ is simultaneously triangularizable, the group $\Gamma'$ is finitely generated.
  Let $\Gamma \le \algc{K}^\times$ be generated by all eigenvalues of $H$ and let $N \coloneqq \exp(G/U)$.
  Then $\Gamma^N = \{\, \gamma^N : \gamma \in \Gamma \,\} \le \Gamma'$ and hence $\Gamma$ is finitely generated by \cref{l:bounded-roots}.
\end{proof}

In the absolutely irreducible case locally Bézivin and locally finitely generated spectrum (plus power-splitting) coincide.
Before showing this, we need one more lemma.

\begin{lemma} \label{l:bezivin-in-extension}
  Let $L/K$ be a field extension.
  If $\Lambda \le L^\times$ is a finitely generated subgroup such that $\Lambda/(\Lambda \cap K^\times)$ is torsion and $M \ge 0$, then there exists a finitely generated $\Gamma \le K^\times$ such that $K \cap \sumset{M}{\Lambda} \subseteq \sumset{M}{\Gamma}$.
\end{lemma}

\begin{proof}
  Let $\Gamma' \coloneqq \Lambda \cap K^\times$.
  By assumption, the quotient $\Lambda / \Gamma'$ is a finitely generated abelian torsion group, and hence finite.
  Let $A \subseteq \Lambda$ be a set of representatives for $\Lambda/\Gamma'$ with $1 \in A$.
  Let $A' \subseteq A$ be a $K$-basis of the finite-dimensional vector space $\lspan_K A$.
  We may without restriction assume $1 \in A'$.
  Expressing each $\alpha \in A$ as $\alpha = \sum_{\alpha' \in A'} c_{\alpha,\alpha'} \alpha'$ with $c_{\alpha,\alpha'} \in K$, we have
  \[
     \Lambda \subseteq \bigcup_{\alpha \in A} \alpha \Gamma' \subseteq \bigcup_{\alpha \in A} \sum_{\alpha' \in A'} c_{\alpha,\alpha'} \alpha' \Gamma' \subseteq \sum_{\alpha' \in A'} \alpha' \Gamma,
  \]
  where $\Gamma$ is the group generated by $\Gamma'$ and the finitely many coefficients $c_{\alpha,\alpha'}$ that are nonzero.
  Thus, 
  \[
   \sumset{M}{\Lambda} \subseteq \sum_{\alpha' \in A'} \alpha' ( \sumset{M}{\Gamma} ).
  \]
  Linear independence of $A'$ implies $\sumset{M}{\Lambda} \cap K \subseteq \sumset{M}{\Gamma}$.
\end{proof}

\begin{proposition} \label{p:absirred-fg-bezivin}
  Let $S \subseteq \GL(V)$ be an \emph{absolutely irreducible} semigroup.
  Then the following statements are equivalent.
  \begin{equivenumerate}
  \item \label{absirred:fg} The semigroup $S$ has locally finitely generated spectrum and $K$ is a power-splitting field for $S$.
  \item \label{absirred:trbez} For every finitely generated subsemigroup $S' \subseteq S$, the set $\Tr(S')$ is Bézivin.
  \item \label{absirred:locbez} The semigroup $S$ satisfies the equivalent conditions of \cref{t:locally-bezivin-equivalent}.
  \end{equivenumerate}
\end{proposition}

\begin{proof}
  \begin{proofenumerate}
  \item[\ref{absirred:fg}$\,\Rightarrow\,$\ref{absirred:trbez}:]
  By assumption $S'$ has finitely generated spectrum.
  Let $\Lambda \le \algc{K}^\times$ be finitely generated such that all eigenvalues of elements of $S'$ are contained in $\Lambda$.
  Then $\Tr(S') \subseteq \sumset{d}{\Lambda} \cap K$.
  Since $K$ is power-splitting for $S$, the group $\Lambda / (\Lambda \cap K^\times)$ is torsion.
  \Cref{l:bezivin-in-extension} implies that $\Tr(S')$ is Bézivin.

  \item[\ref{absirred:trbez}$\,\Rightarrow\,$\ref{absirred:locbez}:]
  We verify that $S$ is locally Bézivin.
  Without restriction $\GL(V)=\GL_d(K)$.
  Because $S$ is absolutely irreducible, there exists a $K$-basis $A_1$, \dots,~$A_{d^2} \in S$ of $K^{d \times d}$
  (see for instance \cite[Theorems 3.32 and 3.43]{curtis-reiner90} for this classical result by Burnside).
  The map 
  \[ 
  K^{d \times d} \to K^{d^2},\quad X \mapsto \big(\Tr(XA_1), \dots, \Tr(XA_{d^2})\big)
  \]
  is a vector space isomorphism.
  In particular there exist $c_{ijk} \in K$ such that $x_{ij} = \sum_{k=1}^{d^2} c_{ijk} \Tr(X A_k)$ for all $X=(x_{ij}) \in K^{d \times d}$.
  Because of \ref{absirred:trbez}, the sets $\sum_{k=1}^{d^2} c_{ijk} \Tr(X A_k)$ are Bézivin.

  \item[\ref{absirred:locbez}$\,\Rightarrow\,$\ref{absirred:fg}:] 
  By \ref{locequiv:locdiag} of \cref{t:locally-bezivin-equivalent}, the group $G\coloneqq \langle S \rangle$ is in particular virtually abelian.
  Then \cref{l:virt-solvable-locfg} implies that $G$ has locally finitely generated spectrum.
  \Cref{p:equiv-field} implies that $K$ is power-splitting for $G$.
  Both properties then also hold for $S$. \qedhere
  \end{proofenumerate}
\end{proof}

In dealing with the locally finitely generated spectrum case, we will need to pass to a finite extension field $L$ of $K$.
It is not trivial, but still true, that the resulting weakly epimonomial $L$-representations are weakly epimonomial when viewed as $K$-representations, \emph{as long as $K$ is power-splitting for $S$}.
This is the following lemma.

\begin{lemma} \label{l:locepimon-descent}
  Let $L/K$ be a finite field extension and let $W$ be a finite-dimensional $L$-vector space.
  If $S \subseteq \GL({}_L W)$ is a weakly epimonomial subsemigroup and $K$ is power-splitting for $S$, then also $S \subseteq \GL({}_K W)$ is weakly epimonomial.
\end{lemma}

\begin{proof}
  For clarity, we say $A \in \GL({}_L W)$ is \defit{$L$-steady} when $A$ is steady as an $L$-vector space endomorphism, and \defit{$K$-steady} when $A$ is considered as a $K$-vector space endomorphism.
  If $A$ is $K$-steady, then it is also $L$-steady: indeed, all eigenvalues of $A$ are in $L$, because they are in $K$, and if $U \subseteq W$ is an $A^n$-invariant $L$-subspace of $W$, then $U$ is also an $A^n$-invariant $K$-subspace of $W$, so that $AU \subseteq U$ by $K$-steadyness of $A$.

  Let ${}_L W = W_1 \oplus \cdots \oplus W_r$ be a weakly epimonomial decomposition of $W$.
  So each $W_i$ is contained in an $L$-eigenspace of every $L$-steady element of $S \subseteq \GL({}_L W) \subseteq \GL({}_K W)$. 
  Property~\ref{d-wm:permutation} of \cref{d:weakly-epimonomial} does not depend on the base field, and~\ref{d-wm:splitting} holds by assumption.
  We have to verify~\ref{d-wm:intersection} when this decomposition is considered as one into $K$-vector spaces.
  However, if $A \in S$ is $K$-steady, then $A$ is $L$-steady, and hence $W_i$ is contained in an $L$-eigenspace of $A$, which is then also contained in a $K$-eigenspace.
\end{proof}

We prove our main results on representations with locally finitely generated spectrum.

\begin{proof}[{Proof of \cref{t:trilocmonomial}}]
  \begin{proofenumerate}
  \item[\ref{triloc:spec}$\,\Rightarrow\,$\ref{triloc:repr}:]
  Passing to a suitable finite extension field $L$ of $K$ we can find in $\widehat V \coloneqq L \otimes_K V$ a chain of $S$-invariant $L$-vector spaces
  \[
  0 = W_0 \subsetneq W_1 \subsetneq \cdots \subsetneq W_r = \widehat V,
  \]
  such that each $W_i/W_{i-1}$ is absolutely irreducible.
  The resulting $L$-representation $\overline\rho \colon S \to \GL({}_L \widehat V)$ is block-triangular with respect to this chain.
  Since $\overline\rho$ has locally finitely generated spectrum as an $L$-representation (as the eigenvalues do not change under scalar extension), the diagonal blocks $\overline\rho_{ii}$ also have locally finitely generated spectrum.
  By \cref{p:absirred-fg-bezivin}, each $\overline\rho_{ii}$ is weakly epimonomial.

  We now consider $\overline \rho$ as a $K$-representation and for clarity we change the name to $\widehat \rho\colon S \to \GL({}_K \widehat V)$.
  By \cref{l:locepimon-descent} the diagonals $\widehat\rho_{ii}$ (with respect to the block structure from before) are still weakly epimonomial, hence locally epimonomial by \cref{t:locally-bezivin-equivalent}.
  
  Let $b_1=1$, $b_2$, \dots,~$b_m \in L$ be a $K$-basis for $L$, so that $\widehat V \cong \bigoplus_{i=1}^m b_i V$, as $K$-vector space.
  Then $\overline\rho(s)(b_1 v_1+ \cdots + b_m v_m) = b_1 \rho(s)(v_1) + \cdots + b_m \rho(s)(v_m)$ for all $s \in S$ and $v_i \in V$.
  The inclusion $\iota\colon V \hookrightarrow \widehat V, v \mapsto 1\cdot v$ is split by the epimorphism $\pi\colon \widehat V \to V$ that maps $b_1$ to $1$ and $b_i$ to $0$ for $i \ne 1$.
  Since $\pi$ is $S$-equivariant, we get an epimorphism $(\widehat V, \widehat \rho) \to (V,\rho)$.

  In $V$, we obtain a chain of $S$-invariant subspaces
  \[
  0 = \pi(W_0) \subseteq \pi(W_1) \subseteq \cdots \subseteq \pi(W_r)=V.
  \]
  This chain induces a block-triangular structure on $\rho$.
  The diagonal block of $\rho$ on $\pi(W_i)/\pi(W_{i-1})$ is the image of the diagonal block $\widehat \rho_{ii}$ of $\widehat \rho$ on $W_i/W_{i-1}$, and it is therefore locally epimonomial (\cref{c:lepi-preserved}).

  \item[\ref{triloc:repr}$\,\Rightarrow\,$\ref{triloc:spec}:]
    Let $S' \subseteq S$ be a finitely generated subsemigroup, and let $G = \langle \rho(S') \rangle$.
    Then $G$ has epimonomial diagonal blocks.
    In a suitable basis, therefore $G$ has a finite-index subgroup $U$ consisting of upper-triangular matrices.
    Since $U$ is solvable, the group $G$ is virtually solvable.
    From \cref{l:virt-solvable-locfg} it follows that $G$ has locally finitely generated spectrum.
    Because the diagonal blocks of $G$ are epimonomial, also $K$ is power-splitting for $G$.
    \qedhere
  \end{proofenumerate}
\end{proof}

\begin{proof}[Proof of \cref{cor:trilocmonomial-lift}]
  By \cref{t:trilocmonomial} the representation $S \hookrightarrow \GL(V)$ can be assumed to have a block-triangular structure with locally epimonomial diagonal blocks.
  Let $G = \langle S \rangle$ and consider $\rho\colon G \hookrightarrow \GL(V)$.
  The diagonal blocks $\rho_{ii}$ are virtually simultaneously diagonalizable by \cref{cor:locbezivin-fg,cor:locbezivin-char0}.
  Let $D_{ii} \triangleleft \rho_{ii}(G)$ be a finite-index simultaneously diagonalizable subgroup for each $i$, and consider
  \[
  N = \begin{pmatrix} 
    D_{11} & \rho_{12}(G) &  \dots & \rho_{1n}(G) \\
    0 & D_{22} &  \dots & \rho_{2n}(G) \\
    \vdots & \ddots & \ddots & \vdots \\
    0 & \dots & 0 & D_{rr}
  \end{pmatrix}
  \]
  Then $N \le \rho(G)$ is a subgroup of finite index.
  The induced representation $\Ind^G_N(\tau|_N)$ is block-triangular with monomial diagonal blocks and has an epimorphism to $\tau$ by \cref{l:factor-through-monomial}.
\end{proof}

\section{Application to Automata} \label{sec:automata}

Let $K$ be a field and $X$ an alphabet (a nonempty finite set).
We set our notation for weighted finite automata (WFA), but will emphasize the (equivalent) viewpoint of \emph{linear representations} (of WFA), because it is easier to work with rigorously.
We refer to the books \cite{sakarovitch09,berstel-reutenauer11} for more background.
In this context, a \defit{linear representation} is a triple $\cA=(u,\mu,v)$ with a vector $u \in K^{1 \times d}$ of \defit{initial weights}, a vector $v \in K^{d \times 1}$ of \defit{terminal weights} and a semigroup homomorphism $\mu \colon X^* \to K^{d \times d}$ (describing the \defit{transitions}).
Thus, to each letter $x \in X$ we assign a matrix $\mu(x)$.
We say that $\cA$ is \defit{invertible} if $\mu(x)$ is an invertible matrix for every $x \in X$ (the inverse need not be contained in $\mu(X^*)$).

To $\cA$ we assign a function $\behav{\cA}\colon X^* \to K$ by defining $\behav{\cA}(w) \coloneqq u \mu(w) v$.
The function $\behav{\cA}$ is the \defit{behavior} of $\cA$ and is typically represented as a noncommutative (formal) series
\[
\behav{\cA} = \sum_{w \in X^*} \behav{\cA}(w) \, w \in K\llangle X \rrangle.
\]
By the fundamental Kleene--Schützenberger Theorem the series representable in this way are precisely the noncommutative rational series \cite[Chapter 1.7]{berstel-reutenauer11}.
Any given rational series can be represented by many linear representations. 
However, the minimal linear representations, that is, those of minimal dimension, are unique up to conjugation by an invertible matrix \cite[Theorem 2.4]{berstel-reutenauer11}.
Two linear representations $\cA$, $\cA'$ are \defit{equivalent} if $\behav{\cA} =\behav{\cA'}$.

From a computational perspective, we may think of a linear representation $\cA$ of dimension $d$ as a \defit{weighted finite automaton (WFA)} \cite[Chapter 1.6]{berstel-reutenauer11} on the set of states $Q=[d]$ as follows: A state $p \in Q$ is assigned an initial weight $u_p$ and terminal weight $v_p$ (where the subscripts denote the corresponding coordinates of the vectors).
For any two states $p$,~$q \in Q$ (not necessarily distinct) and every letter $x \in X$ we have a directed edge (transition) $p \to q$ with weight $\mu(x)_{pq}$ and label $x$ if $\mu(x)_{pq} \ne 0$; and no edge if $\mu(x)_{pq}=0$.

Given a word $w=x_1\cdots x_l \in X^*$, a \defit{run} for $w$ is then a finite sequence $(p_0,x_1,p_1)$, $(p_1,x_2,p_2)$, \dots, $(p_{l-1},x_l,p_l)$ with $p_i \in Q$.
The \defit{weight} (or \defit{output}) of a run is the product
\[ 
u_{p_0}\mu(x_0)_{p_0p_1} \mu(x_1)_{p_1p_2} \cdots \mu(x_l)_{p_{l-1}p_l} v_{p_l},
\]
and the run is \defit{successful} if its output is nonzero.
The \defit{output} of the word $w$ is then the sum of the output of all runs for $w$.

Thus, a WFA is represented as a directed graph (with loops), where each edge has a label (a letter) and a weight (an element of $K$).
The initial and terminal weights are represented by weighted edges that are only attached to the digraph by their sinks, respectively, sources.
A successful run for $w \in X^*$ is then a walk in the digraph, from an initial state to a terminal state, in which the edges are labelled by the letters of $w$.

From the way matrix multiplication is defined, it is easily seen that the output for $w$ is just $\behav{A}(w)$.
Indeed, every successful run corresponds to a nonzero summand when we fully expand the matrix product $u \mu(x_1) \cdots \mu(x_l) v$.
We think of the WFA as \emph{computing} the function $\behav{A} \colon X^* \to K$, that is, inputting a word $w$, the value $\behav{A}(w)$ is computed.

We say that $\cA$ is \defit{trim} if every state lies on a successful run.

The correspondence between WFA and linear representations is bijective up to a relabelling of the states (that is, conjugation by a permutation matrix on the linear representation side).
Since such a relabelling will not matter for any properties we consider, we shall tacitly identify WFA and their linear representations in the following.

\begin{definition} \label{d:wfa-ambiguity}
  A WFA $\cA$ is
  \begin{enumerate}
  \item \defit{deterministic} \textup(or \defit{sequential}\textup) if there is at most one initial state \textup(that is, at most one $p$ with $u_p \ne 0$\textup), and for every $p \in Q$ and $x \in X$ there exists at most one outgoing transition $p\to q$ labelled by $x$ \textup(that is, the $p$-th row of $\mu(x)$ has at most one nonzero entry\textup),
  \item \defit{$M$-ambiguous} \textup($M \ge 0$\textup) if every word has \emph{at most} $M$ successful runs,
  \item \defit{unambiguous} if it is $1$-ambiguous, 
  \item \defit{finitely ambiguous} if it is $M$-ambiguous for some $M$,
  \item \defit{polynomially ambiguous} if there is a constant $C$ and an exponent $k \ge 0$ such that every word has at most $C \length{w}^k$ successful runs, where $\length{w}$ is the length of the word.
  \end{enumerate}
\end{definition}

For a finitely ambiguous $\cA$, the \defit{ambiguity degree} is the minimal $M$ for which $\cA$ is $M$-ambiguous. 
Finitely ambiguous, respectively, polynomially ambiguous automata can be characterized by the absence of certain characteristic features of the underlying graph \cite[Theorems 3.1 and 4.1]{weber-seidl91}.
As a consequence, every WFA that is not polynomially ambiguous is \defit{exponentially ambiguous}, that is, there are words $u$, $v$, $w$ such that $uw^n v$ has at least $2^n$ successful runs.

\begin{remark}
More specifically, for a WFA $\cA$, let $g(n)$ be the maximal number of successful runs for a word of length $n$.
Then trivially $g(n) \in O(d^n)$ where $d$ is the number of states.
The automaton $\cA$ is finitely ambiguous if $g(n) \in O(1)$ and polynomially ambiguous if $g(n) \in O(n^k)$ for some $k \ge 1$.
The characterizations in \cite{weber-seidl91} imply that there are \emph{gaps} in the possible growths of the function $g(n)$: we always have $g(n) \in O(d^n)$, and if $g(n) \in o(c^n)$ for all $c \in \R_{>1}$, then already $g(n) \in O(n^k)$ for some $k \ge 1$.
So any automaton with subexponential ambiguity is already polynomially ambiguous.
Similarly, if $g(n) \in o(n)$ then $g(n) \in O(1)$, so any automaton with sublinear ambiguity is already finitely ambiguous.
From this point of view, the classes considered in \cref{d:wfa-ambiguity} give a natural hierarchy.
\end{remark}

A trim WFA that is polynomially ambiguous is also \defit{cycle-unambiguous}, that is, for every $p \in Q$ and every word $w$, there exists at most one walk from $p$ to itself that is labeled by $w$.
This is easy to see: otherwise, there would be some $q \in [d]$ with two distinct directed cycles, say of lengths $k_1$ and $k_2$ based off it; then $w^{\lcm(k_1,k_2) n}$ has at least $2^n$ closed directed walks based at $q$; using trimness, this contradicts polynomial ambiguity.
Conversely, cycle-unambiguous WFA are polynomially ambiguous \cite[Theorem 4.1]{weber-seidl91}, but this is harder to prove, and we will not need it.

We are now ready to prove the easier direction of the arithmetic characterization of ambiguity classes.
Here we do not need to assume invertibility of the WFA.

\begin{proposition} \label{p:wfa-necessary}
  Let $\cA^\circ=(u^\circ,\mu^\circ,v^\circ)$ be a trim linear representation and let $\cA=(u,\mu,v)$ be a minimal linear representation with $\behav{\cA}=\behav{\cA^\circ}$.
  \begin{enumerate}
  \item \label{wfa-necc:m-ambig} If $\cA^\circ$ is $M$-ambiguous, then there exists a finitely generated $\Gamma \le K^\times$ such that $\behav{\cA}(w) \in \sumset{M}{\Gamma_0}$ for all $w \in X^*$.
  \item \label{wfa-necc:bezivin} If $\llbracket \cA^\circ\rrbracket(X^*)$ is Bézivin, then so is $\mu(X^*)$. In particular, if $\cA^\circ$ is finitely ambiguous, then $\mu(X^*)$ is Bézivin.
  \item \label{wfa-necc:spec} If $\cA^\circ$ is polynomially ambiguous, then $\mu(X^*)$ has finitely generated spectrum\footnote{The \emph{nonzero} eigenvalues are contained in a finitely generated subgroup of $\algc{K}^\times$.} and $K$ is uniformly power-splitting for $\mu(X^*)$.
  \end{enumerate}
\end{proposition}

\begin{proof}
  \begin{proofenumerate}
    \item[\ref{wfa-necc:m-ambig}] 
    Let $\Gamma$ be the group generated by all weights of $\cA$, that is, by the entries of $u$, $v$ and $\mu(x)$ for $x \in X$.
    The product of weights along a walk in $\cA$ gives an element of $\Gamma$, and since every word has at most $M$ successful runs, the claim follows.
    
    \item[\ref{wfa-necc:bezivin}]
    Let $M \ge 0$ and $\Gamma \le K^\times$ be finitely generated such that $u\mu(w)v \in \sumset{M}{\Gamma_0}$ for all $w \in X^*$.
    Because $\cA$ is minimal, there exists a basis of $K^{1 \times d}$ of the form $u \mu(w_1)$, \dots,~$u \mu(w_d)$ with $w_i \in X^*$, and a basis of $K^{d \times 1}$ of the form $\mu(w_1') v$, \dots,~$\mu(w_d')v$ with $w_i' \in X^*$ \cite[Proposition 2.2.1]{berstel-reutenauer11}.
    Let $e_1$, \dots,~$e_d$ denote the standard basis of $K^{d \times 1}$ and $e_1^*$, \dots,~$e_d^*$ its dual basis, so that the $ij$-entry of a matrix $A \in K^{d \times d}$ is just $e_i^* A e_j$. 
    We express $e_i = \sum_{k=1}^d t_{ik} \mu(w_k') v$ and $e_j^* = \sum_{l=1}^d s_{lj} u \mu(w_l)$ with $t_{ik}$,~$s_{lj} \in K$.
    For every word $w \in X^*$, then 
    \[
    e_i^* \mu(w) e_j = \sum_{k,l=1}^d t_{ik} s_{lj} u\mu(w_l)\mu(w)\mu(w_k')v = \sum_{k,l=1}^d t_{ik} s_{lj} u \mu(w_l w w_k')v.
    \]
    Since $u\mu(w_l w w_k)v \in \sumset{M}{\Gamma_0}$, therefore all entries of $\mu(w)$ are contained in $\sumset{d^2M}{\Gamma_0'}$ with $\Gamma'$ generated by $\Gamma$, $t_{ik}$, and $s_{lj}$.

    Finally, if $\cA^\circ$ is finitely ambiguous, then by \ref{wfa-necc:m-ambig} we get $M \ge 0$ and a finitely generated $\Gamma \le K^\times$ such that $u\mu(w)v \in \sumset{M}{\Gamma_0}$ for all $w \in X^*$.
    Hence, the semigroup $\mu(X^*)$ is Bézivin by what we just showed.

    \item[\ref{wfa-necc:spec}]
    Since $\cA^\circ$ is trim, it is cycle-unambiguous.
    Let $d$ be the dimension of $\cA^\circ$.
    After conjugation by a suitable invertible matrix, we can assume that $\mu^\circ(X^*)$ has block-triangular shape with one of the diagonal blocks being $\mu(X^*)$ \cite[Corollary 2.2.5]{berstel-reutenauer11}.
    It is therefore sufficient to establish that $\mu^\circ(X^*)$ has finitely generated spectrum and that $K$ is uniformly power-splitting for $\mu^\circ(X^*)$.
    Let $\Gamma' \le K^\times$ be the subgroup generated by the weights of $\cA^\circ$.
    
    Fix a word $w$ and consider the following directed graph (with loops) $G(w)$:
    the vertex set is $[d]$ and there is an edge $i \to j$ if there is a directed walk in $\cA^\circ$ from $i$ to $j$ that is labeled by $w$. 
    Since $\cA^\circ$ is cycle-unambiguous, the non-trivial strongly connected components of $G(w)$ must be directed cycles\footnote{A directed cycle is a closed directed walk in which all vertices except the first and the last are pairwise distinct.}, each of length at most $d$.
    It follows that the only non-trivial strongly connected components in $G(w^{d!})$ are loops.
    Therefore, the states $[d]$ can be totally ordered in such a way that $i < j$ implies that there is no path from $j$ to $i$ in $G(w^{d!})$. 

    After a corresponding permutation of the standard basis vectors, the matrix $\mu^\circ(w^{d!})=\mu^\circ(w)^{d!}$ is upper-triangular.
    Therefore, its diagonal entries are precisely its eigenvalues.
    The $p$-th diagonal entry of $\mu^\circ(w^{d!})$ is obtained as follows: for each directed walk from $p$ to $p$ in $\cA^\circ$, one first takes the product over all weights along the walk, and then sums these products over all directed walks from $p$ to $p$. 
    However, since $\cA^\circ$ is cycle-unambiguous, there exists at most one directed walk from $p$ to $p$.
    Thus, the nonzero diagonal entries of $\mu^\circ(w^{d!})$ are products of weights of $\cA^\circ$, and hence are contained in $\Gamma'$.
    
    Consequently, if $\lambda \in \algc{K}$ is an eigenvalue of $\mu(w)$, then $\lambda^{d!} \in \Gamma'$.
    Since $\Gamma = \{\, \gamma \in \algc{K} : \gamma^{d!} \in \Gamma' \,\}$ is finitely generated by \cref{l:bounded-roots}, the claim follows. \qedhere
  \end{proofenumerate}
\end{proof}

\begin{proposition} \label{p:nice-wfa}
  Let $\cA=(u,\mu,v)$ be an invertible linear representation.
  \begin{enumerate}
  \item \label{nice-wfa:bez} If $\mu(X^*)$ is Bézivin, then there exists a linear representation $\widehat \cA=(\widehat u, \widehat \mu, \widehat v)$ with $\behav{\widehat \cA}=\behav{\cA}$ and $\widehat \mu$ monomial. In particular, the automaton $\widehat \cA$ is finitely ambiguous.
  \item \label{nice-wfa:spec}  If $\mu(X^*)$ has finitely generated spectrum and $K$ is power-splitting for $\mu(X^*)$, then there exists a linear representation $\widehat \cA$ with $\behav{\widehat \cA}=\behav{\cA}$ and $\widehat \cA$ is block-triangular with monomial diagonal blocks.
  In particular, the automaton $\widehat \cA$ is polynomially ambiguous.
  \end{enumerate}
\end{proposition}

\begin{proof}
  Let $d$ be the dimension of $\cA$.
  We first make a general observation that will be useful for both claims.
  Suppose $\widehat \mu \colon X^* \to K^{n \times n}$ is a semigroup homomorphism, and $T\in K^{d \times n}$ has rank $d$, and is such that $T \widehat \mu(w) = \mu(w) T$ for all $w \in X^*$.
  Since $T \colon K^{n \times 1} \to K^{d \times 1}$ is surjective, there exists $\widehat v \in K^{n \times 1}$ such that $v = T\widehat v$.
  Let $\widehat u = u T$.
  Then $\widehat \cA=(\widehat u, \widehat \mu, \widehat v)$ is equivalent to $\cA$, because
  \[
  \widehat u \widehat \mu(w) \widehat v = u T \widehat \mu(w) \widehat v = u \mu(w) T \widehat v = u \mu(w) v \qquad\text{for all $w \in X^*$.}
  \]

  \begin{proofenumerate}
    \item[\ref{nice-wfa:bez}]
    By \cref{cor:locbezivin-fg} there exists a monomial semigroup representation $\widehat\mu\colon X^* \to \GL_n(K)$ and an $X^*$-invariant epimorphism $\widehat \mu \to \mu$.
    Since $\widehat\mu(X^*)$ is monomial, the resulting linear representation $(\widehat u, \widehat \mu, \widehat v)$ is finitely ambiguous.

    \item[\ref{nice-wfa:spec}] 
    By \cref{cor:trilocmonomial-lift} we can now find a linear representation $(\widehat u, \widehat \mu, \widehat v)$ with $\widehat \mu$ block-triangular with monomial diagonal blocks.
    Such a representation is polynomially ambiguous: corresponding to the block-triangular structure, we obtain a partition of the states $Q_1 \cup \dots \cup Q_r$ for which it is only possible to transition from a state in $Q_i$ to one in $Q_j$ if $j > i$.
    Furthermore, due to the monomial diagonals, transitions from states in $Q_i$ to states in the same set $Q_i$ are unambiguous. \qedhere
  \end{proofenumerate}
\end{proof}

The following shows that the minimal ambiguity degree achievable for a WFA computing a particular series is at the same time the minimal $M$ such that all outputs are contained in $\sumset{M}{\Gamma_0}$ for some finitely generated $\Gamma \le K^\times$.
It is obtained by taking a slightly closer look at the construction resulting from \cref{p:nice-wfa} (and choosing a suitable basis).

\begin{corollary} \label{cor:ambiguity-degree}
  Let $\cA=(u,\mu,v)$ be an invertible linear representation.
  If there exists $M \ge 0$ and a finitely generated $\Gamma \le K^\times$ with $\behav{\cA}(w) \in \sumset{M}{\Gamma_0}$ for all $w \in X^*$, then $\cA$ is equivalent to an $M$-ambiguous WFA.
\end{corollary}

\begin{proof}
  We may assume that $\cA=(u,\mu,v)$ is a minimal linear representation (using \cite[Corollary 2.2.5]{berstel-reutenauer11}).
  By \ref{wfa-necc:bezivin} of \cref{p:wfa-necessary}, then $\mu(X^*) \subseteq \GL_d(K)$ is Bézivin.
  Let $G \coloneqq \langle \mu(X^*) \rangle$.
  By \cref{t:locally-bezivin-equivalent} the finitely generated group $G \le \GL_d(K)$ is epimonomial.
  Let $K^{d \times 1} = V_1 \oplus \cdots \oplus V_r$ be an epimonomial decomposition for $G$, with $D \le G$ the corresponding diagonal.
  Decompose $v=v_1 + \cdots + v_r$ with $v_j \in V_j$, and for each $V_j$ choose a basis $e_{j1}$, \dots,~$e_{js_j}$, with $e_{j1}=v_j$ if $v_j \ne 0$.
  After changing to this basis, we can assume that $D$ consists of diagonal matrices and that each $v_j$ has at most one nonzero entry.

  Let $A_1=I$, $A_2$, \dots,~$A_n$ be a set of representatives for $G/D$ and consider the $dn$-dimensional representation $\rho \colon G \to \GL(W)$ induced from $D \hookrightarrow \GL_d(K)$, with
  \[
    W = \bigoplus_{i=1}^n A_i K^{d \times 1} = \bigoplus_{i=1}^n \bigoplus_{j=1}^r A_i V_j
  \]
  (as in \cref{l:factor-through-monomial}), and set $W_{ij}\coloneqq A_iV_j$.
  We choose as basis of $W_{ij}$ the vectors $f_{ijk} \coloneqq A_i e_{jk}$ (where $k \in [s_j]$), and identify $W = K^{dn \times 1}$.
  Then $\rho$ is monomial; explicitly, let $A \in G$ and let  $i \in [n]$, $j \in [r]$, $k \in [s_j]$.
  If $i' \in [n]$ and $C \in D$ are such that $AA_i=A_{i'}C$ and $C|_{V_j}=\lambda_j \id$ with $\lambda_j \in K$, then we have
  \[
  \rho(A) f_{ijk} = \lambda_j f_{i'jk}.
  \]

  The $G$-equivariant epimorphism
  \[
  \psi \colon W \to K^{d \times 1}, \ f_{ijk} \mapsto A_i e_{jk},
  \]
  is split by
  \[
  \iota \colon K^{d \times 1} \to W, \  e_{jk} \mapsto f_{1jk}.
  \]
  Let $\widehat v = \iota(v)$ and let $\widehat u \in K^{1 \times dn}$ such that $\widehat u f_{ijk} = uA_i e_{jk}$. 
  With $\widehat \mu = \rho \mu$, we obtain a monomial WFA $\widehat \cA=(\widehat u,\widehat \mu, \widehat v)$ that is equivalent to $\cA$ (as in \cref{p:nice-wfa}).

  Using \cref{p:maximally-separating-element}, choose $A \in D$ steady and such that the eigenvalues on all $V_j$ are pairwise distinct.
  Consider $\rho(A)$.
  Each $A_i$ permutes the eigenspaces of $A$, say $A_i V_j = V_{\pi_i(j)}$ with $\pi_i \in \mathfrak S_r$ a permutation.
  Then $A_i^{-1} A A_i$ is again diagonal with the same eigenspaces as $A$ itself, but the eigenvalues permuted: if $A$ has eigenvalue $\lambda_j$ on $V_j$, then $A_i^{-1} A A_i$ has the eigenvalue $\lambda_{\pi_i(j)}$ on $V_j$.
  Setting, for $j \in [r]$,
  \[
    W_j \coloneqq \bigoplus_{i=1}^n A_i V_{\pi_i^{-1}(j)} \subseteq W,
  \]
  we see that the $W_j$ are precisely the eigenspaces of the steady element $\rho(A)$.
  Since $\rho(G)$ is monomial, it is certainly Bézivin, and we can apply \cref{p:full-term-bound} to the block structure $W= W_1 \oplus \cdots \oplus W_r$.
  Thus, for all $x_1$, \dots,~$x_l \in X$ (using square brackets to denote the blocks), we have that
  \begin{equation} \label{eq:ambig-degree-block}
  \widehat u [\mu(x_1)]_{i_0i_1} [\mu(x_2)]_{i_1i_2} \cdots [\mu(x_l)]_{i_{l-1}i_l} [\widehat v]_{i_l} \ne 0,
  \end{equation}
  for at most $M$ choices of $(i_0,i_1,\dots,i_l) \in [r]^{l+1}$.
  We shall argue that each such nonzero block product corresponds to at most one successful run of $\widehat \cA$.

  For $x \in X$ let $\widehat\mu_{pq}(x)$ denote the $pq$ entry of the matrix $\widehat\mu(x)$.
  A successful run for a word $x_1\cdots x_l \in X^*$ is simply a tuple $(p_0,p_1,\dots,p_l) \in [dn]$ such that
  \[
  \widehat u_{p_0} \widehat \mu_{p_0p_1}(x_1) \widehat \mu_{p_1p_2}(x_2) \cdots \widehat \mu_{p_{l-1}p_l}(x_l) \widehat v_{p_l} \ne 0.
  \]
  
  This relates to the block structure in \cref{eq:ambig-degree-block} as follows: 
  Let $V \coloneqq V_1 \oplus \cdots \oplus V_r=K^{d \times 1}$. The terminal vector $\widehat v$ is contained in the summand $IV$ of $W$, and projected onto each $I V_j$, it has at most one nonzero component by the choice of the basis (namely, the component is a multiple of $f_{1j1}$).
  By construction of the induced representation, the matrices $\mu(x)$ permute the spaces $A_1 V$, \dots,~$A_n V \subseteq W$.
  Thus, for each $k \in [1,l]$ there is a unique $i$ with $\mu(x_k\cdots x_l) \widehat v \in A_i V$.
  Fix $k$.
  Then, for $j \in [r]$, there is a unique $t \in [r]$ with $\mu(x_k\cdots x_l) f_{1j1} \in A_i V_{t}$.
  But then in fact $j$ uniquely determines the space $W_{t'}$ for which $\mu(x_k\cdots x_l) f_{1j1} \in W_{t'}$
  (namely $t'=\pi_i(t)$).
  It follows that each successful run (each one necessary starting at a distinct $f_{1j1}$, by monomiality) contributes to a distinct block product of the form \cref{eq:ambig-degree-block}.
  Hence, the maximal number of nonzero such block products, namely $M$, bounds the maximal number of successful runs for a given word.
  Thus, the linear representation $\widehat \cA$ is at most $M$-ambiguous.
\end{proof}

\section{Decidability} \label{sec:decidability}

For finitely generated semigroups we now show decidability of the properties under discussion (and that the monomial, respectively block-triangular representations with monomial blocks are indeed computable, if they exist).
Obviously this requires some restrictions on the field $K$, in particular, we need to be able to perform exact arithmetic and equality tests, so that we can do exact linear algebra, use Gröbner bases techniques, and compute in finite field extensions of $K$ (as splitting fields of suitable polynomials).
We shall call such a field \emph{computable}.

Most importantly, number fields (finite-dimensional extensions of $\Q$) are computable.
In general, a finitely generated semigroup $S$ will always be contained in $K^{d \times d}$ with $K$ a field that is finitely generated over its prime field $K_0$.
Then $K$ is the field of fractions of an affine $K_0$-algebra $R$.
If $K$ can be specified by giving explicit generators and relations of $R$ as a $K_0$-algebra, then $K$ is computable.
For example, for the transcendental extension $K=\Q(X)$ with $X$ a nonempty finite set we trivially know such a representation, but for $K=\Q(e,\pi)$ we do not have such a representation, since we do not know whether there are algebraic relations between $e$ and $\pi$.
\emph{For the remainder of the section, fix a computable field $K$ and a vector space $V$.}

The \defit{linear Zariski topology} on $V$ is the topology having as closed sets the finite unions of vector subspaces (including the empty set as the empty union) \cite{bell-smertnig21,bell-smertnig23b}.
It is a noetherian topology.
If $K$ is infinite, the irreducible sets are precisely the subspaces of $V$.
If $K$ is finite, then the irreducible sets are the subspaces of dimension at most $1$.

The Zariski closure and the linear Zariski closure of a finitely generated group \cite{derksen-jeandel-koiran05,nosan-pouly-schmitz-shirmohammadi-worrell22}, and even a finitely generated semigroup \cite{hrushovski-ouaknine-pouly-worrell18,hrushovski-ouaknine-pouly-worrell23,bell-smertnig23b}, of matrices is computable.

The following observation will reduce everything to the computation of the linear Zariski closure.
The argument is similar to the one in the proof of \cite[Theorem 10]{bell-smertnig23b}.

\begin{proposition} \label{p:closure-of-diagonal}
  Let $G \subseteq \GL(V)$ be a locally Bézivin group, let $V = V_1 \oplus \cdots \oplus V_r$ be the decomposition into joint-eigenspaces of steady elements, and let $D \trianglelefteq G$ be the corresponding diagonal.
  Let $P_i \in \End(V)$ be such that $P_i|_{V_i}=\id$ and $P_i|_{V_j}=0$ for $j \ne i$.
  Then the linear Zariski closure $\overline D$ of $D$ in $\End(V)$ is
  \[
  \overline{D} = \lspan\{ P_1, \dots, P_r \}.
  \]
  In particular, if $G/D$ is finite with set of representatives $A_1$, \dots,~$A_n$, then the irreducible components of $\overline G$ are
  \[
  A_1 \lspan\{ P_1, \dots, P_r \}, \dots, A_n \lspan\{ P_1, \dots, P_r \}.
  \]
\end{proposition}

\begin{proof}
  Identify $\GL(V)=\GL_d(K)$ using a basis that refines the decomposition $V=V_1\oplus\cdots \oplus V_r$.
  Let $\overline{D} = Z_0 \cup \cdots \cup Z_l$ with $Z_k \subseteq V$ the irreducible components of $\overline{D}$.
  Because $D \subseteq \lspan\{ P_1, \dots, P_r \}$, also $Z_0 \cup \cdots \cup Z_l \subseteq \lspan\{ P_1, \dots, P_r \}$.
  Note that $\overline{D} \cap \GL_d(K)$ is the set of $K$-rational points of a linear algebraic group.
  There exists a unique component, say $Z_0$, containing the identity, and $Z_0 \cap \GL_d(K)$ is a subgroup of $G$ \cite[Lemma 9]{bell-smertnig23b}.
  
  By \cref{p:maximally-separating-element} there exists a steady $A \in G$ that has distinct eigenvalues on the $V_i$, say $A= \lambda_1 P_1 + \cdots + \lambda_r P_r$ with $\lambda_i \in K$ such that $\lambda_i/\lambda_j$ is not a root of unity for $i \ne j$.
  Since multiplication by $A$ permutes the components $Z_0$, \dots,~$Z_l$, there exists an $N \ge 1$ such that $A^N \in Z_0$.
  Then $A^{Nn} = \lambda_1^{Nn} P_1 + \cdots + \lambda_r^{Nn} P_r \in Z_0$ for all $n \in \Z$, and, since $Z_0$ is a vector space, we have $W \coloneqq \lspan\{ A^{Nn} :n \in \Z  \} \subseteq Z_0$.
  Keeping in mind that $\lambda_i^N \ne \lambda_j^N$ for all $i \ne j \in [r]$, a Vandermonde determinant argument shows $\dim W = r$.
  Because of $W \subseteq Z_0 \subseteq \lspan\{P_1,\dots,P_r\}$ we must have equality throughout, hence $\overline{D}=\lspan \{ P_1,\dots, P_r\}$.
\end{proof}

\begin{proposition} \label{p:decide-semigroup}
For a finitely generated semigroup $S \subseteq \GL_d(K)$, given by generators $B_1$, \dots,~$B_m \in S$, it is decidable whether $S$
\begin{enumerate}
\item \label{decidesgrp:bezivin} is Bézivin;
\item \label{decidesgrp:fgspec} has finitely generated spectrum and $K$ is power-splitting for $S$.
\end{enumerate}
One can moreover compute a representation $\rho\colon S \to \GL_{d'}(K)$ and an epimorphism of representations $\rho \twoheadrightarrow (S \hookrightarrow \GL_d(K))$ such that $\rho$ is monomial in the first case, and such that $\rho$ is block-triangular with monomial diagonal blocks in the second case.
\end{proposition}

\begin{proof}
  The linear Zariski closure $\overline{S}$ is computable; more explicitly, one can compute a basis of the irreducible component $Z_0$ containing the identity, and $A_0=I$, $A_1$, \dots,~$A_n \in S$ such that $\overline{S} = \bigcup_{i=0}^n A_i Z_0$, with $A_i Z_0 \ne A_j Z_0$ for $j \ne i$.
  The closure $\overline{S}$ contains the identity matrix and $A^{-1}$ for every $A \in S$, so that in fact $\overline{S}$ is also the closure of the group $G$ generated by $S$ \cite[Lemma 9]{bell-smertnig23b}.

  \begin{proofenumerate}
    \item[\ref{decidesgrp:bezivin}]
    In light of \cref{p:closure-of-diagonal} and \cref{t:locally-bezivin-equivalent}, it is enough to check whether $Z_0$ is simultaneously diagonalizable over $K$.
    This is easily done iteratively: suppose the component $Z_0$ is not simultaneously diagonal in the given basis. 
    We can then compute a non-diagonal $A \in Z_0 \cap \GL_d(K)$ using \cite[Lemma 17]{bell-smertnig23b} and the fact that $Z_0 \cap \GL_d(K) \supseteq Z_0 \cap S$ is dense in $Z_0$ in the linear Zariski topology.
    If $A$ is not diagonalizable over $K$, then $S$ is not Bézivin and we abort.
    Otherwise, we change the basis to one in which $A$ is diagonal.
    We now check if $Z_0$ has block-diagonal structure with respect to the eigenspaces of $A$ (it is sufficient to check this on the basis vectors): if not, then $S$ is not Bézivin, and we abort. If $Z_0$ has the required block-diagonal structure, but is not diagonal, then we can pick a non-diagonal $A' \in Z_0 \cap \GL_d(K)$ (again using \cite[Lemma 17]{bell-smertnig23b}), and we attempt to diagonalize it.
    This will either result in a finer-grained block-diagonal structure, or show that $S$ is not Bézivin.
    Repeating this finitely many times, we will either have diagonalized all of $Z_0$ (showing that $S$ is Bézivin), or discovered that $S$ is not Bézivin.

    \item[\ref{decidesgrp:fgspec}]
    We first compute a finite field extension $L/K$ and a chain of $G$-invariant $L$-subspaces $0 = W_1 \subsetneq W_2 \subsetneq \cdots \subsetneq W_k=L^{d \times 1}$ such that each $W_{i}/W_{i-1}$ is absolutely irreducible as a $G$-representation (that is, there exists no $G$-invariant $\algc{K}$-subspace $\algc{K} \otimes_K W_{i-1} \subsetneq W' \subsetneq \algc{K} \otimes_K W_{i}$).
    This is possible, because the $\algc{K}$-invariant subspaces of a $K$-vector space $V$ can be described as the $\algc{K}$-points on a projective variety that is defined over $K$, by using exterior powers and Plücker coordinates, see \cite{arapura-peterson04}.
    Testing whether this variety is empty can be done using Gröbner basis techniques, and if it is not empty, one can compute a finite field extension $L/K$ and a point defined over $L$,%
    \footnote{It suffices to treat the affine case. One possibility is outlined in \cite[Section 3.3]{hrushovski-ouaknine-pouly-worrell23}, based on the fact that one can compute projections of \emph{constructible} sets \cite[Corollary 4.12]{schauenburg07}  \cite[Chapter 4.7]{cox-little-osheah15}.} %
    giving rise to an invariant subspace. 
    Iterating this procedure finitely many times, we find the required decomposition.

    Viewing $L^{d \times 1}$ as a $K$-vector space, we obtain a $K$-representation $\rho$ of $G$.
    If $G$ has locally finitely generated spectrum and $K$ is power-splitting for $G$, then the resulting diagonal blocks of $\rho$ are epimonomial (using \cref{p:absirred-fg-bezivin} and \cref{t:locally-bezivin-equivalent}), and we can check this using \ref{decidesgrp:bezivin}.
    So, if one of the diagonal blocks is not epimonomial, then $G$ did not have locally finitely generated spectrum or $K$ is not power-splitting.
    If they are all epimonomial, then $S \to \GL_d(K)$ factors through a representation with monomial diagonal blocks by \cref{cor:trilocmonomial-lift}.
  \end{proofenumerate}

  Computing the claimed representations is now a matter of computing suitable induced representations.
  To compute the induced representation, we need a set of representatives for $G / (G \cap Z_0)$.
  However, such a set is given by the elements $A_0$, \dots, ~$A_n$ computed at the beginning.
\end{proof}

\begin{corollary} \label{cor:decide-wfa}
  Let $\cA$ be a linear representation of a WFA.
  \begin{enumerate}
  \item \label{decidewfa:class}  It is decidable whether $\cA$ is equivalent to a polynomially ambiguous, respectively, a finitely ambiguous automaton.
  \item \label{decidewfa:M} If $\cA$ is equivalent to a finitely ambiguous automaton, it is also decidable which is the minimal $M \ge 0$ such that $\cA$ is $M$-ambiguous
  \end{enumerate}
\end{corollary}

\begin{proof}
  \begin{proofenumerate}
  \item[\ref{decidewfa:class}] 
  From a linear representation one can compute a minimal linear representation \cite[Chapter 2.3]{berstel-reutenauer11}.
  So without restriction assume that $\cA=(u,\mu,v)$ is minimal.
  By \cref{p:nice-wfa,p:wfa-necessary} it is now sufficient to decide whether $\mu(X^*)$ is Bézivin, respectively, whether $\mu(X^*)$ has finitely generated spectrum and $K$ is power-splitting for $\mu(X^*)$.
  These properties are decidable by \cref{p:decide-semigroup}.

  \item[\ref{decidewfa:M}]
  The construction described in \cref{cor:ambiguity-degree} produces an automaton $(u,\mu,v)$ of minimal ambiguity degree.
  Hence, it is a matter of computing its ambiguity.
  
  The matrices in $\mu(X^*)$ are all monomial matrices, so that any application of a matrix $\mu(x)$ permutes the entries in the vector and then possibly multiplies them by (different) weights.
  Thus, for every word $w \in X^*$ the vector $\mu(w) v$ will have the same number of nonzero entries as the vector $v$, but possibly in different places.

  The number of successful runs for $w$ is then simply the number of $i \in [d]$ for which the $i$-th entry of both $u$ and $\mu(w) v$ is nonzero.
  Deciding the degree of ambiguity therefore boils down to finding the possibly patterns of nonzero entries in $\mu(w) v$ as $w$ ranges through $X^*$.
  This is easily done, because mapping each $\mu(w)$ to the corresponding permutation matrix $\overline{\mu(w)}$ (replacing every nonzero entry by $1$) is a group homomorphism to a finite group, and so we can compute all the permutations $\{\, \overline{\mu(w)} : w \in X^* \,\}$. \qedhere
\end{proofenumerate}
\end{proof}

\begin{remark}
The algorithm also produces a polynomially ambiguous, respectively, finitely ambiguous automaton.
Since \cref{cor:ambiguity-degree} is constructive, we can also find a finitely ambiguous automaton of minimal ambiguity degree $M \ge 0$ (if there exists one).
Because of the monomial structure, the only ambiguity arises from the multiple initial and terminal states. 
The constructed automaton can therefore be easily transformed into a sum of $M$ deterministic automata.
(It is known that every finitely ambiguous automaton is a sum of unambiguous automata.)
\end{remark}

\begin{remark}[Nonefficiency] \label{rem:nonefficiency}
  The algorithms presented here are conceptually simple, but not efficient.
  For computing the linear Zariski closure of a semigroup, the only known upper bounds are \emph{double-exponential} in the dimension, as recently shown by Benalioua, Lhote, and Reynier \cite{benalioua-lhote-reynier24}.
  The linear Zariski closure can have \emph{super-exponentially} many irreducible components \cite[Remark 7]{bell-smertnig23b}, and so we cannot expect efficient decidability algorithms that first compute the entire (linear) Zariski closure.
  
  \begin{enumerate}
  \item In case we are willing to decide existence of a block-triangular representation with monomial diagonal blocks over a finite field extension (so that the power-splitting condition plays no role), checking virtual solvability, that is, deciding the Tits' alternative, is sufficient by \cref{c:intro-spec}.
  
  This algorithmic problem is well-studied and efficient and practical algorithms exist.
  For instance, Beals \cite{beals99,beals01} proved that, over a number field, virtual solvability is decidable in \emph{polynomial time}.
  Further algorithms can be found in the works of Detinko, Flannery, and O'Brien \cite{detinko-flannery-obrien11,detinko-flannery-obrien11}.
  For subgroups of $\GL_d(\Z)$, virtual solvability is equivalent to virtual polycyclicity, which can be decided by work of Baumslag, Cannonito, Robinson, and Segal \cite{baumslag-cannonito-robinson-segal91} and Ostheimer \cite{ostheimer99}.

  \item Very recently, Jecker, Mazowiecki, and Purser \cite{jecker-mazowiecki-purser24} showed (over $K=\Q$): if one is given a polynomially ambiguous WFA, then it is possible to decide whether it is equivalent to a deterministic or an unambiguous one in \emph{polynomial space} (\emph{PSPACE}).
  \end{enumerate}
\end{remark}

\section{Conclusion and Open Questions} \label{sec:conclusion}

In the present paper we comprehensively resolve the problem of characterizing, by arithmetic properties, the representability of a linear \emph{group} by monomial respectively block-triangular matrices with monomial diagonal blocks over any field.
For WFA \emph{with invertible transition matrices} we used this to relate ambiguity classes of rational series to arithmetic properties, and resolve the arising decidability questions.
In particular, it is possibly to decide whether an invertible WFA is equivalent to a finitely ambiguous, respectively, a polynomial ambiguous one.

\begin{enumerate}
  \item Given the applications to WFA, the most pressing issue is whether analogous results to the main theorem can be established for finitely generated semigroups $S \subseteq \End(V)$ of possibly non-invertible matrices, that is, to find sufficiently nice representations of Bézivin semigroups and semigroups with finitely generated spectrum.
  
  Not every such semigroup can have a monomial representation.
  For instance, if $\lambda \in \Q\setminus \{\pm 1\}$ the semigroup generated by
  \[
  \begin{pmatrix}
  \lambda & 0 \\
  0 & 1
  \end{pmatrix},
  \begin{pmatrix}
  0 & 1 \\
  0 & 1
  \end{pmatrix}
  \]
  has infinitely many idempotents $\begin{psmallmatrix} 0 & \lambda^n \\ 0 & 1 \end{psmallmatrix}$, which is impossible in an epimorphic image of a monomial semigroup.
  However, there are some positive results. 

  \begin{enumerate}
  \item In the case that all outputs of the WFA are contained in $\Gamma_0$ (Pólya property), it is known that the linear representation of the WFA factors through a \defit{semimonomial} representation \cite{bell-smertnig21}.
  Here the matrices have a block structure: each row of blocks contains at most one nonzero blocks, and each column of each block contains at most one nonzero entry.

  \item There are several proofs showing that finitely generated torsion matrix semigroups are finite, that is, Burnside--Schur extends to semigroups (see \cite{steinberg12} for many references).

  \item Suppose that $S$ contains an invertible steady $A$ that has $d$ distinct eigenvalues.
  Then applying \cref{c:term-bound} immediately shows that $S$ is finitely ambiguous, and so we get the following.
  \begin{corollary} \label{c:automaton-one-separating}
    Let $\cA=(u,\mu,v)$ be a $d$-dimensional linear representation of a WFA over the field $K$.
    If $\mu(\Sigma^*)$ is a Bézivin semigroup, and it contains at least one steady invertible matrix that has $d$ pairwise distinct eigenvalues, then a change of basis transforms $\cA$ into a finitely ambiguous linear representation.
  \end{corollary}
  \end{enumerate}

  The algebraic structure of finitely generated matrix semigroups and their Zariski closure is well-understood based on work of Putcha and Okniński \cite{putcha88,okninski98,renner05}.
  In particular, if one considers the (linear) Zariski closure, there is a finite chain of semigroup ideals, the successive Rees quotients of which are either completely $0$-simple or nilpotent.
  This structure has already played a key role in the computation of the (linear) Zariski closure \cite{hrushovski-ouaknine-pouly-worrell18,hrushovski-ouaknine-pouly-worrell23,bell-smertnig23b}.
  While the semigroup case makes use of the group case, in these instances, it turned out that the step from groups to arbitrary semigroups is the most complex one.
  
  \item From the point of view of theoretical computer science, it would be interesting to establish which complexity classes the decidability problems arising in \cref{sec:decidability} actually fall into.
  This is particularly interesting because the algorithms based on the (linear) Zariski closure are necessarily quite inefficient, while more efficient algorithms are already known for certain subproblems and the decidability of virtual solvability (\cref{rem:nonefficiency}).
  
  Further, there is a trade-off between ambiguity and the number of states.
  From this point of view it is interesting to what degree our constructions are optimal (in minimizing the number of states in the given complexity class) and what the minimum number of states for a WFA in a given complexity class, as a function of the number of states of a minimal WFA, is.

  Little appears to be known here. 
  For instance, it is well-known that the maximal subgroups of $\GL_d(\Q)$ have cardinality $2^d d!$ for $d$ sufficiently large (and can be realized by signed permutation matrices).
  For $n$-generated subsemigroups of $\Q^{d \times d}$, the only known upper bounds appear to be double-exponential in $d$.
\end{enumerate}

\begin{adjustwidth}{-1cm}{-1cm}
\begin{singlespace}
  \printbibliography
\end{singlespace}
\end{adjustwidth}

\end{document}